\documentclass[a4paper, 11pt, reqno]{amsart}
\usepackage[utf8]{inputenc}
\usepackage{subfiles}

\usepackage{latexsym,amsmath,amssymb,amsfonts}
\usepackage{dsfont}
\usepackage{color}
\usepackage{graphicx}
\usepackage{float}
\usepackage{esint}
\usepackage{bm}
\usepackage{mathtools}
\usepackage[top=1in, bottom=1.25in, left=1in, right=1in]{geometry}
\usepackage{enumitem}
\usepackage{hyperref}
\usepackage{cleveref}
\usepackage{bbm}
\usepackage{tikz}                        
\usepackage{soul}   
\usepackage{calrsfs}
\usepackage{parskip}
\usepackage{comment}
\usepackage{shuffle}
\usepackage{nicefrac}
\usepackage{enumitem}

\newcommand{\R}{{\mathbb{R}}}

\newcommand{\N}{{\mathbb{N}}}

\newcommand{\wkst}{\stackrel{*}{\rightharpoonup}}

\newcommand{\seq}[1]{\{#1_m\}_{m\in\N}}
\newcommand{\inprod}[2]{\left\langle #1, #2 \right\rangle}
\newcommand{\alpharough}{\mathcal{C}^\alpha([0,T];E)}
\newcommand{\alphageom}{\mathcal{C}^{\alpha}_g ([0,T];E)}

\newcommand{\BR}{\left\{ \mathbf{x} \in \mathcal{C}^{\alpha} ([0,T];E ) \mid \|\mathbf{x}\|_{\mathcal{C}^{\alpha}} \leq R\right\}}
\newcommand{\BRn}[1]{\left\{ \mathbf{x} \in \mathcal{C}^{\alpha,{#1}} ([0,T];E ) \mid \|\mathbf{ x}\|_{\mathcal{C}^{\alpha,{#1}}} \leq R\right\}}
\newcommand{\BRg}{\left\{ \mathbf{x} \in \mathcal{C}_g^{\alpha} ([0,T];E ) \mid \|\mathbf{x}\|_{\mathcal{C}^{\alpha}} \leq R\right\}}
\newcommand{\BRgn}[1]{\left\{ \mathbf{x} \in \mathcal{C}_g^{\alpha,#1} ([0,T];E ) \mid \|\mathbf{x}\|_{\mathcal{C}^{\alpha,#1}} \leq R\right\}}

\DeclareMathOperator{\id}{id}
\DeclareMathOperator{\supp}{supp}

\NewDocumentCommand{\varnorm}{sO{}m}{%
  \IfBooleanTF{#1}{
    \left|\opnormkern\left|\opnormkern\left|
    #3
    \right|\opnormkern\right|\opnormkern\right|
  }{
    \mathopen{#2|\opnormkern #2|\opnormkern #2|}
    #3
    \mathclose{#2|\opnormkern #2|\opnormkern #2|}
  }%
}
\newcommand{\opnormkern}{\mkern-1.5mu\relax}%

\newcommand*\D{\mathop{}\!\mathrm{d}}

\theoremstyle{plain}
\newtheorem{theorem}{Theorem}[section]
\newtheorem{lemma}[theorem]{Lemma}
\newtheorem{proposition}[theorem]{Proposition}
\newtheorem{propdef}[theorem]{Proposition/Definition}
\newtheorem{thmdef}[theorem]{Theorem/Definition}
\newtheorem{corollary}[theorem]{Corollary}

\theoremstyle{definition}
\newtheorem{definition}[theorem]{Definition}
\newtheorem{example}[theorem]{Example}
\newtheorem{setting}[theorem]{Setting}
\newtheorem{assumption}[theorem]{Assumption}

\theoremstyle{remark}
\newtheorem{remark}[theorem]{Remark}

\title{Universal approximation by signatures for infinite-dimensional rough paths}
\author{Sonja Cox, Asma Khedher, Thijs Maessen}

\begin{document}
\begin{abstract}
We establish universal approximation theorems for infinite-dimensional geometric rough paths, i.e., we show that continuous functions on the space of infinite-dimensional weakly geometric H\"older continuous rough paths can be approximated by functions that are linear in the signature of the path. The underlying topology determining continuity and compactness can be either the norm topology or the weak-$^*$ topology. Whereas considerably more effort is required to obtain the universal approximation theorem with respect to the weak-$^*$ topology, this setting ensures uniform approximation on norm-bounded sets. The motivation for establishing universal approximation theorems lies in the desire to approximate quantities derived from the solution of a stochastic partial differential equation. More specifically, our universal approximation theorems form the foundations of a novel approach to e.g.\ pricing of forward rates within the Heath--Jarrow--Morton--Musiela framework.
\end{abstract}
\maketitle

\section{Introduction}

The goal of this work is to establish {\it a universal approximation theorem} (UAT) for {\it infinite-dimensional  weakly geometric H\"older continuous rough paths}. More specifically, we show that continuous functions of such paths can be approximated uniformly on compact sets by functions that are linear in the \emph{signature} of the path.  
The signature of a path $x\colon [0,T]\rightarrow E$ is an infinite sequence of iterated integrals over the interval \([0,T]\) (see, e.g., \cite{lyons_differential_2007}). 
The theory of rough paths provides a rigorous framework for defining these iterated integrals unambiguously for paths with low regularity, e.g.\ solutions to stochastic differential equations (see, e.g.,~\cite{friz_multidimensional_2010, friz_course_2020}). Moreover, the notion of \emph{weakly geometric} rough paths ensures that the product of two linear functionals on the signature is again a linear functional on the signature, rendering the set of linear functionals an algebra -- a key element in the Stone--Weierstrass theorem which is the back-bone of our universal approximation theorems. 

Broadly speaking, UATs can be viewed as applications of the Stone-Weierstrass theorem tailored to a specific setting; they provide a way to approximate continuous functions uniformly on compacta using an algebra of simpler functions. Such results are of interest for several reasons: the simpler functions may be easier to manipulate mathematically and/or may be computationally more tractable. Indeed, one of the major applications of UATs lies in the field of machine learning, where UATs guarantee that feedforward neural networks can approximate continuous function uniformly on compact sets. Before discussing such applications in more detail, we briefly discuss the available literature regarding UATs for signatures of paths.

UATs for continuous functions of the signature of \emph{finite-dimensional} paths are established e.g.\ in~\cite[Proposition 4.5]{lyons_non-parametric_2020} and~\cite[Theorem 2.12]{cuchiero_signature-based_2023} for continuous semimartingales, and in~\cite{Cuchiero_Primavera_2025} for stochastic processes with c\`adl\`ag paths. A global UAT, i.e., going beyond the usual approximation on compact sets, is proven in \cite[Chapter 5]{cuchiero_global_2023} using weighted spaces of continuous functions of signatures. These results underpin powerful applications in machine learning; signatures have been used for feature extraction in time series \cite{gyurko2013extracting}, sparse regression methods \cite{graham2013sparse}, and data-driven prediction and learning of path-dependent functionals \cite{levin2013learning, Liao_2024, Bayraktar_2024}.

Another motivation for studying UATs in the rough path setting is their connection to stochastic differential equations (SDEs). Specifically, under suitable smoothness assumptions, the rough differential equation (RDE) driven by an (Itô/Stratonovich) Brownian rough path yields a strong solution of the corresponding (Itô/Stratonovich) SDE (see \cite[Theorem 9.1]{friz_course_2020} or \cite[Theorem 17.3]{friz_multidimensional_2010}). Furthermore, under similar smoothness conditions, the solutions of the RDE can be expressed as a continuous map on the underlying rough path (see \cite[Theorem 8.5]{friz_course_2020} or \cite[Theorem 10.26]{friz_multidimensional_2010}).
This implies that UATs can be used to uniformly approximate solutions of SDEs as functions of the underlying Brownian rough path. 
Building on this perspective, several applications in mathematical finance replace classical SDE models with signature models, in which the coefficients are taken as linear functionals of the signature of an underlying process, typically a semimartingale, referred to as the primer process. Such formulations are particularly well-suited for applying UATs for signatures, for example in mathematical finance for computing option prices, and have been explored in the finite-dimensional setting in a range of works, see, e.g., \cite{Bayer_2023, cuchiero_signature-based_2023, Kalsi_2020, Lyons_Nejad_2019, arribas2020sig, lyons_non-parametric_2020, AbiJaber_2025, Cuchiero_Gazzani_2025} and the references therein. A preliminary UAT in the infinite-dimensional setting is provided in~\cite[Theorems 21 and 26]{chevyrev_signature_2022}.

A final motivation for studying UATs for continuous functions of the signature is that such UATs provide conditions under which the expectation of the signature of a stochastic rough path fully characterizes its distribution, see e.g.~\cite{chevyrev_signature_2022}, and~\cite{chevyrev_characteristic_2016} for a similar result. Adapting this reasoning, our main results imply that the distribution of a stochastic rough path is fully characterized by the expectation of certain \emph{linear functionals} of the signature -- specifically, expectations of elements in the algebra~\eqref{eq: algebra} below. 

However, the key motivation for this work is the desire to extend the aforementioned approach regarding the approximation of SDEs to stochastic \emph{partial} differential equations (SPDEs). This requires us to consider paths that take values in an infinite-dimensional (Banach) space. This broader scope is motivated by applications in mathematical finance. For example, in the bond and commodity markets, the dynamics of future prices is modeled within the Heath--Jarrow--Morton (HJM) framework by means of an SPDE \cite{benth_representation_2014, Benth_kruhner_2015, cox_karbach-2022, cox_infinitedimensional_2022,Cox_cuchiero_2024, he2025pricing, Cuchiero_Dipersio_2025, benth2025measure}. Approximating such SPDEs by linear functionals on the signature of a rough Brownian path provides a powerful tool: it enables both efficient calibration to market data and the computation of option prices in this setting. We believe that our work provides the first full-fledged and thorough analysis of UATs for infinite-dimensional signatures.

Before presenting our main results let us briefly dwell on some key concepts. Signatures of paths taking values in an (infinite-dimensional) Banach space $E$ are studied in~\cite{lyons_differential_2007}. Recall that the $n$-times iterated integral of an $\R^d$-valued path takes values in $\R^{n\times d}$. However, when we consider a path taking values in a Banach space $E$, the $n$-times iterated integral of an $E$-valued path takes values in a \emph{topological tensor space} $\prescript{n}{j=1}{\bigotimes}_{\|\cdot \|_n}E =\colon E^{\otimes n}$, and there is some freedom regarding the choice of the tensor norm $\| \cdot \|_n$. Following~\cite[Definition 1.25]{lyons_differential_2007} we identify so-called \emph{admissible} families of tensor norms, see also Definition~\ref{def: admissible crossnorms} below. Examples of admissible families of tensor norms include:
\begin{enumerate}
    \item $E$ is a Hilbert space and the tensor norms are the Hilbert tensor norms,
    \item $E$ is a Banach space and the tensor norms are either the projective tensor norms or the injective tensor norms.
\end{enumerate} 

For $\alpha\in (0,1)$, $T>0$, and $E$ a real Banach space endowed with an admissible family of tensor norms, we let $\mathcal{C}^{\alpha}([0,T],E)$ denote the space of $\alpha$-H\"older continuous $E$-valued rough paths, see also Section~\ref{sec: holder norms and rough paths} below. Moreover, let 
\begin{equation} 
S(\cdot)_{[0,T]}\colon \mathcal{C}^{\alpha}([0,T],E) \rightarrow \bigoplus_{n=0}^\infty E^{\otimes n} 
\end{equation} 
be the \emph{signature map} (also known as the \emph{Lyons lift}); i.e., $S(\cdot)_{[0,T]}$ maps a H\"older continuous $E$-valued rough path $\mathbf{x}$ to $S(\mathbf{x})_{[0,T]}$, the sequence of $n$-times iterated integrals of $\mathbf{x}$ over $[0,T]$, $n\in \N$ (see also Section~\ref{sec: Lyons lift}). Finally, we let $(E^*)^{\otimes_a n}$ denote the algebraic $n$-tensor product of $E^*$, the dual of $E$ (see Appendix~\ref{app: tensor spaces}) and assume the convention that $(E^*)^{\otimes_a 0} =\R$.

Our starting point is the following theorem, which is essentially a direct consequence of the Stone-Weierstrass theorem (see also Theorem~\ref{thm: general setup of uat} below):
\begin{theorem}\label{thm: general setup of uat intro}
Let $E$ be a real Banach space, 
let $\tau$ be a Hausdorff topology on $\mathcal{C}^{\alpha}([0,T];E)$, let $K$ be a $\tau$-compact set, and let $D$ be a subspace of $\bigoplus_{n=0}^\infty (E^*)^{\otimes_a n}$ such that
\begin{enumerate}
    \item\label{it:continuous}  $K \ni\mathbf{x} \mapsto \inprod{S(\mathbf{x})_{0,T}}{l}$ is a $\tau$-continuous map for any $l\in D$,
    \item\label{it:point-sep2}$S(\mathbf{x})_{0,T} \neq S(\mathbf{x}')_{0,T}$ whenever $\mathbf{x}\neq \mathbf{x}'$, with $\mathbf{x},\mathbf{x}'\in K$,
    \item\label{it:one_in_algebra} $\mathbf{1} = (1,0,\ldots) \in D$,
    \item\label{it:point-sep1} $D$ separates points in $\bigoplus_{n=0}^\infty E^{\otimes n} $,
    \item\label{it:algebra} for all $l,l' \in D$ there exists an $l''\in D$ such that 
    \[
    \forall\, \mathbf{x}\in K \colon \inprod{S(\mathbf{x})_{0,T}}{l}\inprod{S(\mathbf{x})_{0,T}}{l'} = \inprod{S(\mathbf{x})_{0,T}}{l''}.
    \]
\end{enumerate}
Then for any $\tau$-continuous function $f\colon K \to \R$ and for all $\epsilon>0$, there exists $l \in D$ such that \[
    \sup_{\mathbf{x}\in K} | \inprod{S(\mathbf{x})_{0,T}}{l} - f(\mathbf{x}) | < \epsilon.
    \]
\end{theorem}
In other words, Theorem~\ref{thm: general setup of uat intro} ensures that we can approximate continuous functions on certain compact subsets of $\mathcal{C}^{\alpha}([0,T],E)$ by elements of the following algebra:
\begin{equation}\label{eq: algebra}
\mathcal{A} = \{ \langle S(\cdot)_{[0,T]} , l \rangle \colon l\in D \}.
\end{equation}
We briefly discuss Assumptions~\eqref{it:continuous}-\eqref{it:algebra}. Assumption~\eqref{it:one_in_algebra} simply ensures that $\mathcal{A}$ is non-vanishing, whereas Assumption~\eqref{it:continuous} ensures that the elements of $\mathcal{A}$ are in fact continuous maps.
Assumption~\eqref{it:algebra} ensures that $\mathcal{A}$ is indeed an algebra. This property is satisfied if we restrict ourselves to so-called \emph{weakly geometric} rough paths, see Section~\ref{sec: weakly geometric}, i.e., $K\subset \mathcal{C}^{\alpha}_g([0,T];E)$. 
Assumptions~\eqref{it:point-sep2} and~\eqref{it:point-sep1} guarantee that $\mathcal{A}$ separates points in $K$. As $D\subseteq \oplus_{n=0}^{\infty} (E^*)^{\otimes_a n}$, we typically require that $E$ has the approximation property to ensure Assumption~\eqref{it:point-sep1}. Assumption~\eqref{it:point-sep2} is satisfied if one considers time-extended paths, see Section~\ref{ssec: time-extended rough paths}. 

The key issue now is that \emph{we have not yet specified the topology $\tau$} in Theorem~\ref{thm: general setup of uat intro}. Indeed, we employ Theorem~\ref{thm: general setup of uat intro} to obtain a universal approximation theorem for the following two cases:
\begin{enumerate}[label=(\Roman*)]
    \item\label{it:normtopology} $\tau$ is the topology induced by the $\alpha$-H\"older metric $\varrho_{\alpha}^{\text{hom}}$, see Section~\ref{sec: norm-compact UAT},
    \item\label{it:weak*topology} $\tau$ is the so-called $i^*$-topology, a topology inherited from a weak$^*$-topology on a suitably chosen H\"older space, see Section~\ref{sec: norm-bounded UAT}.
\end{enumerate}
The universal approximation theorem in Case~\ref{it:normtopology} follows immediately from Theorem~\ref{thm: general setup of uat intro}, the considerations regarding Assumptions~\eqref{it:continuous}-\eqref{it:algebra} above, and the fact that the signature map is $\varrho_{\alpha}^{\text{hom}}$-continuous, see~\cite[Theorem 3.10]{lyons_differential_2007}.\par 
Case~\ref{it:weak*topology} requires significantly more work -- Sections~\ref{sec: predual of holder}--\ref{sec: i star continuity lyons lift} deal with the proof of the UAT for this case. However, this case has the great advantage that the Banach-Alaoglu theorem ensures that \textbf{$\varrho_{\alpha}^{\textnormal{hom}}$-bounded sets are relatively compact in the $i^*$ topology}, i.e., compact sets are easier to identify. Here we would like to stress that while the H\"older space $C^{\alpha}([0,T],\R^d)$ embeds compactly into $C^{\beta}([0,T],\R^d)$ whenever $\alpha>\beta$, this is no longer true if we replace $\R^d$ by an infinite-dimensional space $E$.

Regarding the proof for Case~\ref{it:weak*topology} a few remarks are at hand: firstly, the space $\mathcal{C}^{\alpha}([0,T];E)$ is not a linear space (iterated integrals do not commute with pointwise summation). Thus, there is no weak-$^*$ topology on $\mathcal{C}^{\alpha}([0,T];E)$, instead, we embed $\mathcal{C}^{\alpha}([0,T];E)$ into a (linear) H\"older space and inherit the weak-$^*$ topology from there; we refer to the resulting topology as the $i^*$-topology. Sections~\ref{sec: predual of holder} and~\ref{sec: weak star topology on alpha rough} deal with this construction, as well as providing an explicit representation of the pre-dual of a H\"older space. Note that the construction of the $i^*$ topology requires that the topological tensor spaces $E^{\otimes n}$ allow for a predual, in particular, it generally does not make sense to consider injective tensor norms in this setting. Finally, a proof that the signature map $S(\cdot)_{[0,T]}$ is continuous with respect to the $i^*$ topology can be found in \cite{cuchiero_global_2023} \emph{for the case $E=\R^d$}. This proof does not extend directly to the infinite-dimensional setting (this is related to $C^{\alpha}([0,T],E)$ \emph{not} embedding compactly into $C^{\beta}([0,T],E)$), but we can reduce to the finite-dimensional setting using a carefully chosen `projection'. For details we refer to Section~\ref{sec: i star continuity lyons lift}.

When compared to analogous results in finite-dimensional setting, in particular to the results in~\cite{cuchiero_global_2023}, there are several challenges we had to overcome to arrive at our UATs. We already mentioned the need to work with topological tensor spaces, which considerably raises the level of technicality. Additionally in infinite dimensions the two characterizations of group-like elements, which allow us to define weakly geometric paths, are not neccessarily equivalent. Nevertheless, \cite{cuchiero_global_2023} was still an important inspiration, although some objects are treated differently: in \cite{cuchiero_global_2023} (weakly) geometric rough paths are defined as Hölder continuous paths taking values in the group-like elements under the Carnot--Carath\'eodory metric, and an alternative construction of the predual $C^{\alpha}([0,T],E)$ is given. The Carnot--Carath\'eodory metric is not finite on all elements of $G(E)$ in infinite dimensions (see  Remark~\ref{rem:cc_problems}), and as such would need to be replaced by an alternative homogenuous norm (see e.g. \cite[Definition 7.34]{friz_multidimensional_2010}). We believe our approach to be more direct.

\subsection{Outlook} A natural direction for future research, motivated by the finite-dimensional case, is to investigate under what circumstances an SPDE can be interpreted as a random differential equation (RDE) driven by the (Stratonovich) rough path associated with an infinite-dimensional Brownian motion or an infinite-dimensional Ornstein--Uhlenbeck process. One can then study whether the solution map of the RDE depends continuously (with respect to the relevant topology) on the underlying rough path. Establishing such continuity would enable us to apply our UAT, for example to approximate models within the Heath–Jarrow–Morton–Musiela framework and to explore optimal control problems in energy markets. In particular, approximating functions of the solution of the SPDE by a linear functional of the signature of an infinite-dimensional rough path would allow us to treat quadratic-type optimal control problems, analogously to the finite-dimensional case in \cite{lyons2020non}. Our UATs also provide a justification for constructing infinite-dimensional volatility models for which the volatility operator is taken to be a linear map of the signature of an infinite-dimensional Brownian rough path.

Another direction for future research is to extend the UAT for norm-bounded sets of geometric rough paths, to a global UAT over the entire rough path space. This could be achieved by exploiting weighted function spaces to control growth at infinity developed in~\cite[Section 5]{cuchiero_global_2023}. A global UAT would ensure that these approximations remain valid uniformly for all input paths, rather than only on high-probability or truncated sets, thereby strengthening the theoretical foundations of the potential applications. In particular, this would provide global approximation guarantees for the applications discussed above: approximating models within the HJMM framework and  constructing operator-valued stochastic volatility models.

\subsection{Structure of the article}

In Section~\ref{sec: intro to signatures} we provide a brief introduction to infinite-dimensional rough paths and their signatures. Having introduced all relevant concepts, we present our main results -- the universal approximation theorems (UATs) -- in Section~\ref{sec: UATs}. The remaining sections contain the ingredients necessary to prove the UAT with respect to the weak-$^*$ topology: as the space of H\"older continuous geometric rough paths is not a linear space (the sums of two paths violates Chen's identity), we embed this space into a H\"older space and inherit the weak-$^*$ topology via this embedding -- this is explained in Section~\ref{sec: weak star topology on alpha rough}, and a characterization of the predual is given in Section~\ref{sec: predual of holder}. Finally, Section~\ref{sec: i star continuity lyons lift} deals with the weak-$^*$ continuity of the Lyons lift; the mapping that builds the signature of a given rough path.
Relevant aspects of topological tensor spaces are gathered in Appendices~\ref{sec: tensor spaces and norms}--~\ref{app: tensor spaces}.

\subsection{Notation}
We set $\N=\{0,1,\ldots\}$ and $\N_{>0}= \{1,2,\ldots\}$. 
The algebraic dual of a vector space $V$ is denoted by $V^*$, and we denote the topological dual of a Banach space $(E,\|\cdot\|_{E})$ by $(E^*,\|\cdot\|_{E^*})$. 
Let $X,Y$ be Banach spaces, the $\mathcal{L}(X,Y)$ denotes the Banach space of bounded linear operators from $X$ to $Y$ endowed with the operator norm; we set $\mathcal{L}(X):=\mathcal{L}(X,X)$.
For $\{x_{\lambda}\}_{\lambda \in I}$ a net in a Banach space $E$ that allows for a predual, we write $x_\lambda\wkst x \in E$ to indicate that the net converges in the weak-$^*$ topology to $x$.\par

The tensor algebra $T((E))$ over a Banach space $E$, and related algebras $T_1((E)))$, $T_0((E))$ and $T_a(F)$, are defined in Section~\ref{ssec: tensor algebra}. We use $\hat{\otimes}$ to denote the product on the tensor algebra, see Definition~\ref{def: tensor algebra}.

For $n,m\in \N$ we let $\text{Sh}(n,m)$ denote the set of all $(n,m)$-shuffles, see Definition~\ref{def: shuffle}.

For the definition of $\mathcal{C}^{\alpha,n} ([0,T];E )$ and $\mathcal{C}^{\alpha} ([0,T];E )$, spaces of H\"older continuous multiplicative functionals, see page~\pageref{def: holder multiplicative functionals}. For the definition of $\mathcal{C}_g^{\alpha,n} ([0,T];E )$ and $\mathcal{C}_g^{\alpha} ([0,T];E )$, spaces of \emph{weakly geometric} H\"older continuous multiplicative functionals, see Section~\ref{sec: weakly geometric}.

The $i^*$-topology is introduced in Definition~\ref{def: i* star topology}.

\section{A brief introduction to signatures and rough paths in infinite dimensions}\label{sec: intro to signatures}
Signatures and rough paths are closely related concepts, and iterated integrals play an important role in both. To explain this in more detail we begin by introducing iterated integrals of a smooth, real-valued path $x:[0,T]\rightarrow \R$. The $1$-times iterated integral of $x$ over the interval $[s,t]$ is given by $\mathbf{x}_{s,t}^{(1)} = \int_s^t \D x_r = x_t - x_s$, $0\leq s\leq t \leq T$. For $k \geq 2$, the $k$-times iterated integral of $x$ over the interval $[s,t]$ is given by
\begin{equation}
\mathbf{x}_{s,t}^{(k)} = \int_{s}^{t} \int_s^{t_k} \ldots \int_s^{t_{2}} \D x_{t_{1}} \ldots \D x_{t_{k}},\quad 0\leq s\leq t \leq T.
\end{equation}
The theory of signatures revolves around the fact that properties of the whole path $x$ are efficiently encoded by its \emph{signature}; the $\R$-valued sequence $\{\mathbf{x}_{0,T}^{(k)}\}_{k\in \N}$.\par 
The theory of rough paths concerns the situation when $x$ is not smooth enough to unambiguously define the iterated integrals, in particular, in the situation that $x$ is not differentiable, but still $\alpha$-H\"older continuous for some $\alpha \in (0,1)$. In this case, it has been shown that if one can somehow \emph{prescribe} the iterated integrals over the intervals $[s,t]$ up to order $\lfloor \nicefrac{1}{\alpha} \rfloor$ \emph{for all} $0\leq s \leq t\leq T$, then this suffices to unambiguously obtain the higher order iterated integrals. This process is known as the \emph{Lyons lift} (see~\cite[Theorem 3.7]{lyons_differential_2007}); it paves the way for once again encoding properties of the whole path in terms of its signature. \par 
Note that prescribing the iterated integrals is more natural than it may seem at first sight: for example, in the setting of continuous semimartingales (which are typically $\alpha$-H\"older continuous for all $\alpha \in [0,\frac{1}{2})$ but no smoother than that), the $1$-times iterated integral is provided by It\^o calculus (and can be defined both in the It\^o and the Stratonovich sense). \par 
If we now turn to an $\R^d$-valued path $x\colon [0,T]\rightarrow \R^d$, we see that $(\R^{d})^{\otimes k} (=\R^{kd})$ is the natural state space for $\mathbf{x}^{(k)}_{s,t}$, the $k$-times iterated integral over $[s,t]$: we have 
\begin{equation}\label{eq: iterated integrals in tensor form}
\mathbf{x}^{(k)}_{s,t} = \int_{s}^{t} \int_s^{t_1} \ldots \int_s^{t_{k}} \D x_{t_{k+1}} \otimes \ldots \otimes \D x_{t_{1}}
\end{equation}
Extending this idea to paths taking values in an infinite-dimensional real Banach space $E$, we are confronted with the fact that integration involves taking limits: while $(\R^d)^{\otimes k}$ is complete under any tensor norm, this is no longer the case for the algebraic tensor space $\prescript{k+1}{j=1}{\bigotimes}_{a} E$. Thus, the natural state space for the $k$-times iterated integral is a topological tensor space $\prescript{k+1}{j=1}{\bigotimes}_{\| \cdot \|} E $. Part of the challenge in the infinite-dimensional setting is that different tensor norms give rise to different topological tensor spaces, and one must identify which tensor norms are suitable. \par 
In this paper, we deal with the theory of signatures applied to infinite-dimensional stochastic differential equations (and stochastic partial differential equations). In particular, we are applying the theory of signatures to infinite-dimensional rough paths. As such, we begin by providing a short introduction to infinite-dimensional signatures and rough paths. The reader who is unfamiliar with infinite-dimensional topological tensor spaces is referred to Appendix~\ref{sec: tensor spaces and norms} for an overview. In Section~\ref{ssec: tensor algebra} we introduce the relevant structures for defining signatures; Section~\ref{sec: holder norms and rough paths} discusses the relevant structures for rough paths. The Lyons lift is discussed in Section~\ref{sec: Lyons lift}. In Section~\ref{sec: weakly geometric} we discuss weakly geometric rough paths, roughly speaking these are rough paths that respect the product rule. Finally, in Section~\ref{ssec: time-extended rough paths} we discuss so-called time-extended rough paths. Overal, we provide only the absolute minimum needed to present our results, more detailed introductions to the topic can be found in e.g.~\cite{friz_course_2020, friz_multidimensional_2010, lyons_differential_2007}.

\subsection{The tensor algebra and multiplicative functionals}\label{ssec: tensor algebra}

The natural structure for infinite-dimensional signatures is the \emph{tensor algebra}, see Definition \ref{def: tensor algebra} below. The algebraic structure is designed to encapture \emph{Chen's relation for iterated integrals}. To grasp this relation, observe that if $x\colon [0,T]\rightarrow\R^d$ is smooth, then for all $r,s,t\in [0,T]$ one has:
\begin{equation}\label{eq:Chensimple}
\mathbf{x}^{(2)}_{s,t} = \mathbf{x}^{(2)}_{s,r} + \mathbf{x}^{(1)}_{s,r}\otimes \mathbf{x}^{(1)}_{r,t} + \mathbf{x}^{(2)}_{r,t} . 
\end{equation}
Analogous identities hold for higher order iterated integrals, and Chen's relation provides a description of these identities in terms of the tensor algebraic relations. The elements of the tensor algebra satisfying Chen's relation are called \emph{multiplicative functionals}, see Definition~\ref{def: multiplicative functional} below. 

As mentioned above, if $E$ is a real Banach space and $x\colon [0,T]\rightarrow E$ is smooth, then the $k$-times iterated integral of $x$ takes  values in a topological tensor space $\prescript{k}{j=1}{\bigotimes}_{\| \cdot \|} E$ (see Definition~\ref{def: topological tensor space}). Some technicalities must be taken into account when choosing the tensor norm: following~\cite[Definition 1.25]{lyons_differential_2007} we assume we are dealing\footnote{Technically, the assumptions in~\cite{lyons_differential_2007} are slightly weaker, although this does not seem to have any consequences in practice. We refer to Remark~\ref{rem: note on lyons definition of admissible} for details.}

\begin{definition}\label{def: admissible crossnorms}
Let $E$ be a real Banach space and for $n\in \N_{> 0}$ let $\|\cdot \|_{n}$ be a norm on the algebraic tensor space $E^{\otimes_a n}$ (see Definition~\ref{def: algebraic tensor space}). We say that $\{\| \cdot \|_{n}\}_{n\in \N_{>0}}$ is a \emph{admissible family of tensor norms for $E$} if it is a symmetric strong crossnorm family, and we set $E^{\otimes n} := \prescript{n}{i=1}{\bigotimes}_{\| \cdot \|_n} E$, $n\in \N_{>0}$. We adopt the convention that $E^{\otimes 0} = \R$.
\end{definition}

\begin{remark}
In what follows, whenever testing against elements in the dual plays a role, we typically also need that the tensor norms are \emph{reasonable} (see Definition \ref{def: reasonable crossnorm d geq 3}) or even the stronger condition of being \emph{strongly uniform} (see Definition \ref{def:strong-unif-crossnorm}). Whenever these additional conditions are needed this is stated explicitely.
\end{remark}

\begin{definition}[Tensor algebra, see e.g.\ Definition 2.4 in \cite{lyons_differential_2007}]\label{def: tensor algebra}
     Let $E$ be a real Banach space and let $\{\| \cdot \|_{n}\}_{n\in \N_{>0}}$ be an admissible family of tensor norms for $E$. Then, \emph{the space of formal series of tensors} $T((E))$ is defined as the following space of sequences:
    \[
    T((E)) \coloneqq \{ \mathbf{a} \mid \mathbf{a} = (\mathbf{a}^{(0)}, \mathbf{a}^{(1)},\ldots) \text{ with } \mathbf{a}^{(n)}\in E^{\otimes n} \text{ for all } n\in \N \}.
    \]
    Let $\mathbf{a} = (\mathbf{a}^{(0)}, \mathbf{a}^{(1)},\ldots)$ and $\mathbf{b} = (\mathbf{b}^{(0)}, \mathbf{b}^{(1)},\ldots)$ be formal series of tensors. We define their addition and product by

    \[
    \mathbf{a} + \mathbf{b} \coloneqq (\mathbf{a}^{(0)}+\mathbf{b}^{(0)},\mathbf{a}^{(1)} +\mathbf{b}^{(1)}, \ldots),
    \]
    \begin{subequations}
    and
    \begin{align}
    \mathbf{a} \hat \otimes \mathbf{b} \coloneqq (\mathbf{c}^{(0)},\mathbf{c}^{(1)},\ldots),
    \end{align}
    where for all $n\geq 0$
    \begin{align}
    \mathbf{c}^{(n)}  = \sum_{i=0}^n \mathbf{a}^{(i)} \otimes \mathbf{b}^{(n-i)}.
    \end{align}
    \end{subequations}
    (Note that we set $c \otimes \mathbf{x} = \mathbf{x} \otimes c = c\mathbf{x} $ for all $c\in \R$, $\mathbf{x}\in E^{\otimes n}$, $n\in \N$.)
    We write respectively $\mathbf{0} \coloneqq (0, 0, \ldots)$ and $\mathbf{1} = (1,0,0,\ldots)$ for the addition and product unit elements. The space $T((E))$ endowed with addition and product defined above is called a \emph{tensor algebra} (over $E$).
\end{definition}

Let $E$ be a real Banach space, let $x\colon [0,T]\rightarrow E$ be smooth, and let $s,t\in [0,T]$. Setting $\mathbf{x}^{(0)}_{s,t}   = 1$ (we will address this convention shortly) and defining $\mathbf{x}^{(n)}_{s,t}$, $n\in \N_{>0}$, to be the $n$-times iterated integral of $x$ over $[s,t]$ (see~\eqref{eq: iterated integrals in tensor form}), we obtain that  $\{\mathbf{x}^{(n)}_{s,t}\}_{n\in \N}$ is an element of the tensor algebra (the fact that the involved tensor norms are crossnorms implies that the iterated integrals exist).

However, there are situations when we only want to consider the first $n$ iterated integrals, to this end we introduce the \emph{truncated tensor algebra}:

\begin{definition}[Truncated tensor algebra, see e.g.\ Definition 2.5 in \cite{lyons_differential_2007}]\label{def:truncated-topological-tensor}
    Let $E$ be a real Banach space and let $\{\| \cdot \|_{n}\}_{n\in \N_{>0}}$ be an admissible family of tensor norms for $E$. Then we define $T^{(n)}(E)$, the \emph{truncated tensor algebra of order $n$}, to be the following space of sequences:
    \[
    T_{}^{(n)}(E) \coloneqq \{ \mathbf{a} \mid \mathbf{a} = (\mathbf{a}^{(0)}, \mathbf{a}^{(1)},\ldots, \mathbf{a}^{(n)}) \text{ with } \mathbf{a}^{(i)}\in E^{\otimes i} \text{ for all } 0 \leq i\leq n \}.
    \]
    Alternatively, we can define the ideal $B_n (E)$ as
    \[
    B_n(E) = \{ \mathbf{a} \in T((E)) \mid \mathbf{a}^{(0)} = \ldots = \mathbf{a}^{(n)} = 0 \}, 
    \]
    and then set $T^{(n)}(E) = T((E)) / B_n(E)$. The truncated tensor algebra inherits the algebra structure from the tensor algebra. 
\end{definition}
\begin{remark}\label{rem: banach space structure of truncated tensor algebra}
    The truncated tensor algebra $T^{(n)}(E)$ can also be interpreted as a direct sum:
    \[
    T^{(n)}(E) = \R \oplus E \ldots \oplus E^{\otimes n}.
    \]
    This interpretation immediately allows us to endow $T^{(n)}(E)$ with a norm that turns it into a Banach space.
\end{remark}

The algebraic structure introduced above allows us to efficiently encode the above-mentioned \emph{Chen's relation}. 

\begin{definition}[Multiplicative functionals, see e.g.\ Definition 3.1 in~\cite{lyons_differential_2007}]\label{def: multiplicative functional}
Let $E$ be a real Banach space, let $\{\| \cdot \|_{n}\}_{n\in \N_{>0}}$ be an admissible family of tensor norms for $E$, and let $\mathbf{x}\colon [0,T]^2 \rightarrow T((E))$ (or $\mathbf{x}\colon [0,T]^2 \rightarrow T^{(n)}((E))$). Then $\mathbf{x}$ is called a \emph{multiplicative functional} if it satisfies Chen's relation, i.e., if 
\begin{equation}\label{eq: Chen final?}
    \mathbf{x}_{s,t} = \mathbf{x}_{s,u} \hat\otimes \mathbf{x}_{u,t}, \quad  s,u,t \in [0,T].
\end{equation}
\end{definition}

Note that Chen's relation requires 
$\mathbf{x}^{(0)}\equiv 1$, this motivates the following definition: 

\begin{definition}\label{def: tensor algebra with 1}
    Let $E$ be a real Banach space and let $\{\| \cdot \|_{n}\}_{n\in \N_{>0}}$ be an admissible family of tensor norms for $E$. We define
    \[
    T_1((E)) \coloneqq \{ \mathbf{a} \in T((E)) \mid \mathbf{a}^{(0)} = 1 \}    \]
    and
    \[
    T^{(n)}_1((E)) \coloneqq \{ \mathbf{a} \in  T^{(n)}(E)\mid \mathbf{a}^{(0)} = 1 \} .   \]
\end{definition}

A consequence of Chen's relation is that for a multiplicative functional $\mathbf{x}$ and $s,t \in [0,T]$

\[
\mathbf{x}_{t,s} = (\mathbf{x}_{s,t})^{-1},
\]

where the inverse is with respect to the tensor product. This inverse exists for all elements in $T_1((E))$ or $T_1^{(n)}(E)$ and can be explicitly constructed:
\begin{proposition}[From Section 2.2.1 of \cite{lyons_differential_2007}]\label{prop: group structure}
Let $E$ be a real Banach space and let $\{\| \cdot \|_{n}\}_{n\in \N_{>0}}$ be an admissible family of tensor norms for $E$. Let $\mathbf{a}$ be an element of $T_1((E))$ or $T_1^{(n)}(E)$, then the inverse of $\mathbf{a}$ with respect to the tensor product is given by
    \begin{equation}\label{eq: explicit inverse tensor algebra element}
    \mathbf{a}^{-1} = \sum_{i=0}^\infty (\mathbf{1} - \mathbf{a})^{\hat\otimes i}.
    \end{equation} 
\end{proposition}
\begin{proof}
    First note that the infinite sum of Equation \eqref{eq: explicit inverse tensor algebra element} is well defined, as each term in the sequence is a sum of only finitely many non-zero elements.
    Writing $\mathbf{b} \coloneqq (\mathbf{1}- \mathbf{a})$, we get this is equivalent to 
    \[
   ( \mathbf{1}-\mathbf{b} ) \hat \otimes\sum_{i=0}^\infty \mathbf{b}^{\hat\otimes i } = \left(\sum_{i=0}^\infty \mathbf{b}^{\hat\otimes i }\right) \hat\otimes( \mathbf{1}-\mathbf{b} ) = \mathbf{1}
    \]
\end{proof}

In what follows we shall also use the algebraic tensor algebra; we view it as a natural space for functionals acting on the tensor algebra. The key difference between the algebraic tensor algebra and the tensor algebra from Definition~\ref{def: tensor algebra} is that the algebraic tensor algebra is the direct sum of the \emph{algebraic} tensor spaces, i.e., the vector spaces of linear combinations of simple tensors (without a topology).  

\begin{definition}\label{def: algebraic tensor algebra}
    Let $F$ be a vector space. The \emph{algebraic tensor algebra} $T_a(F)$ and the \emph{truncated algebraic tensor space} $T^{(n)}_a(F)$ are given by
    \[
    T_a(F) \coloneqq  \bigoplus_{i=0}^\infty F^{\otimes_a i}\quad \text{and} \quad T_a^{(n)}(F) \coloneqq  \bigoplus_{i=0}^n F^{\otimes_a i},
    \]
    where $F^{\otimes_a i}$ denotes the $i$-th algebraic tensor power of $F$.
\end{definition}

Note that if $E$ is a Banach space, $\{\| \cdot \|_{n}\}_{n\in \N_{>0}}$ an admissible family of tensor norms for $E$, and $F\subseteq E^*$, then elements of $T_a(F)$ represent elements of the (algebraic) dual of $T((E))$. Indeed, we set
    \[
    \inprod{\mathbf{x}}{\mathbf{y}} \coloneqq \inprod{\mathbf{x}^{(i)}}{\mathbf{y}},  \quad \mathbf{y}\in F^{\otimes_a i},\mathbf{x} \in T((E)) 
    \]
(the fact that the tensor norms are crossnorms ensure that the above is well-defined). The definition extends to $\mathbf{y}\in T_a(F)$ by linearity.

\subsection{\texorpdfstring{$\alpha$-}{}H\"older norms and the \texorpdfstring{$\varrho^{\text{hom}}_{\alpha}$}{H\"older}-metric}\label{sec: holder norms and rough paths}

In order to interpret a multiplicative functional as a rough path one must impose regularity conditions. These regularity conditions are typically expressed either in terms of $p$-variation or in terms of $\alpha$-H\"older continuity; in our work we use the latter.

\begin{definition}[$\alpha$-Hölder continuity, see e.g.\ Definition 5.1 in~\cite{friz_multidimensional_2010}]\label{def: normal alpha holder}
    Let $0< \alpha \leq 1$, and let $(E,\|\cdot\|)$ be a real Banach  space. For any path $x \colon [s,t] \to E $ we define the \emph{$\alpha$-Hölder coefficient }
    \begin{equation}\label{eq: alpha holder continuity}
        \|x\|_{C^{\alpha}} \coloneqq \sup_{s\leq u < v \leq t} \frac{\|x_u - x_v\|}{|v-u|^\alpha}.
    \end{equation}
    If a path has a finite $\alpha$-Hölder coefficient, it is said to be \emph{$\alpha$-Hölder continuous}. The vector space of all such functions is denoted by $C^\alpha ([s,t];E)$. Moreover, we define the space $C_0^\alpha ([s,t];E)$ as
    \[
    C_0^\alpha ([s,t];E) \coloneqq \{ x \in C^\alpha ([s,t];E) \mid x(s) = 0 \}.
    \]
    On this space, the $\alpha$-Hölder coefficient is a norm.
\end{definition}
One could also define these concepts respectively for $\alpha>1$, but this would have no real use, since functions belonging to such a class would be constant, see \cite[Proposition 5.2]{friz_multidimensional_2010}. 

There are essentially two ways of extending the definition of $\alpha$-Hölder continuity to multiplicative functionals. Either one defines a notion of these concepts directly for higher order terms of the signature (see, e.g.,~\cite{friz_course_2020}), or one defines a new metric on the truncated tensor algebra $T^{(n)}(E)$ and then considers H\"older continuity with respect to this new metric (see, e.g.,~\cite{cuchiero_global_2023}). Both approaches give rise to the same set of functionals and the associated H\"older coefficients can be shown to be equivalent (see, e.g., the introduction of ~\cite[Chapter 8]{friz_multidimensional_2010}, although only the finite dimensional case is discussed, the argument remains). In this work, we take the first approach.

\begin{definition}[Homogeneous $\alpha$-Hölder metric, see e.g.\ Definition 2.1~\cite{friz_course_2020}]\label{def: homogeneous alpha holder}
Let $E$ be a real Banach space and let $\{\| \cdot \|_{n}\}_{n\in \N_{>0}}$ be an admissible family of tensor norms for $E$. Let $\mathbf{x},\mathbf{y}\colon [0,T]^2\rightarrow T_1^{(n)}(E)$ be multiplicative functionals (see Definition~\ref{def: multiplicative functional}), and let $0<\alpha\leq 1$. Then the \emph{homogeneous $\alpha$-Hölder metric} ${\varrho}^{\text{hom}}_{\alpha}(\mathbf{x},\mathbf y)$ of order $n$ is given by
  \[
  \varrho^{\text{hom}}_{\alpha}(\mathbf{x},\mathbf{y}) \coloneqq \max_{1\leq i \leq n} \sup_{u,v\in [0,T], u \neq v} \left( \frac{\|\mathbf x^{(i)}_{u,v} - \mathbf y^{(i)}_{u,v}\|_i}{|v-u|^{i \alpha}}\right)^{\nicefrac{1}{i}}.
  \]
  A multiplicative functional $x$ is said to be \emph{$\alpha$-Hölder continuous} if its \emph{homogeneous $\alpha$-Hölder `norm'}\footnote{The terminology is somewhat of a misnomer as the set of multiplicative functionals is not a linear space.} $\|\mathbf{x}\|_{\mathcal{C}^{\alpha}}$ is finite, where 
  \[
  \|\mathbf{x}\|_{\mathcal{C}^{\alpha}} \coloneqq \max_{1\leq i \leq n} \sup_{u,v\in [0,T], u \neq v} \left( \frac{\|\mathbf x^{(i)}_{u,v}\|_i}{|v-u|^{i \alpha}}\right)^{\nicefrac{1}{i}}.
  \] 
  The set of all these $\alpha$-Hölder continuous multiplicative functionals of order $n$ is given by 
  \[
   \mathcal{C}^{\alpha,n} ([0,T];E ) \coloneqq \{ \mathbf{x} \colon [0,T]^2 \to T^{(n)}(E) \mid \mathbf{x} \text { is a multiplicative functional and } \|\mathbf{x}\|_{\mathcal{C}^\alpha} < \infty\}. \label{def: holder multiplicative functionals}
  \]
  Unless stated otherwise, we assume this space to be equipped with the topology induced by the metric ${\varrho}^{\text{hom}}_{\alpha}$. In the special case that $n = \lfloor\nicefrac{1}{\alpha }\rfloor$, $n$ is omitted (see Section \ref{sec: Lyons lift} for why this case is special) and the space is denoted by 
  \[
  \mathcal{C}^{\alpha} ([0,T];E ) \coloneqq \mathcal{C}^{\alpha,\lfloor\nicefrac{1}{\alpha}\rfloor} ([0,T];E ) .
  \] \label{def: natural holder multiplicative functionals}
\end{definition}

\subsection{The Lyons lift and the signature}\label{sec: Lyons lift}
As mentioned in the introduction of this section, the Lyons lift provides a method for lifting a multiplicative functional $\mathbf{x} \in \alpharough $ to a multiplicative functional in $\mathcal{C}^{\alpha, n} ([0,T];E)$ for $n > \lfloor \nicefrac{1}{\alpha}\rfloor$:

\begin{thmdef}\label{thm: Lyons lift}
Let $E$ be a real Banach space and let $\{\| \cdot \|_{n}\}_{n\in \N_{>0}}$ be an admissible family of tensor norms for $E$. Let $0<\alpha \leq 1$, let $n\in \N_{>0}$ with $n\geq \lfloor \nicefrac{1}{\alpha} \rfloor$,  and let $\pi\colon T^{(n)}(E) \to T^{\left(\lfloor\nicefrac{1}{\alpha}\rfloor\right)}(E)$ be the natural projection. Then there exists a unique $\varrho^{\text{hom}}_{\alpha}$ continuous map
    \[
    S^{n}\colon \mathcal{C}^{\alpha}([0,T];E) \to  \mathcal{C}^{\alpha,n}([0,T];E),
    \]
    such that $\pi (S^n(\mathbf{x})_{s,t}) = \mathbf{x}_{s,t} $ for all $s,t \in [0,T]$ and all $\mathbf{x}\in \mathcal{C}^\alpha([0,T];E)$. Furthermore, there exists a $C>0$ such that $\|S^n(\mathbf{x})\|_\mathcal{C^\alpha} \leq C \|\mathbf{x}\|_\mathcal{C^\alpha}$, for all $\mathbf{x}\in \mathcal{C}^{\alpha}([0,T];E)$. The map $S^{n}$ is called the \emph{Lyons map}.
\end{thmdef}
\begin{proof}
This is essentially \cite[Theorem 3.7]{lyons_differential_2007}, although a different notation is used there. One sees this by taking $\omega(s,t) = C'(t-s)^\alpha$ for the control $\omega$ used in \cite[Theorem 3.7]{lyons_differential_2007}, where $C'$ is a constant depending on $\mathbf{x} \in \alpharough$.
\end{proof}

As the construction of the Lyons' lift is important for Section \ref{sec: i star continuity lyons lift}, we sketch it here. The construction uses induction. Indeed, for $n = \lfloor \nicefrac{1}{\alpha} \rfloor$, the map is just the identity. Assuming that the map $S^n$ exists, the map $S^{n+1}$ is constructed as follows: first, for $\mathbf{x} \in \alpharough$, one defines $\tilde{\mathbf{x}} \colon [0,T]^2 \rightarrow T^{(n+1)}(E)$ as follows:
\begin{equation}
\tilde{\mathbf{x}}:=(S^{n} (\mathbf{x}),0),
\end{equation}
where $0\in E^{\otimes (n+1)}$. 
For $s<t$ and $D=(t_1,\ldots,t_N)$ a partition of $[s,t]$ we define
\[
\tilde {\mathbf{x}}_{s,t}^D \coloneqq   \tilde {\mathbf{x}}_{s,t_{1}} \hat\otimes \ldots \hat\otimes \tilde {\mathbf{x}}_{t_{N},t};
\]
the Lyons lift of $\mathbf{x}$ at $(s,t)$ is given by the limit
\[
S^{n+1}(\mathbf{x})_{s,t} \coloneqq \lim_{|D|\to 0 } \tilde {\mathbf{x}}_{s,t}^D .
\]
The construction of $\mathbf{x}_{s,t}$ when $s>t$ is analogous.
See the proof of \cite[Theorem 3.7]{lyons_differential_2007} for details why this converges to the desired mapping. 

\begin{remark}\label{rem: lyons lift same notation}
    The construction of the Lyons lift of Theorem/Definition \ref{thm: Lyons lift} also allows one to define analogous maps $S^n \colon \mathcal{C}^{\alpha,k}([0,T];E) \to  \mathcal{C}^{\alpha,n}([0,T];E)$, for $\lfloor \frac{1}{\alpha} \rfloor<k < n$. Although these maps have a different domain, they can be considered as the `same' map, and so the same notation will be used for them. In particular, for any $\mathbf{x} \in \mathcal{C}^{\alpha,k}([0,T]; E)$, we have
\[
S^n(\mathbf{x})  = S^n (\pi (\mathbf{x})),
\]
where $\pi$ is the extension of the natural projection $\pi\colon T^{(k)}(E) \to T^{\lfloor \frac{1}{\alpha}\rfloor}(E)$ to $\pi\colon\mathcal{C}^{\alpha,k}([0,T];E)\to\alpharough$. 
\end{remark}

As we can lift a path in $\alpharough$ to $ \mathcal{C}^{\alpha,n}([0,T;]E)$ for arbitrary high $n$, we can also lift it to take values in the tensor algebra $T((E))$. In this case we denote the map by $S$. The object $S(\mathbf{x}) $ is still a multiplicative functional (see Theorem 3.7 of\cite{lyons_differential_2007}). Its main use is found in considering $S(\mathbf{x})_{0,T}$, an object knows as \emph{the signature of $\mathbf{x}$}:

\begin{definition}[Signature]\label{def:signature}
Let $E$ be a real Banach space and let $\{\| \cdot \|_{n}\}_{n\in \N_{>0}}$ be an admissible family of tensor norms for $E$. Let $0<\alpha \leq 1$, let $\mathbf{x} \in \alpharough$, and let $s,t\in [0,T]$. Then the object $S(\mathbf{x})_{s,t}\in T((E))$ is called the \emph{signature of $\mathbf{x}$ at $(s,t)$}.
\end{definition}
The signature $S(\mathbf{x})_{0,T}$ contains a lot of information on the multiplicative functional $\mathbf{x}$, but it does not fully describe it. For example, if a path goes to some point, and then goes back in exactly the same way, its signature is $\mathbf{1}$. In particular, the signature of such a path is indistinguishable from that of a path that remains constant. Another example is that two multiplicative functionals that are reparameterizations of each other have the same signature. These two cases are examples of so-called tree-like equivalences. It has been proven that the map $\mathbf{x}\mapsto S(\mathbf{x})$ is injective upto tree-like equivalences, see 
\cite[Theorem 1.1]{boedihardjo_signature_2016}. Tree-like equivalences can sometimes be eliminated by considering time-extended rough paths, see Section~\ref{ssec: time-extended rough paths}.

\subsection{Weakly geometric rough paths}\label{sec: weakly geometric}

When constructing a rough path `from scratch', one typically prescribes the iterated integral using e.g.\ It\^o or Stratonovich stochastic calculus. An advantage of the latter is that the resulting calculus satisfies the classical chain rule, e.g., $d[ W(t)^2] = 2W(t)\circ dW(t)$. For rough paths, the concept of `satisfying the classical chain rule' can be encoded formally by testing against shuffled functionals: one requires $\mathbf{x}\in T((E))$ to satisfy
    \begin{equation}\label{eq: weak group like formulation intro}
    \inprod{\mathbf{x}}{\mathbf{y}} \inprod{\mathbf{x}}{\mathbf{y}'} = \inprod{\mathbf{x}}{\mathbf{y}\shuffle \mathbf{y}'} ,\quad \mathbf{y}, \mathbf{y}' \in T_a(E^*)
    \end{equation}
(see Definition~\ref{def: shuffle} for the definition of the shuffle product $\shuffle$, and see Remark~\ref{def: shuffle inductive (is remark)} below for details). The celebrated Chen's theorem states that for \emph{finite-dimensional rough paths}, property~\eqref{eq: weak group like formulation intro} is equivalent to imposing that $\mathbf{x}_{s,t}$ is `group-like' for all $s,t\in [0,T]$, and the paths that satisfy these equivalent conditions are called \emph{weakly geometric} (see Definitions~\ref{def: group like elements strong} and~\ref{def: weakly geometric rough paths} below; roughly speaking, $\mathbf{x}\in T_1((E))$ is called `group-like' if there exists an $X$ in the free Lie algebra generated by $E$ such that $\mathbf{x} = \exp(X)$). A consequence of Chen's theorem is that tools and techniques from Lie theory can be used to analyze weakly geometric paths. For the sake of completeness, we note that the concept of \emph{geometric rough paths} can also be found in the literature; these are rough paths that can be obtained as the $\|\cdot \|_{\mathcal{C}^{\alpha}}$-limit of lifted differentiable paths. However, this concept does not play a significant role in our work.

Let us briefly reflect on the relevance of weakly geometric rough paths in the setting of this paper. For a start, property~\eqref{eq: weak group like formulation intro} lies at the heart of our approximation result, as it essentially implies that \emph{polynomials} of functionals applied to lower order terms of the rough path can be expressed as \emph{linear functionals} of higher-order terms. Unfortunately, it is not clear in the infinite-dimensional setting whether satisfying condition \eqref{eq: weak group like formulation intro} is equivalent to being group-like, yet many results in the literature concerning weakly geometric rough paths in fact rely on the `group-like' property. Thus an infinite-dimensional analogue of Chen's theorem is desirable: indeed, we prove that the two notions are equivalent if the state space $E$ satisfies the approximation property and the crossnorms involved are strongly uniform (see Definition~\ref{def:strong-unif-crossnorm} and Proposition~\ref{prop: equality of group like definitions} below).

We begin by introducing the concept of `group-like' elements. In doing so, we borrow language from the theory of Lie groups and Lie algebras (such as `exponential map' and `bracket'). However, we are working in the infinite-dimensional setting, and we do not want to get into the technicalities of infinite-dimensional Lie groups. Thus, we refrain from rigorously identifying an underlying Lie algebra/Lie group structure and only highlight the properties of the involved spaces that are of relevance for our work (but see \cite[Section 2]{grong_geometric_2022} for a more rigorous treatment). 

Let $E$ be a real Banach space and let $\{\| \cdot \|_{n}\}_{n\in \N_{>0}}$ be an admissible family of tensor norms for $E$. Let $T_0((E))$ denote the elements $\mathbf{x}$ of $T((E))$ with $\mathbf{x}^{0}=0$, and let $T_0^{(n)}(E)$ be defined analogously. We think of $T_0((E))$ as being the `Lie algebra' associated with the `Lie group' $T_1((E))$, the `Lie bracket' being given by 
\begin{equation}
[\mathbf{x},\mathbf{x}'] = \mathbf{x}\hat{\otimes}\mathbf{x}' - \mathbf{x}\hat{\otimes} \mathbf{x}',\qquad \mathbf{x},\mathbf{x}'\in T_0((E)).
\end{equation}
Note that while the Lie group and Lie algebra structure is not made rigorous, the bracket $[\cdot , \cdot] \colon T_0((E)) \times T_0((E)) \rightarrow T_0((E))$ is a well-defined alternating bi-linear map satisfying the Jacobi-identity. Moroever, we can define the `exponential map' $\exp\colon T_0((E)) \rightarrow T_1((E))$ which formally describes the flow in $T_1((E))$ starting in $\mathbf{1}\in T_1((E))$ in the direction $\mathbf{x}\in T_0((E))$:

\begin{definition}[Exponential mapping and logarithmic mapping]\label{def: exponential and logarithmic mapping}
Let $E$ be a real Banach space and let $\{\| \cdot \|_{n}\}_{n\in \N_{>0}}$ be an admissible family of tensor norms for $E$. We define the exponential mapping $\exp \colon T_0((E)) \to T_1((E))$ by its classical power expansion:
    \begin{equation}
        \exp (\mathbf{x}) \coloneqq \sum_{n=0} ^\infty\frac{\mathbf{x}^{\hat \otimes n}}{n!},
    \end{equation}
    for any $\mathbf{x} \in T_0((E))$, which is well defined as each term in the sequence is a sum of finitely many non-zero elements. Similarly, we define the logarithmic mapping $\log \colon T_1((E)) \to T_0((E))$ as
    \begin{equation}
        \log (\mathbf{x})  \coloneqq \sum_{n=1}^\infty \frac{(-1)^n}{n} ( 1 - \mathbf{x})^{\hat \otimes n},
    \end{equation}
     for any $\mathbf{x} \in T_1((E))$, which is again well defined as each term in the sequence is a sum of finitely many non-zero elements. For the truncated tensor algebra $T^{(n)}(E)$ the exponential and logarithmic mapping are defined analogously.
\end{definition}

It is possible to verify that the exponential and logarithmic mappings defined above are each-others inverse. The smallest subalgebra of the `Lie algebra' $T_0((E))$ that contains $E$ (when interpreted as a subspace of $T((E))$ is called the space of Lie series:

\begin{definition}[Lie series, see e.g.\ Definition 2.7 in \cite{cass_integration_2016}]\label{def: Lie Series}
Let $E$ be a real Banach space, let $\{\| \cdot \|_{n}\}_{n\in \N_{>0}}$ be an admissible family of tensor norms for $E$. For any $n,m\in \N_{>0}$ and any $\mathbf{x}\in E^{\otimes n},\mathbf{x}' \in E^{\otimes m}$, we define the Lie bracket 
    \[
    [\mathbf{x} , \mathbf{x}'] = \mathbf{x} \otimes \mathbf{x}' - \mathbf{x}' \otimes \mathbf{x}.
    \]
    Furthermore, for $n\geq 2$, we define the space $\mathcal{M}^{(n)}(E)$ to be
    \[
    \mathcal{M}^{(n)}(E)\coloneqq \overline{\text{span} \{ [x_1, [x_2, \ldots, [ x_{n-1}, x_n ]]] \mid x_i \in E \}}.
    \]
    We also set
    \[
    \mathcal{M}^{(1)}(E)\coloneqq E \quad \text{ and } \mathcal{M}^{(0)}(E)\coloneqq \{0 \} .
    \]
    We define the space of Lie series $\text{Lie}(E)$ [or $\text{Lie}_n(E)$] to be the space of elements $\mathbf{x} \in T_0((E))$ [or $\mathbf{x} \in T_0^{(n)}(E)$] such that for any $m\in \N$ [or $0\leq m \leq n$], we have that $\mathbf{x}^{(m)} \in  \mathcal{M}^{(n)}(E)$.
\end{definition}

We now identify group-like elements as being the image of a Lie series under the exponential mapping: 

\begin{definition}[Group-like elements, see e.g.\ Definition 2.9 in~\cite{cass_integration_2016}]\label{def: group like elements strong}    
Let $E$ be a real Banach space and let $\{\| \cdot \|_{n}\}_{n\in \N_{>0}}$ be an admissible family of tensor norms for $E$. We define the set of \emph{group-like elements} to be the image of the set of Lie series under the exponential mapping.
    \begin{equation}
        G(E) \coloneqq \exp ( \text{Lie} (E)).
    \end{equation}
    Similarly, for $n\in \N$ we set
    \[
    G^{n}(E) \coloneqq \exp ( \text{Lie}_n(E)).
    \]
\end{definition}

\begin{remark}\label{remark: group-like closed under multiplication}
It is a consequence of the fact that $\exp|_{T_0^{(n)}(E)}\colon T_0^{(n)}(E) \rightarrow T_1^{(n)}(E)$ is a diffeomorphism (\cite[Theorem 2.4]{cass_integration_2016}) and the Campbell--Baker--Hausdorff--Dynkin formula that if $\mathfrak{g}\subseteq T_0^{(n)}$ is a closed Lie subalgebra, then $\exp(\mathfrak{g})$ is a closed subgroup of $T_1^{(n)}$ (with respect to the tensor product), see \cite[Corollary 2.6]{cass_integration_2016}. In particular, $G^{n}(E)$ is a closed group, and thus $G(E)$ is a closed group.
\end{remark}

As announced in the introduction of this section, the concept of `satisfying the classical chain rule' can be encoded by testing against shuffled functionals. We now make this statement more precise.

\begin{definition}[Shuffle product, see e.g.\ Section 2.2.3 of~\cite{lyons_differential_2007}]\label{def: shuffle}
    Let $F$ be a vector space. Then, for each $n,m\in \N$, we define the \emph{shuffle product} $\shuffle \colon F^{\otimes_a n} \times F^{\otimes_a m}\to F^{\otimes_a (n+m)}$ as 
    \[
    \mathbf{y} \shuffle \mathbf{y}' \coloneqq \sum_{\sigma\in \text{Sh}(n,m)}P_{\sigma^{-1}} ( \mathbf{y} \otimes \mathbf{y}'),
    \]
    where $\mathbf{y}\in F^{\otimes_a n}$, $\mathbf{y}' \in F^{\otimes_a m}$, $P_\sigma$ is the permutation map defined as in Definition \ref{def: permutations on alg tensor space}, and $\text{Sh}(n,m)$ is the set of all $(n,m)$-shuffles, that is, all permutations $\sigma$ of $\{1,\ldots,n+m\}$ such that $\sigma(1) < \ldots < \sigma(n)$ and $\sigma(n+1) < \ldots < \sigma(n+m)$. 
    Note that the shuffle product extends to a bilinear map $\shuffle \colon T_a(F) \times T_a(F) \to T_a(F)$.
\end{definition}
\begin{remark}[See Definition 2.2 in \cite{cuchiero_signature-based_2023}]\label{def: shuffle inductive (is remark)}
    An alternative but equivalent definition of the shuffle product is given inductively: first, for $c,c'\in \R$, $\mathbf{y}\in ^{\otimes_a n}$, $n\in\N$ we set \begin{align*}
        c \shuffle c' &:= cc',& \mathbf{y}\shuffle c = c \shuffle \mathbf{y} & := c\mathbf{y}.
    \end{align*} 
    Next, for $\mathbf{y}_1\in F^{\otimes_a n}$,  $\mathbf{y}'_1\in F^{\otimes_a m}$, $y_2, y_2'\in F$, $n,m\in \N$, we set
    \begin{equation}\label{eq: shuffle prod definition}
    (\mathbf{y}_1 \otimes y_2) \shuffle (\mathbf{y}_1'\otimes y_2') 
    \coloneqq 
    ((\mathbf{y}_1\otimes y_2)\shuffle \mathbf{y}_1') \otimes  y_2' +  (\mathbf{y}_1 \shuffle (\mathbf{y}_1'\otimes y_2')) \otimes  y_2.
    \end{equation}
By induction and multilinear extension, these rules define the shuffle product $\shuffle \colon F^{\otimes_a n} \times F^{\otimes_a m} \rightarrow F^{\otimes_a (n+m)}$ for all $n,m\in\N$.
\end{remark}
The shuffle product allows us to define elements that we call \emph{weakly} group-like. When $E$ is finite-dimensional, the notions of group-like and weakly group-like coincide (this is known as Chen's theorem), hence the term `weakly group-like' is not used elsewhere in the literature. In the infinite-dimensional setting it is not clear whether these notions coincide in general, but Proposition~\ref{prop: equality of group like definitions} provides conditions under which this is the case.

\begin{definition}[Weakly group-like elements, see Definition 2.18 in \cite{lyons_differential_2007}]\label{def: weak group like elements dual rep}
Let $E$ be a real Banach space and let $\{\| \cdot \|_{n}\}_{n\in \N_{>0}}$ be an admissible family of tensor norms for $E$. Then we define the space of \emph{weakly group-like elements} $G_w(E)$ as the set of $\mathbf{x}\in T((E))$ such that
    \begin{equation}\label{eq: weak group like formulation}
    \inprod{\mathbf{x}}{\mathbf{y}} \inprod{\mathbf{x}}{\mathbf{y}'} = \inprod{\mathbf{x}}{\mathbf{y}\shuffle \mathbf{y}'} ,
    \end{equation}
    for all $\mathbf{y}, \mathbf{y}' \in T_a(E^*)$. Similarly, for $n\in \N$, we define $G^{(n)}_w(E)$ as the set of elements $\mathbf{x} \in T^{(n)}(E)$ such that Equation \eqref{eq: weak group like formulation} holds for all $\mathbf{y}, \mathbf{y}' \in T^{(n)}(E)$ such that $\deg (\mathbf{y}) + \deg(\mathbf{y}') \leq n$, where $\deg(\mathbf{z})$ stands for the highest order of $\mathbf{z}$ that is non-zero.
\end{definition} 

\begin{proposition}\label{prop: equality of group like definitions}
Let $E$ be a real Banach space and let $\{\| \cdot \|_{n}\}_{n\in \N_{>0}}$ be an admissible family of tensor norms for $E$. Then $G(E) \subseteq G_w(E)$, and, for any $n\in \N$, $G^{(n)}(E) \subseteq G_w^{(n)}(E)$. Furthermore, if $E$ has the approximation property, and if the norms over $E^{\otimes n}$ are strongly uniform crossnorms, we have an equality
    \[
    G_w(E) = G(E),
    \]
    and similarly, for any $n\in \N$, $G^{(n)}(E) = G_w^{(n)}(E)$.
\end{proposition}
\begin{proof}
    See Appendix \ref{sec: group-like elements}.
\end{proof}
\begin{remark}\label{rem:cc_problems}
   It does not hold true that all group-like elements in infinite dimensions can be represented as the lift of continuous finite-variation paths as is the case in the finite-dimensional setting. As a consequence, the Carnot--Carathéodory metric is not finite for all group-like elements. For more details, we refer the reader to \cite{grong_geometric_2022}.
\end{remark}
Having identified group-like elements, we can define weakly geometric rough paths.
\begin{definition}[Weakly geometric rough paths, see e.g.\ Definition 3.14 in~\cite{lyons_differential_2007}]\label{def: weakly geometric rough paths} Let $E$ be a real Banach space, let $\{\| \cdot \|_{n}\}_{n\in \N_{>0}}$ be an admissible family of tensor norms for $E$, let $n\in \N$ and let $\alpha \in (0,1]$. Then the space of \emph{weakly geometric rough paths} $\mathcal{C}^{\alpha,n}_g ( [0,T] ; E)$ is the set of multiplicative functionals $\mathbf{x} \in \mathcal{C}^{\alpha,n} ( [0,T] ; E)$ such that for all $s,t \in [0,T]$ we have $\mathbf{x}_{s,t} \in G^{(n)}(E)$ (see Definition \ref{def: group like elements strong}). If $n =\lfloor \nicefrac{1}{\alpha}\rfloor$, we drop the $n$ and denote this space by $\mathcal{C}^{\alpha}_g ( [0,T] ; E)$
\end{definition}

The Lyons lift $S^n$ of a weakly geometric rough path (see Definition~\ref{thm: Lyons lift} is still weakly geometric, which is shown in \cite[Corollary 3.9]{cass_integration_2016}:
\begin{theorem}[Corollary 3.9 in \cite{cass_integration_2016}]\label{thm: preservation of weakly geometric under lyons lift} Let $E$ be a real Banach space, let $\{\| \cdot \|_{n}\}_{n\in \N_{>0}}$ be an admissible family of tensor norms for $E$, let $\alpha \in (0,1]$ and let $n\in \N$, $n\geq \lfloor \nicefrac{1}{\alpha} \rfloor$. Then the Lyons lift $S^n$ maps weakly geometric rough paths to weakly geometric rough paths:
    \[
    S^n( \mathcal{C}_g^\alpha([0,T];E)) \subseteq \mathcal{C}^{\alpha,n}_g([0,T];E).
    \]
\end{theorem}

\subsection{Time-extended rough paths}\label{ssec: time-extended rough paths}
As mentioned in the end of Section \ref{sec: Lyons lift}, the signature of a path contains most of the relevant information of the multiplicative functional, but it does not describe it completely, as it cannot distinguish between so-called tree-like equivalent paths. In \cite{cuchiero_signature-based_2023} the time-extended signature is invoked to deal with this problem (see Lemma 2.6 in~\cite{cuchiero_signature-based_2023}). The idea is that a rough path $\mathbf{x}$ with corresponding path $t\mapsto x(t)$ can be extended to a rough path $\tilde{ \mathbf{x}}$ with corresponding path $t\mapsto (t,x(t))$, see Proposition~\ref{prop: time extension} below. The signature of such a time-extended rough path \emph{is} a complete characterization of the path, see Proposition~\ref{prop: when there is a time component then unique signature} below. In other words, the map $\mathbf{x} \mapsto S(\tilde{\mathbf{x}})$ is injective. For simplicity we only consider multiplicative functionals of $\alpha$-Hölder regularity for $\alpha > \frac{1}{3}$.\par 
The time extensions of rough paths on a Banach space $E$ are themselves rough paths with values in $\R \oplus E$; where the latter is a Banach space when endowed e.g.\ with the norm
\begin{equation}\label{eq: R plus E norm}
\varnorm{(t,x)} = \varnorm{t,x}_1 := \left( |t|^2 + \| x \|^2 \right)^{\frac{1}{2}}, \quad (t,x) \in \R \oplus E.
\end{equation}
Thus, to define such a time extension we also require an admissible family of tensor norms $\{\varnorm{\,\cdot\,}_n\}_{n\in \N_{>0}}$ on $\R \oplus E$. Moreover, the two families of tensor norms must be compatible in the following sense: 

\begin{assumption}\label{ass: time-extended}
Let $E$ be a real Banach space and let $\{\| \cdot \|_{n}\}_{n\in \N_{>0}}$ be an admissible family of tensor norms for $E$. Let $\R\oplus E$ be the Banach space endowed with the norm~\eqref{eq: R plus E norm} and let $\{\varnorm{\,\cdot\,}_{n}\}_{n\in \N_{>0}}$ be an admissible family of tensor norms for $E$.
Let  $j\colon (\R \oplus E) \otimes_a (\R \oplus E) \rightarrow \R \oplus E \oplus E \oplus (E \otimes_a E)$ 
be defined by $j((t,x) \otimes (s,y)) = (ts, ty, sx, x\otimes y)$ for $t,s\in \R$ and $x,y\in E$ (and extended to the whole domain by linearity). Assume that $j$ extends to a homeomorphism 
$j \colon(\R \oplus E) \otimes (\R \oplus E) \rightarrow \R \oplus E \oplus E \oplus (E \otimes E)$. 
\end{assumption}

\begin{remark}
As we are working with $\alpha > \frac{1}{3}$ we only need compatibility of $\|\cdot \|_1, \|\cdot \|_2$, and $\varnorm{\,\cdot\,}_2$, for $\alpha \leq \frac{1}{3}$ one would need analogous compatibility statements for the norms  $\|\cdot \|_1, \ldots, \|\cdot \|_{\lfloor \nicefrac{1}{\alpha}\rfloor}$, and $\varnorm{\,\cdot\,}_2,\ldots, \varnorm{\,\cdot\,}_{\lfloor \nicefrac{1}{\alpha}\rfloor}$.
\end{remark}

Unfortunately, there is no obvious canonical way to construct a compatible family of tensor norms on $\R \oplus E$ from a family of tensor norms on $E$ that is guaranteed to preserve the strongly uniform crossnorm property which is needed e.g.\ in Theorem~\ref{thm: norm-compact uat}; see Section \ref{sec: time extension1 of Banach spaces}. However, the tensor norms that are of key interest satisfy the desired compatibility condition: 

\begin{lemma}\label{lem : standard tensor norms give compatible norms}
Let $E$ be a real Banach space, and let $\R \oplus E$ be endowed with the norm~\eqref{eq: R plus E norm}. Then, if the families of tensor norms for both $E$ and $\R \oplus E$ are given by the projective [injective] tensor norm, Assumption \ref{ass: time-extended} is satisfied. Furthermore, Assumption \ref{ass: time-extended} is also satisfied if $E$ is a Hilbert space, and the families of tensor norms for both $E$ and $\R \oplus E$ are given by the Hilbert tensor norm.
\end{lemma}

See Section~\ref{sec: proof of standard tensor norms give compatible norms} for a proof. The following proposition ensures that the time-extended rough path can be constructed:

\begin{proposition}[Time-extended rough paths, see e.g.\ \cite{cuchiero_signature-based_2023}]\label{prop: time extension}
Let $\frac{1}{3}<\alpha \leq 1$ and let Assumption \ref{ass: time-extended} hold. Then for any $\mathbf{x} \in \alpharough$ (where we endow the tensor products of $E$ with $\{\|\cdot\|_n\}_{n\in \N_{>0}}$, we can define the time-extended rough path $\hat{\mathbf{x}}_{s,t} \in \mathcal{C} ^\alpha ([0,T] ; \R \oplus E )$ (where we endow the tensor products of $\R \oplus E$ with $\{\varnorm{\cdot }_n\}_{n\in \N_{>0}}$) such that for $\pi$ the natural projection (which exists by virtue of Assumption \ref{ass: time-extended})
      \[\pi: ( \R \oplus E, (\R \oplus E)^{\otimes 2})\to (E ,E^{\otimes 2}),\]
      and for any $s,t\in[0,T]$ we have
      \[
         \pi(\hat {\mathbf{x}}_{s,t}) = \mathbf{x}_{s,t},
      \] 
      and
      \[
      \inprod{\hat{\mathbf{x}}_{s,t}}{ (1, 0 )} = t - s, 
      \]
      where we interpret $(1,0)$ as an element in $(\R \oplus E)^*$ in the obvious way. Furthermore, there exists a constant $C\in (0,\infty)$ depending only on $T$ and $\alpha$ such that 
      \begin{equation}\label{eq: time-extended hoelder bound}
          \|{\mathbf{x}}_{s,t}  \|_{\mathcal{C}^\alpha}
          \leq 
          \|\hat{\mathbf{x}}_{s,t}  \|_{\mathcal{C}^\alpha} \leq C \max (\|{\mathbf{x}}_{s,t}  \|_{\mathcal{C}^\alpha},1).
      \end{equation}
\end{proposition}
\begin{proof}
    For the first order component of $\hat{\mathbf{x}}$, we set
    \[
    \hat{\mathbf{x}}_{s,t} ^{(1)}\coloneqq (t-s,\mathbf{x}_{s,t}).
    \]
     For the second component, note that by \cite[Theorem 1.16]{lyons_differential_2007}, we have that the Young integrals
    \begin{equation}\label{eq: young integrals in time-extended case}
     \int_{s}^t (u - s) dx_u \quad \text{ and } \quad \int_{s}^t (x_u - x_s) du,\quad s,t \in [0,T],
    \end{equation}
    are well defined and furthermore bounded by $C (t-s)^{1+\alpha}$, for some $C>0$ depending only on $\alpha$ and $T$. Therefore, we set
    \[
    \hat{\mathbf{x}}^{(2)}_{s,t} = \left ( \tfrac{1}{2} (t-s)^2,\int_{s}^t (u - s) dx_u, \int_{s}^t (x_u - x _s) du, \mathbf{x}^{(2)}_{s,t} \right)\quad s,t\in [0,T],
    \]
    which can be interpreted as an $(\R\oplus E)\otimes (\R\oplus E)$-valued process by Assumption~\ref{ass: time-extended}, moreover, the bound~\eqref{eq: time-extended hoelder bound} is immediate. Finally, note that the map $\hat {\mathbf{x}} \colon [0,T]^2 \to T^{(2)}(E)$ is a multiplicative functional by the additivity of the Young integral. 
\end{proof}

The following proposition ensures that the signature (see Definition~\ref{def:signature}) of a time-extended rough path uniquely identifies the path.
We provide an alternative to the proof of the analogous result for finite-dimensional paths (\cite[Lemma 2.6]{cuchiero_signature-based_2023}); indeed, for that proof to carry over to the infinite-dimensional setting one would need to assume that $\prescript{n}{j=1} \otimes_a E^*$ separates points in $E^{\otimes n}$, $n\in \N$. Our proof relies on the fact that the signature of a rough path is trivial if and only if it is tree-like~\cite{HamblyLyons:2010,boedihardjo_signature_2016}.

\begin{proposition}\label{prop: when there is a time component then unique signature}
Let $E$ be a real Banach space and let $\{\| \cdot \|_{n}\}_{n\in \N_{>0}}$ be an admissible family of tensor norms for $E$. Assume moreover that $\| \cdot \|_{m+n}$ is a reasonable crossnorm with respect to $E^{\otimes m} \otimes_a E^{\otimes n}$ for all $n,m\in \N$ (see Definition~\ref{def: reasonable crossnorm d geq 3}). Let $0< \alpha \leq 1$, $\phi \in E^*$ and let $A$ be any subset of $\alphageom$ such that for all $s,t \in [0,T]$ and $\mathbf{x}\in A$ it holds that
    \begin{equation}\label{eq: aux tree like uniqueness beta}
     \inprod{\mathbf{x}_{s,t}}{\phi} = t-s.  
    \end{equation}
    Then for $\mathbf{x},\mathbf{x}' \in A$ it holds that $S(\mathbf{x})_{0,T} = S(\mathbf{x}')_{0,T}$ if and only if $\mathbf{x} = \mathbf{x}'$.
\end{proposition}
\begin{proof}

Take any two elements $\mathbf{x},\mathbf{x}' \in A$ such that $S(\mathbf{x})_{0,T}= S(\mathbf{x}')_{0,T}$, and define 
$\mathbf{z} \in \mathcal{C}_g^{\alpha}([0,2T];E)$ by $\mathbf{z}=\pi_{\lfloor 1/\alpha\rfloor} \tilde{\mathbf{z}}$, where $\pi_{{\lfloor 1/\alpha\rfloor}} \colon T((E)) \to T^{\lfloor\frac{1}{\alpha}\rfloor}(E) $ is the natural projection and where
\begin{equation*}
\tilde{\mathbf{z}}_{0,t}
=
\begin{cases}
S(\mathbf{x})_{0,t},& t\in [0,T];\\
S(\mathbf{x})_{0,T}\hat{\otimes} S(\mathbf{x}')_{T,2T-t}, & t\in (T,2T],
\end{cases}
\end{equation*}
this completely defines $(\mathbf{z}_{s,t})_{s,t\in [0,T]}$ thanks to Chen's identity. From Remark~\ref{remark: group-like closed under multiplication} we obtain that $\mathbf{z}$ is weakly geometric, and the H\"older regularity of $\mathbf{z}$ is easily verified. By the uniqueness of the Lyons lift (see Theorem \ref{thm: Lyons lift}), one finds that similarly that $S(\mathbf{z})_{0,t} = \hat{\mathbf{z}}_{0,t}$, for all $t\in[0.2T]$.\par 
In particular, it holds that $S(\mathbf{z} )_{0,2T} = \mathbf{1}$. It follows from \cite[Theorem 1.1]{boedihardjo_signature_2016} that $\mathbf{z} $ is tree-like, i.e., there exists an $\R$-tree $\tau$ and continuous maps $\varphi\colon [0,2T]\to \tau, \psi \colon \tau \to T^{(\lfloor \nicefrac{1}{\alpha}\rfloor)}(E)$, such that 
    \begin{equation}\label{eq: aux tree like property}
    \mathbf{z}_{0,t} = (\psi \circ \varphi) (t), \quad t\in [0,2T].
    \end{equation}
The fact that $\tau$ is an $\R$-tree implies in particular (see \cite[Definition 2.1]{bestvina_-trees_2001}) that if $\varphi|_{[0,T]}$ is an embedding (i.e., injective), then $\varphi([0,T]) \subseteq \varphi([T,2T])$ (as $\{\varphi(0),\varphi(T)\} \subseteq \varphi([T,2T])$). It follows from~\eqref{eq: aux tree like uniqueness beta}, the fact that $\mathbf{z}_{0,t}=\mathbf{x}_{0,t}$ for $t\in [0,T]$, and~\eqref{eq: aux tree like property} that $\varphi|_{[0,T]}$ is an embedding. We will now prove that if $\mathbf{x}\neq \mathbf{x}'$, then $\varphi([0,T]) \not\subseteq \varphi([T,2T])$, from which we conclude that $\mathbf{x}=\mathbf{x}'$.

Indeed, if $\mathbf{x} \neq \mathbf{x}'$, then (again by Chen's identity and the fact that $\mathbf{x}_{0,T} =\mathbf{x}'_{0,T}$) there exists a $t\in(0,T)$ such that $\mathbf{x}_{0,t} \neq \mathbf{x}'_{0,t}$. Moreover, by Equation \eqref{eq: aux tree like uniqueness beta}, for all $s\in [0,T]$ with $ s\neq t$, we have that $\mathbf{x}_{0,t} \neq \mathbf{x}'_{0,s}$.
  In particular, $\mathbf{x}_{t,T} \hat{\otimes} \mathbf{x}_{T,s}' \neq \mathbf{1}$ for all $s\in [0,T]$, from which it follows that for all $s\in [T,2T]$ we have 
  \begin{equation}
    \mathbf{z}_{0,s} = \mathbf{x}_{0,t}\hat{\otimes} \mathbf{x}_{t,T} \hat{\otimes} \mathbf{x}'_{T,2T-s} \neq \mathbf{x}_{0,t} =
    \mathbf{z}_{0,t}.
  \end{equation}
In particular, 
    \begin{equation}\label{eq: aux tree thing to be contradicted}
        \varphi(t) \neq \varphi(s), \quad s\in [T,2T],
    \end{equation}
and thus $\varphi([0,T])\not\subseteq \varphi([T,2T])$.
\end{proof}

\section{Main results: universal approximation theorems for Banach space valued geometric rough paths}\label{sec: UATs}

Having provided all necessary concepts in the previous section, we are now ready to formulate our main results: universal approximation theorems for weakly geometric multiplicative functions $\mathbf{x}\in \mathcal{C}_g^{\alpha}([0,T];E)$, where $E$ is a Banach space. More specifically, in Section~\ref{sec: abstract UAT} we establish an abstract universal approximation theorem for compact subsets of $\mathcal{C}_g^{\alpha}([0,T];E)$, see Theorem~\ref{thm: general setup of uat} (here `compact' refers to a topology that is yet to be specified!). Theorem~\ref{thm: general setup of uat} is essentially a direct consequence of the Stone--Weierstrass theorem. In Section~\ref{sec: norm-compact UAT} we demonstrate how Theorem~\ref{thm: general setup of uat} gives rise to an explicit uniform approximation result on norm-compact sets in $\mathcal{C}^{\alpha}_g([0,T];E)$, see Theorem~\ref{thm: norm-compact uat} and Corollary~\ref{cor: norm-compact uat}. In Section~\ref{sec: norm-bounded UAT} we again employ Theorem~\ref{thm: general setup of uat} to obtain a universal approximation theorem on $\mathcal{C}^{\alpha}$ norm bounded sets (these sets are compact with respect to a weak$^*$-like topology), see Theorem~\ref{thm: weak star version uat} and Corollary~\ref{cor: weak star version uat but more applied} below. Note that while the proof of Theorems~\ref{thm: general setup of uat} and~\ref{thm: norm-compact uat} are relatively straight-forward, the proof of Theorem~\ref{thm: weak star version uat} requires all the results presented in Sections~\ref{sec: predual of holder}--\ref{sec: i star continuity lyons lift}.

\subsection{An abstract UAT for geometric rough paths based on Stone--Weierstrass}\label{sec: abstract UAT}

We begin by recalling the version of the Stone--Weierstrass theorem that is relevant to our setting. Note that given a topological space $(X,\tau_X)$, we denote the space of continuous functions from $X$ to $\R$ by $C(X;\R)$.

\begin{definition}\label{def:point-sep}[Point separating]
    Let $(X,\tau_X)$ be a topological space and let $\mathcal{A} \subset C(X ; \R)$. Then, $\mathcal{A}$ is called \emph{point separating} on $X$ when for every $x,y \in X $ there exists $a\in \mathcal{A}$ with $a(x) \neq a(y)$.
\end{definition}
 \begin{definition}\label{def:non-van}[Non-vanishing]
     Let $(X,\tau_X)$ be a topological space and let $\mathcal{A} \subset C(X ; \R)$. Then $\mathcal{A}$ is called \emph{non-vanishing} if for every $x\in X$ there exists $a\in\mathcal{A}$ with $a(x) \neq 0$.
 \end{definition}

\begin{theorem}[Stone--Weierstrass, see Theorem 5 in \cite{stone_generalized_1948}, and see also Theorem 3.2 in \cite{cuchiero_global_2023}]\label{thm: Stone--Weierstrass}
    Let $(X,\tau_X)$ be a compact Hausdorff space, and let $\mathcal{A}$ be a subalgebra of $C(X;\R)$. Then $\mathcal{A}$ is dense in $C(X;\tau_X)$ (with respect to uniform convergence) if and only if $\mathcal{A}$ is point separating and non-vanishing. 
\end{theorem}

The Stone--Weierstrass allows to provide a general format for universal approximation theorems on rough paths with signatures. More specifically, we take $X$ to be some subset of $\alphageom$ and assume that the algebra consists of linear functions on the signature; i.e., we assume there exists a $D\subseteq T_a(E^*)$ such that the algebra $\mathcal{A}$ is given by\begin{equation}\label{eq: standard version of linear functionals}
\mathcal{A} = \left\{ \alphageom\ni\mathbf{x} \mapsto \inprod{S(\mathbf{x})_{0,T}}{l} \colon l\in D \right\}
\end{equation} 

As $X\subseteq \alphageom$, by Theorem \ref{thm: preservation of weakly geometric under lyons lift} combined with Proposition \ref{prop: equality of group like definitions} (see also Definitions \ref{def: weak group like elements dual rep} and \ref{def: weakly geometric rough paths}) we have
\[
\inprod{S(\mathbf{x})_{0,T}}{l} \inprod{S(\mathbf{x})_{0,T}}{l'} = \inprod{S(\mathbf{x})_{0,T}}{l \shuffle l'}.
\]
Thus, $\mathcal{A}$ is indeed an algebra provided $D$ is a subspace that is closed under the shuffle product. We arrive at the following abstract uniform approximation theorem:

\begin{theorem}\label{thm: general setup of uat}
Let $E$ be a real Banach space and let $\{\| \cdot \|_{n}\}_{n\in \N_{>0}}$ be an admissible family of tensor norms on $E$. Let $\tau$ be a Hausdorff topology on $\alphageom$ and let $K$ be a $\tau$-compact set such that $S(\mathbf{x})_{0,T} \neq S(\mathbf{x}')_{0,T}$ whenever $\mathbf{x}\neq \mathbf{x}'$, with $\mathbf{x},\mathbf{x}'\in K$. Furthermore, let $D$ be a subspace of $T_a(E^*)$ 
satisfying the following properties:
\begin{enumerate}
    \item $\mathbf{1} \in D$,
    \item\label{item: general set up point separtion} $D$ separates points in $T((E))$,
    \item $D$ is closed under the shuffle product, and 
    \item\label{it:continuous signature} $K\ni\mathbf{x} \mapsto \inprod{S(\mathbf{x})_{0,T}}{l}$ is a $\tau$-continuous map for any $l\in D$.
\end{enumerate}
Then for any $\tau$-continuous function $f\colon \alphageom \to \R$ and for all $\epsilon>0$, there exists $l \in D$ such that \[
    \sup_{\mathbf{x}\in K} | \inprod{S(\mathbf{x})_{0,T}}{l} - f(\mathbf{x}) | < \epsilon.
    \]
\end{theorem}
\begin{proof}
    This follows from the Stone--Weierstrass theorem. Indeed, we take $K$ as the compact set and $\tau|_K$ as the topology, which is Hausdorff as $\tau$ is Hausdorff. The algebra $\mathcal{A}$ is given by~\eqref{eq: standard version of linear functionals}, note that assumption~\eqref{it:continuous signature} above implies that $\mathcal{A}\subseteq C(K;\R)$, and the discussion preceding Theorem~\ref{thm: general setup of uat} implies that $\mathcal{A}$ is indeed an algebra. The assumption that $D$ is point separating on $T((E))$ and that each element of $K$ has a unique signature implies that $\mathcal{A}$ is point separating on $K$. The non-vanishing property is guaranteed by the fact that $\mathbf{1} \in D$; note that $\inprod{S(\mathbf{x})_{0,T}}{\mathbf{1}} = 1$ for all $\mathbf{x} \in \alphageom$. Thus the assertion indeed follows from~\ref{thm: Stone--Weierstrass}.
\end{proof}
\begin{remark}
    In this work the only linear functionals on the signature that are considered come from a subset of $T_a(E^*)$, where $E$ is the underlying Banach space. In theory one could also consider general functionals in $T(E)^*$, for which the separating points of $T((E))$ (see Item \ref{item: general set up point separtion} of Theorem \ref{thm: general setup of uat}) may be easier to prove. However, general elements of $T(E)^*$ are not so easily identified/parametrized; in applications it is preferable to work only with elements of $T_a(E^*)$. 
\end{remark}

\subsection{A UAT for norm-compact sets of geometric rough paths}\label{sec: norm-compact UAT}
Our first application of the abstract UAT proven in the previous section involves taking the norm topology on $\alphageom$, i.e., the topology arising from the metric $\varrho^{\text{hom}}_{\alpha}$ introduced in Section~\ref{sec: holder norms and rough paths}. This is in line with the approach typically taken in the finite-dimensional setting, see e.g.~\cite{cuchiero_signature-based_2023,Cuchiero_Primavera_2025,lyons_non-parametric_2020}.\par 
Regarding the assumptions in the theorem below: for the definition of an admissible family of tensor norms see Definition~\ref{def: admissible crossnorms}, for the definition of strongly uniform crossnorms see Definition~\ref{def:strong-unif-crossnorm}. Recall from Propositions~\ref{prop: properties of projective and injective} and~\ref{prop: properties hilbert norm} that the injective, projective, and Hilbert tensor norms each provide admissible family of tensor norms that are strongly uniform crossnorms. The assumption that the signature $S$ separates points in $K$ is satisfied when one considers time-extended paths, see Section~\ref{ssec: time-extended rough paths}. The algebraic tensor algebra $T_a(G)$ is defined in Definition~\ref{def: algebraic tensor algebra}.

\begin{theorem}\label{thm: norm-compact uat}
Let $E$ be a real Banach space satisfying the approximation property and let $\{\| \cdot \|_{n}\}_{n\in \N_{>0}}$ be an admissible family of tensor norms on $E$ that are moreover strongly uniform crossnorms. Let $G$ be a norm-dense subspace of $E^*$. 
Let $ 0 < \alpha \leq 1 $, $T>0$, let $K\subseteq \alphageom$ be $\varrho^{\text{hom}}_{\alpha}$-compact, and assume that $S(\mathbf{x})_{0,T} \neq S(\mathbf{x}')_{0,T}$ whenever $\mathbf{x}\neq \mathbf{x}'$, $\mathbf{x},\mathbf{x}'\in K$. Then for any $\varrho^{\text{hom}}_{\alpha}$ continuous function $f\colon \alphageom \to \R$ and for all $\epsilon>0$, there exists $l \in T_a(G)$ such that \[
    \sup_{\mathbf{x}\in K} | \inprod{S(\mathbf{x})_{0,T}}{l} - f(\mathbf{x}) | < \epsilon.
    \]
\end{theorem}
\begin{proof}
    This is a direct application of Theorem \ref{thm: general setup of uat} with $\tau$ the topology generated by $\varrho^{\text{hom}}_{\alpha}$ and $D = T_a(G)$. The assumption that $E$ has the approximation property and that the family of tensor norms is strongly uniform is needed to establish that $T_a(G)$ separates points in $T((E)))$, indeed, this follows by combining Proposition \ref{prop: the dual tensor algebra can actually separate points} and Lemma \ref{lem : point separating dense is enough} with Proposition \ref{prop: density in tensor space higher orders}. The set $T_a(G)$ is closed under the shuffle product and contains $\mathbf{1}$ by definition (see Definition~\ref{def: shuffle}). Moreover, Theorem \ref{thm: Lyons lift} (and the fact that we have an admissible family of tensor norms) implies that functions of the form 
    \[
    \alphageom\ni\mathbf{x} \mapsto \inprod{S(\mathbf{x})_{0,T}}{l}
    \]
    are continuous for all $l\in T_a(G)$, indeed, for all $l\in T_a(G)$ there exists an $n\in \N$ such that $l\in T_a^{(n)}(G)$ and thus 
$\inprod{S(\cdot)_{0,T}}{l} =  \inprod{S^n(\cdot)_{0,T}}{l}$. Thus all conditions of Theorem~\ref{thm: general setup of uat} are satisfied and the assertion follows.
    \end{proof}

\begin{remark}\label{rem: isnt it better to take a norm dense subset?}
There is a practical advantage in considering a norm-dense set $G$ in Theorem~\ref{thm: norm-compact uat} instead of simply taking $G=E^*$. Indeed, if $E$ is a separable Hilbert space and $\{ e_i\}_{i\in\N_{>0}}$ is an orthonormal basis for $E$, we can simply take $G$ to be the linear span of this basis. As a consequence, all our linear functionals are finite linear combinations of objects of the form $ e_{i_1} \otimes \ldots \otimes e_{i_n}$, $n\in \N_{>0}$. 
\end{remark}

Although $\varrho^{\text{hom}}_{\alpha}$-compact subsets of $\alphageom$ are not easily identified explicitly, tightness of a Borel measure on a Polish space allows us to obtain the following approximation statement for stochastic rough paths:

\begin{corollary}\label{cor: norm-compact uat} Let $E$ be a real Banach space satisfying the approximation property and let $\{\| \cdot \|_{n}\}_{n\in \N_{>0}}$ be an admissible family of tensor norms on $E$ that are moreover strongly uniform crossnorms and let $G$ be a norm-dense subspace of $E^*$.
Assume moreover that there exists a $\phi \in E^*$ and a (Borel) measurable set $A\subseteq \alphageom$ such that 
\begin{equation}\label{eq:time_extended process}
\forall\,\mathbf{x}\in A,\, \forall\, 0\leq s\leq t  \leq T\colon \quad \langle \mathbf{x}_{s,t}, \phi \rangle = t-s.
\end{equation}
Let $ 0 < \alpha \leq 1 $, $T>0$.
Let $(\Omega,\mathcal{F},\mathbb{P})$ be a probability space and let $\mathbf{X}\colon \Omega \rightarrow \alphageom$ be measurable and separably valued and assume moreover that $\mathbb{P}(\mathbf{X}\in A)=1$. 
Then for any $\varrho^{\textnormal{hom}}_{\alpha}$ continuous function $f\colon \alphageom \to \R$ and for all $\epsilon>0$, there exists an $l \in T_a(G)$ such that \[
   \mathbb{P}( |\inprod{S(\mathbf{X})_{0,T}}{l} - f(\mathbf{X}) | < \epsilon) \geq  1-\epsilon.
    \]
\end{corollary}

\begin{proof} 
Without loss of generality we can assume that $\mathbf{X}\in A$ for \emph{all} $\omega \in \Omega$. 
By assumption~\eqref{eq:time_extended process} and Proposition~\ref{prop: when there is a time component then unique signature} we have $S(\mathbf{x})\neq S(\mathbf{x}')$ whenever $\mathbf{x}\neq \mathbf{x}'$, $\mathbf{x},\mathbf{x}'\in A$. 
As $\mathbf{X}$ takes values in a separable subset of $\alphageom$ we can assume it takes values in a complete separable metric space. By Prohorov's Theorem (see \cite[Theorem 8.6.2]{bogachev_measure_2007}) this implies that the induced measure is tight. In other words, for all $\epsilon>0$ there exists a compact set $K_{\epsilon}\subseteq \alphageom$ such that $\mathbb{P}(\mathbf{X}\in K_{\epsilon})\geq 1-\epsilon$. One then applies Theorem~\ref{thm: norm-compact uat} to this set $K_{\epsilon}$.
\end{proof}

\begin{remark}
Recall that the time-extension introduced in Section~\ref{ssec: time-extended rough paths} provides an approach to construct paths that satisfy assumption~\eqref{eq:time_extended process}.
Regarding the assumption that $\mathbf{X}$ is separably valued: one approach could be to work in the context of \emph{geometric rough paths}, i.e., paths that can be approximated in the $\varrho^{\text{hom}}_{\alpha}$ metric by lifted differentiable paths, which is a separable space. It was proven in~\cite{grong_geometric_2022} that when $E$ is a Hilbert space, then $\alphageom$ embeds continuously in the space $\alpha'$-H\"older continuous geometric rough paths provided $\frac{1}{3} < \alpha'<\alpha$. Alternatively, one may be able to exploit the fact that the H\"older space $C^{\alpha}_0([0,T];E)$ is a separable subset of $C^{\alpha'}_0([0,T];E)$ whenever $\alpha>\alpha'$ and $E$ is separable.

\end{remark}

\begin{remark}
In Theorem~\ref{thm: norm-compact uat} we assumed that every element of $K$ has a unique signature. Alternatively, one can consider the quotient topology on $K$, see for example \cite[Section 1.4]{cass_lecture_2024}, where the equivalence classes are rough paths with the same signature. Functions that are continuous in the chosen topology on $K$ will give rise to continuous functions in the quotient topology as long as the functions map all elements of an equivalence class to the same object. 
\end{remark}

\subsection{A UAT for norm-bounded geometric rough paths}\label{sec: norm-bounded UAT}
The main disadvantage of the UAT for norm-compact sets, i.e., of Theorem~\ref{thm: norm-compact uat} above, is that it is not easy to identify norm-compact sets explicitly. In particular, Theorem~\ref{thm: norm-compact uat} typically cannot be applied `universally'; say to solutions of a whole class of SPDEs. To overcome this problem, we consider a weak$^*$-type topology on $\alphageom$ instead of the topology induced by $\varrho^{\textnormal{hom}}_{\alpha}$ and use the Banach-Alaoglu theorem to establish a UAT that applies to norm \emph{bounded} sets, see Theorem~\ref{thm: weak star version uat} below. However, there is a minor difficulty: due to the algebraic structure that is imposed on multiplicative functionals, $\mathcal{C}^{\alpha,n}([0,T];E)$ is \emph{not} a Banach space. However, $\mathcal{C}^{\alpha,n}([0,T];E)$ can be interpreted as a subset of the H\"older space $C_0^{\alpha}([0,T];T^{(n)}(E))$ (see Proposition~\ref{prop: inclusion of alpha rough into alpha holder} below). Thus, we can consider the topology that $\alphageom$ inherits from $C_0^{\alpha}([0,T];T^{(n)}(E))$, see Definition~\ref{def: i* star topology} below.

For this approach to work, we need $E$ to have a predual $F$, and we need tensor norms on the algebraic tensor space $F^{\otimes_a n}$ ($n\in \N_{>0}$) that align with the tensor norms on $E^{\otimes_a n}$. More specifically, we assume the following:
\begin{setting}\label{assumptions}
   Let $E$ be a real Banach space with predual $F$ and let $\{\| \cdot \|_{n}\}_{n\in \N_{>0}}$ and \mbox{$\{\| \cdot \|^*_{n}\}_{n\in \N_{>0}}$} be admissible families of tensor norms on respectively $E$ and $F$, such that for each $n \in \N_{>0}$, $F^{\otimes n}$ is the predual of $E^{\otimes n}$. 
\end{setting}
We identify two examples which fall under Setting~\ref{assumptions}:
 \begin{example}\label{ex:2}
     Let $E$ be a real Hilbert space, so that by the Riesz representation theorem we can identify $E^*$ with $E$.  Let $\|\cdot \|_n$ and $\| \cdot \|^*_n$ ($n\in \N_{>0}$) both be the Hilbert tensor norms (see Section~\ref{ssec: Hilbert tensors}). Proposition \ref{prop: properties hilbert norm} ensures that these tensor norms form admissible families. Moreover, as $E^{\otimes n}$ is a Hilbert space, we can identify it with $(E^{\otimes n})^*$. 
\end{example}
 \begin{example}\label{example: projective tensor norm Radon--Nikodym}\label{ex:1}
     Let $E$ be a real Banach space with predual $F$, and assume $E$ has the approximation property and the Radon--Nikodým property (see respectively Definitions \ref{def: approximation property} and \ref{def: Radon-nikodym}). Let $\| \cdot \|_n$ (the norm on $E^{\otimes_a n}$, $n\in \N_{>0}$) be the projective tensor norm, and let $\| \cdot \|^*_n$ (the norm on $F^{\otimes_a n}$, $n\in \N_{>0}$) be the injective tensor norm (see Section~\ref{ssec: projective and injective}). By Proposition \ref{prop: properties of projective and injective} these families of tensor norms are admissible. As $E$ is Radon--Nikod\'ym and satisfies the approximation property, Proposition \ref{prop: Radon--Nikodym} implies that $F^{\otimes n}$ is the predual of $E^{\otimes n}$, $n\in \N$. Recall that if $E$ is separable, then it has the Radon--Nikod\'ym property, see Proposition \ref{prop: when Radon--Niko}.
 \end{example}

\begin{proposition}\label{prop: inclusion of alpha rough into alpha holder}
Let $E$ be a real Banach space, let $\{\| \cdot \|_{n}\}_{n\in \N_{>0}}$ be an admissible family of tensor norms for $E$, and let $0< \alpha \leq 1$. Then there exists an (injective) inclusion map\[
   i\colon \mathcal{C}^{\alpha,n} ([0,T];E ) \hookrightarrow  C_0^{\alpha}([0,T];T^{(n)}(E)).
    \]
    Furthermore, there exists $C>0$ such that 
    \[
   \|i(\mathbf{x})\|_{C^\alpha} \leq C  \big(\|\mathbf{x}\|_{\mathcal{C^\alpha}} + 1\big)
    \]
\end{proposition}
The proof is postponed to Section~\ref{sec: weak star topology on alpha rough}. Note that Setting~\ref{assumptions} and Theorem \ref{thm: duality of gen holder space} below imply that $C^{\alpha}_0([0,T];T^{(n)}(E))$ has a predual. As we have just established that $\mathcal{C}^{\alpha,n}([0,T];E)$ can be viewed as a subset of $C_0^{\alpha}([0,T];T^{(n)}(E)) $, we obtain that the weak-$^*$ topology on $C_0^{\alpha}([0,T];T^{(n)}(E))$ induces a topology on $\mathcal{C}^{\alpha,n} ([0,T];E )$ that we refer to as the $i^*$-topology:

\begin{definition}[$i^*$ topology]\label{def: i* star topology}
    Assume Setting \ref{assumptions}, let $0 < \alpha \leq 1$, and let $n\in \N_{>0}$. Then, with the inclusion map $i$ defined in Proposition \ref{prop: inclusion of alpha rough into alpha holder} we interpret $ \mathcal{C}^{\alpha,n} ([0,T];E )$ as a subspace of $C_0^{\alpha}([0,T];T^{(n)}(E))$ and we define the $i^*$ topology on $ \mathcal{C}^{\alpha,n} ([0,T];E )$ as the subspace topology with respect to the weak-$^*$ topology on $C_0^{\alpha}([0,T];T^{(n)}(E))$. Equivalently, we define the $i^*$ topology to be the initial topology generated by all the functions $m \circ i$, where $m$ is in the predual of $C_0^{\alpha}([0,T];T^{(n)}(E))$. 
\end{definition}

The key advantage of the $i^*$ topology is that $\BR$ is compact in this topology. However, there is a pay-off: as the $i^*$ topology is weaker than the norm topology, the set of $i^*$-continuous functions is smaller than the set of $\varrho^{\textnormal{hom}}_{\alpha}$-continuous functions. For example, the mapping 
\begin{equation}\label{eq: function that is norm cont but not i continuous}
\mathbf{x}\mapsto \|\mathbf{x}^{(k)}\|_{0,T},
\end{equation}
is not $i^*$ continuous for any $1 \leq k \leq n$. Remarks \ref{rem: nets instead of sequences} and \ref{rem: weak star topology on alpha rough} below provide key insights into identifying $i^*$-continuous functions. 

We first give a more general version of the UAT in this topology, followed by corollaries that are applicable in the time-extended setting provided by Section~\ref{ssec: time-extended rough paths}.

\begin{theorem}\label{thm: weak star version uat}
    Assume Setting \ref{assumptions}, let $0<\alpha\leq 1$, $T,R>0$, and let $G$ be a norm dense subset of $F$. Let $A$ be an $i^*$ closed subset of $\BRg$ such that $S(\mathbf{x}) \neq S(\mathbf{x}')$ whenever  $\mathbf{x} \neq \mathbf{x}'$, $\mathbf{x},\mathbf{x}'\in A$. Then for every $i^*$-continuous function $f\colon A \to \R$ and every $\epsilon>0$, there exists $l\in T_a(G)$ such that
    \[
    \sup_{\mathbf{x}\in A} | \inprod{S(\mathbf{x})_{0,T}}{l} - f(\mathbf{x}) | < \epsilon.
    \]
    
\end{theorem}
\begin{proof}
    This proof is an application of Theorem \ref{thm: general setup of uat}. Indeed, we first note that as $F^{\otimes n}$ is the predual of $E^{\otimes n}$, it separates points in $E^{\otimes n}$, $n\in \N_{>0}$. Thus the truncated algebraic tensor algebra $T^{(n)}(F)$ separates points in the truncated algebraic tensor algebra $T^{(n)}(E)$, $n\in \N_{>0}$. As we are dealing with an admissible family of tensor norms, Proposition \ref{prop: density in tensor space higher orders} implies that $T_a^{(n)}(G)$ is norm-dense in $T^{(n)}(F)$. Lemma \ref{lem : point separating dense is enough} thus implies that $T_a^{(n)}(G)$ separates points in $T^{(n)}(E)$. The set $T_a(G)$ is closed under the shuffle product and contains $\mathbf{1}$ by definition.
    
    The $i^*$ topology is Hausdorff as the inclusion map defined in Proposition $\ref{prop: inclusion of alpha rough into alpha holder}$ is injective. Proposition \ref{prop: compactness of rough alpha ball weak} implies that $\BR$ is $i^*$-compact, and given that $A$ is an $i^*$ closed subset of this compact set we conclude that $A$ is $i^*$-compact. Proposition \ref{prop: continuity of Lyons lift} implies that the linear functionals are continuous, indeed, for all $l\in T_a(G)$ there exists an $n\in \N$ such that $l\in T_a^{(n)}(G)$ and thus 
$\inprod{S(\cdot)_{0,T}}{l} =  \inprod{S^n(\cdot)_{0,T}}{l}$.
    Thus the assumptions of Theorem~\ref{thm: general setup of uat} are satisfied and the assertion follows.
\end{proof}

By combining Theorem~\ref{thm: weak star version uat} with Proposition~\ref{prop: when there is a time component then unique signature} one obtains the following corollary. The time-extended paths introduced in Section~\eqref{ssec: time-extended rough paths} provide a way to ensure that our objects of interest lie in the set $A$. 

\begin{corollary}\label{cor: weak star version uat but more applied}
Assume Setting \ref{assumptions}, let $0<\alpha\leq 1$, $T,R>0$, and let $G$ be a norm dense subset of $F$. Let $\phi\in F$ and consider
    \[
    A \coloneqq \left\{ \mathbf{x} \in \alphageom \, \,\,\big | \,\,\,  \inprod{\mathbf{x_{s,t}}}{ \phi} =  t-s \text{  and } \|\mathbf{x}\|_{\mathcal{C}^\alpha} \leq R\right\}.
    \]
    Then for any function $f\colon A \to \R$ that is continuous with respect to the $i^*$ topology and any $\epsilon>0$, there exists $l\in T_a(G)$ such that
    \[
    \sup_{\mathbf{x}\in A} | \inprod{S(\mathbf{x})_{0,T}}{l} - f(\mathbf{x}) | < \epsilon.
    \]
\end{corollary}

\begin{proof}
    By Theorem \ref{thm: weak star version uat} and Proposition~\ref{prop: when there is a time component then unique signature}, this just comes down to showing that the set $A$ we defined is $i^*$ closed. To show this we take any net $\{ \mathbf{x}_\lambda\}_{\lambda\in I}$ in $A$ converging to some $\mathbf{x}$ in $\BRg$. Denoting $t\mapsto \mathbf{x}_{0,t}$ as $t\mapsto \mathbf{x}(t)$, we can use the characterization of convergence given in Remark \ref{rem: weak star topology on alpha rough} to conclude that $A$ is indeed closed. 
\end{proof}

We obtain the following corollary for sets of stochastic processes in $\alphageom$ that are bounded in expectation. This is to be compared with Corollary~\ref{cor: norm-compact uat}, which applies to a single process (and extends only to tight sets of processes).

\begin{corollary}\label{cor: weak star version uat for stochastic processes}
Assume Setting \ref{assumptions}, let $0<\alpha\leq 1$, $T>0$, and let $G$ be a norm dense subset of $F$. Let $\phi\in F$ and consider
    \[
    A \coloneqq \left\{ \mathbf{x} \in \alphageom \, \,\,\big | \,\,\,  \inprod{\mathbf{x_{s,t}}}{ \phi} =  t-s \text{  and } \|\mathbf{x}\|_{\mathcal{C}^\alpha} \leq R\right\}.
    \]
Let $(\Omega,\mathcal{F},\mathbb{P})$ be a probability space, let $J$ be an index set, and for all $j\in J$ let $\mathbf{X}_j\colon \Omega \rightarrow \alphageom$ be measurable and assume  
\begin{align}
\forall\, j\in J\colon \mathbb{P}(\mathbf{X}_j\in A) = 1 &&\text{ and }&& \sup_{j\in J} \mathbb{E} \| \mathbf{X}_j \|_{\mathcal{C}^{\alpha}} =: M< \infty.
\end{align}
Then for any function $f\colon \alphageom \to \R$ such that $f|_{\BR}$ is $i^*$ continuous for all $R>0$, and any $\epsilon>0$, there exists $l\in T_a(G)$ such that
    \[
    \sup_{j\in J} \mathbb{P} ( | \inprod{S(\mathbf{X}_j)_{0,T}}{l} - f(\mathbf{X}_j) | > \epsilon ) < 1- \epsilon.
    \]
\end{corollary}
\begin{proof}
By Markov's inequality we have $\sup_{j\in J} \mathbb{P}( \|\mathbf{X}_j \|_{\mathcal{C}^{\alpha}} > \nicefrac{M}{\epsilon}) \leq \epsilon $. By applying Corollary~\ref{cor: weak star version uat but more applied} with $R=\nicefrac{M}{\epsilon}$ we arrive at the desired result.
\end{proof}

\begin{remark}\label{rem: nets instead of sequences}
As the weak-$^*$ topology is generally not metrizable, one typically needs to consider nets rather than sequences when establishing $i^*$ continuity. However, recall that if $X$ is a separable Banach space, then the weak-$^*$ topology on $X^*$ is metrizable on norm-bounded sets. By Remark~\ref{rem: AE is separable} we have that the predual of $C_0^{\alpha}([0,T],F)$ is separable whenever $F$ is separable; in this case the $i^*$ topology is metrizable on norm-bounded sets. 
\end{remark}

\section{The Banach space \texorpdfstring{$C_0^\alpha([0,T];E)$}{of H\"older continuous functions} and its predual}\label{sec: predual of holder}
Recall that $C_0^{\alpha}([0,T];E)$ (with $E$ a Banach space) is the space of $\alpha$-H\"older continuous $E$-valued functions starting in $0$, see Section~\ref{sec: holder norms and rough paths}. From e.g.~\cite[Proposition 2.3]{weaver_lipschitz_2018}, the proof of which also applies to the space $C_0^\alpha([0,T];E)$), we know that $C_0^\alpha([0,T];E)$ endowed with the $\alpha$-H\"older coefficient $\|\cdot \|_{C^{\alpha}}$ (see~\eqref{eq: alpha holder continuity}) is a Banach space whenever $E$ is a Banach space. In this section, we provide an explicit representation of the predual of $C_0^\alpha([0,T];E)$ in the case that $E$ has a predual; the space we construct is known as the Arens--Eells space and the construction is directly adapted from \cite[Section 3.1]{weaver_lipschitz_2018}, where it is provided for $C_0^\alpha([0,T];\R)$. Throughout this whole section, $E$ is a real Banach space that allows for a predual that is denoted by $F$.

As $C_0^\alpha([0,T];E)$ is a space of functions from the interval $[0,T]$ to $E$, it is to be expected that its predual can be represented by some space of functions from $[0,T]$ to $F$. The first step in the construction is identifying some `elementary' functions. We then proceed to define a norm on the elementary functions that pairs nicely with the $\alpha$-Hölder norm; the desired representation of the predual is then obtained by taking the completion of the set of elementary functions with respect to this norm. 

\begin{definition}\label{def: molecules}
    Let $F$ be a Banach space, and let $T> 0$. Then, a \emph{molecule} $m\colon [0,T] \to F$ is a function with finite support such that
    \[
    \sum_{t_i\in \supp m} m(t_i) = 0.
    \]
    The space of all molecules is denoted by $M([0,T];F)$. 
    For $t,s\in [0,T]$ and $y\in F$, we define $m_{t,s}^y \in M([0,T];F)$ as
    \begin{equation}\label{eq: explicit molecule definition}
    m_{t,s}^y = (\mathbf{1}_{\{t\}} -  \mathbf{1}_{\{s\}})y.
    \end{equation}
    The set of \emph{elementary molecules} is given by
    \begin{equation}\label{eq: def EM}
        EM([0,T];F)=\{ m_{t,s}^y \colon s,t\in [0,T],\, y\in F,\, \| y \|^* = 1\}.
    \end{equation}
\end{definition}
Any molecule can be written, although not uniquely, as a linear combination of elementary molecules, indeed, for all $m\in M([0,T];F)$ we have 
\[
m = \sum_{t_i\in \supp{m}}  m^{m(t_i)}_{t_i,0} =  \sum_{t_i\in \supp{m}} \|m(t_i)\|\,  m^{ \|m(t_i)\|^{-1}m(t_i)}_{t_i,0}.
\]
Inserting any $t \in (0,T]$ in the right hand-side gives $m(t)$ immediately. Inserting $t=0$ on the right-hand side gives $-\sum_{t_i\in (\supp{m}\setminus\{0\})}   m(t_i) $, which is equal to $m(0)$ by $\sum_{t_i \in \supp{m}} m(t) = 0$. 

Before we define the norm on $M([0,T];F)$, we should keep in the mind that $C_0^\alpha([0,T];E)$ should be the dual. A natural way for a $x\in C_0^\alpha([0,T];E)$ to act on a molecule $m$ would be
\begin{equation}\label{eq: def of how holder acts on ae}
\inprod{x}{m} \coloneqq \sum _{t\in [0,T]} \inprod{x(t)}{m(t)}.
\end{equation}
Note that if $m=\sum_{i=1}^{n} a_i m_{t_i,s_i}^{y_i}$, with $m_{t_i,s_i}^{y_i} \in EM([0,T];F)$, then
\begin{equation}\label{eq: Arens-Eells inequality}
|\inprod{x}{m}| \leq \sum_{i=1}^{n} |a_i\inprod{x(t_i) - x(s_i)}{y_i}| \leq \sum_{i=1}^{n} |a_i| \|x(t_i) - x(s_i)\| \leq \|x\|_{C^\alpha}\sum_{i=1}^{n} |a_i||t_i-s_i|^\alpha.
\end{equation}
This means that if we want the $\alpha$-Hölder norm to correspond to the dual norm, we should define the following norm.
\begin{propdef}\label{def: Arens--Eells}
    Let $F$ be a Banach space, let $T>0$ and let $0< \alpha \leq 1$. Then the mapping $\|\cdot\|_{\textnormal{\AE},\alpha} \colon M([0,T];F)\rightarrow [0,\infty)$ given by 
\begin{equation*}
\begin{aligned}
    & \|m\|_{\textnormal{\AE},\alpha} \\
    & = \inf \left \{ \sum_{i=1}^n |a_i| |t_i-s_i|^\alpha \, \, \, \bigg | \,\,\, m = \sum_{i=1}^n a_i m_{t_i,s_i}^{y_i} \text{ and } m_{t_i,s_i}^{y_i}\in EM([0,T];F) \text{ for all } 1\leq i \leq n \right \}
\end{aligned}
\end{equation*}
defines a norm. This norm is called the \emph{Arens--Eells norm} and the Banach space $\textnormal{\AE}_{\alpha}([0,T];F)$ obtained by taking the completion of $M([0,T];F)$ under this norm is called the \emph{Arens--Eells space}. 
\end{propdef}
\begin{proof}
    Firstly, the Arens--Eells norm is well defined, because any molecule can be written as linear combination of elementary molecules. The triangle inequality and homogeneity follow directly from the definition, and so we only have to show that $\|m\|_{\textnormal{\AE},\alpha} \neq 0$ when $m\neq0$.
    
    Take $m\in M([0,T];F)$ such that $m\neq 0$. Equation \eqref{eq: Arens-Eells inequality} implies that for any $x\in C_0^\alpha([0,T];E)$
    \[
    |\inprod{x}{m}| \leq \|x\|_{C^\alpha} \|m\|_{\textnormal{\AE},\alpha}.
    \]
    Therefore, it suffices to find $x\in C_0^\alpha([0,T];E)$ such that $\inprod{x}{m}\neq 0$. The molecule $m$ has finite support, and so there exists open interval $(r,s)\subset [0,T]$ containing exactly one element $t$ of the support. Let $g\in E$ be such that $\inprod{g}{m(t)} \neq 0$, and let $f\in C^{\infty}([0,T];\R)$ be such that $f(t)=1$ and $\supp(f)\subseteq (r,s)$. Then $\inprod{gf}{m} = \langle g,m(t)\rangle \neq 0$, and so we are done.
\end{proof}

\begin{remark}\label{rem: AE is separable}
Note that the space $\textnormal{\AE}_{\alpha}([0,T];F)$ is separable whenever $F$ is separable. Indeed, suppose $F$ is separable and let $A\subseteq F$ be countable and dense in $F$. Let $D=\{ k2^{-j} T \colon j\in \N,\, 1\leq k \leq 2^j \}$. Consider the set 
\begin{equation}
    \mathcal{D} = \left\{ \sum_{k=1}^{n} x_k \mathbf{1}_{\{s_k\}} \colon n\in \N,\, x_1,\ldots,x_n\in A,\, s_1,\ldots,s_n\in D,\, \sum_{k=1}^{n} x_n =0\right\}.
\end{equation}
Clearly $\mathcal{D}$ is countable; we claim moreover that $\mathcal{D}$ is dense in $M([0,T];F)$. By definition of the Arens--Eells space this implies that $\mathcal{D}$ is dense in $\textnormal{\AE}_{\alpha}([0,T];F)$. To verify the claim, let $m\in M$; $m=\sum_{k=1}^{n} y_k \mathbf{1}_{\{t_k\}}$ for some $y_1,\ldots,y_n\in F$ and $t_1,\ldots,t_n\in [0,T]$. Let $\epsilon>0$ be given and pick $x_1,\ldots,x_{n-1}\in D$ such that $\| x_i - y_i \| \leq \nicefrac{\epsilon}{2n}$ for all $1\leq i\leq n-1$. Set $x_n = - \sum_{i=1}^{n-1} x_i$. Let $s_1,\ldots,s_n$ be such that $|s_i - t_i|\leq (2\sum_{i=1}^{n} \| y_i \|)^{-1} \epsilon$. Finally, define $m_{\epsilon} \in \mathcal{D}$ by $m_{\epsilon} = \sum_{k=1}^{n} x_i \mathbf{1}_{\{s_i\}}$. By observing that 
\begin{equation*}
    m - m_{\epsilon} =
\sum_{k=1}^{n-1} (x_i - y_i) (\mathbf{1}_{\{t_i\}} - \mathbf{1}_{\{t_n\}}) +
\sum_{k=1}^{n} y_i (\mathbf{1}_{\{t_i\}} - \mathbf{1}_{\{s_i\}})
\end{equation*}
one readily verifies that $\| m - m_{\epsilon}\|_{\textnormal{\AE},\alpha} \leq \epsilon (T^{\alpha} + 1)$.
\end{remark}
Now that we have defined the Arens--Eells space, we want to show that its dual is $C_0^\alpha([0,T];E)$. We are also interested in understanding the weak$^*$ topology on $C_0^\alpha([0,T];E)$ generated by the Arens--Eells space. It turns out that on norm bounded subsets, this corresponds to pointwise weak$^*$ convergence. This is all summarized in the following theorem.
\begin{theorem}\label{thm: duality of gen holder space}
    Let $E$ be a real Banach space with a predual $F$, and let $0< \alpha \leq 1$. Then, $C_0^\alpha([0,T];E)$ is isomorphic to the dual of $\textnormal{\AE}_{\alpha}([0,T];F)$. Furthermore, an $\alpha$-Hölder norm bounded net $\{x_\lambda\}_{\lambda\in I}$ in $C_0^\alpha([0,T];E)$ converges to $x$ in the weak-$^*$ topology if and only if $x_\lambda(t)\wkst x(t)$ for all $t\in [0,T]$.
\end{theorem}
\begin{proof} 
    \textbf{Part 1: showing $C_0^\alpha([0,T];E)$ is isomorphic to the dual of $\textnormal{\AE}_{\alpha}([0,T];F)$}\\
    We will provide mappings 
    \begin{equation*}
        L_1\in L((\textnormal{\AE}_{\alpha}([0,T];F))^*, C_0^\alpha([0,T];E)) \text{ and }
        L_2 \in L(C_0^\alpha([0,T];E),(\textnormal{\AE}_{\alpha}([0,T];F))^*),
    \end{equation*} and then show that $L_1 L_2 = \operatorname{id}_{C_0^\alpha([0,T];E)}$ whereas $L_2 L_1 = \operatorname{id}_{(\textnormal{\AE}_{\alpha}([0,T];F))^*}$.\par 
    
    We first observe that for $t\in[0,T]$ and $\phi\in \left(\textnormal{\AE}_{\alpha}([0,T];F)\right)^*$ the mapping
    \[
    F\ni y \mapsto \inprod{\phi}{m^y_{t,0}}
    \]
    is an element of $E$ (interpreted as $F^*$). By the fact that $m^{y+ y'}_{t,0} =m^y_{t,0} + m^{y'}_{t,0} $ for any $y,y'\in F$, the map is indeed linear. Additionally, $\|m^y_{t,0}\|_{\textnormal{\AE},\alpha} \leq \|y\| t^\alpha$ for any $y\in F$ and $t\in[0,T]$, and so the map is continuous by the continuity of $\phi$.
    Next, we define
    \[
    \left(L_1 \phi\right) (t) \coloneqq \left(y \mapsto \inprod{\phi}{m^y_{t,0}}\right), \quad t\in [0,T],\, \phi\in \left(\textnormal{\AE}_{\alpha}([0,T];F)\right)^*.
    \]
    In order to prove that $L_1 \phi \in C_0^\alpha([0,T];E)$, fix a $\phi \in \left(\textnormal{\AE}_{\alpha}([0,T];F)\right)^*$. First note that $(L_1 \phi)(0) = 0$. Next, by definition, for $s,t \in [0,T]$ and $y\in F$,
    \[
    \big(\left(L_1 \phi\right) (t) - \left(L_1 \phi\right) (s)\big) (y)= \inprod{\phi}{m_{t,0}^y} - \inprod{\phi}{m_{s,0}^y} = \inprod{\phi}{m_{t,s}^y}.
    \]
    Because $\|m_{t,s}^y\|_{\textnormal{\AE},\alpha} \leq \|y\| (t-s)^\alpha$ and $\phi$ is continuous, we have (with $\| \cdot \|_{\textnormal{\AE}}^*$ the canonical norm on $\left(\textnormal{\AE}_{\alpha}([0,T];F)\right)^*$) that
    \[
    \|\left(L_1 \phi\right) (t) - \left(L_1 \phi\right) (s) \|_{C^\alpha} \leq (t-s)^\alpha \|\phi\|_{\textnormal{\AE},\alpha}^*,
    \]
    showing not only that $L_1 \phi$ is $\alpha$-Hölder continuous, but also that $\|L_1 \phi\|_{C^\alpha} \leq \|\phi\|_{\textnormal{\AE},\alpha}^*$. Thus, $L_1$ is bounded and $\|L_1\| \leq 1$.

    To define a map $L_2 \colon C_0^\alpha([0,T];E) \to  \left(\textnormal{\AE}_{\alpha}([0,T];F)\right)^*$, we start with defining how a $\alpha$-Hölder continuous functions act on the molecules. As in Equation \eqref{eq: def of how holder acts on ae}, we define, for $x\in C_0^\alpha([0,T];E)$, the linear map
    \begin{equation}\label{eq: def of T_2 AE proof}
        m \mapsto \inprod{x}{m} \coloneqq \sum _{t\in [0,T]} \inprod{x(t)}{m(t)},
    \end{equation}
    for any $m \in M([0,T];F)$. By Equation \eqref{eq: Arens-Eells inequality}, we have for $m \in M([0,T];F)$ \[
    |\inprod{x}{m}| \leq \|x\|_{C^\alpha} \|m\|_{\textnormal{\AE},\alpha},
    \]
    which implies that the map defined in Equation \eqref{eq: def of T_2 AE proof} can be uniquely extended to a continuous linear map on $\textnormal{\AE}_{\alpha}([0,T];F)$ with norm less or equal to $\|x\|_{C^\alpha} $. We define $L_2 \colon C_0^\alpha([0,T];E) \to  \left(\textnormal{\AE}_{\alpha}([0,T];F)\right)^*$, as the map that sends $x\in C_0^\alpha([0,T];E)$ to the extension of the map defined in \eqref{eq: def of T_2 AE proof}. Because $\|L_2 x\|_{\textnormal{\AE},\alpha}^* \leq\|x\|_{C^\alpha}  $, for $\alpha$-Hölder continuous $x$, we have that $\|L_2\| \leq 1$.

    It remains to show that the maps $L_1$ and $L_2$ are each others inverses. Take a $x\in C_0^\alpha([0,T];E)$. Then for any $t\in [0,T]$ we have
    \[
    (L_1 L_2x)(t) = \left( y\mapsto \inprod{L_2 x}{ m^y_{t,0}} \right) = \left( y\mapsto \inprod{x}{m^y_{t,0}} \right) = \left( y\mapsto \inprod{x(t)}{y}\right) = x(t),
    \]
    where the equalities follow by definition of respectively $L_1$ and $L_2$. This implies that $L_1 L_2 x =x$.
    
    Take a  $\phi \in \left(\textnormal{\AE}_{\alpha}([0,T];F)\right)^*$. For any $t\in[0,T]$ and $y \in F$, we have
    \begin{equation}\label{eq: ae proof aux 1}
        (L_2 L_1 \phi) (m^y_{t,0}) = \inprod{(L_1\phi) }{m^y_{t,0}}=\inprod{(L_1\phi) (t)}{y} = \inprod{\phi}{m^y_{t,0}},
    \end{equation}
    where again the equalities follow by the definition of respectively $L_2$ and $L_1$. Elementary molecules of the form $m^y_{t,0}$ span $M([0,T]:F)$ and so Equation \eqref{eq: ae proof aux 1} implies $L_2 L_1 \phi = \phi$.

    This shows that $L_1$ and $L_2$ are indeed each others inverses, so we have that $C_0^\alpha([0,T];E)$ is isomorphic to the dual of $\textnormal{\AE}_{\alpha}([0,T];F)$.

    \textbf{Part 2: showing weak$^*$ convergence corresponds to boundedness and pointwise convergence}\\
    From now we directly associate an element $x\in  C_0^\alpha([0,T];E) $ as an element of $\left(\textnormal{\AE}_{\alpha}([0,T];F)\right)^*$. We want to prove that a norm bounded net $\{x_\lambda\}_{\lambda\in I}$ in $C_0^\alpha([0,T];E)$ converges to $x\in C_0^\alpha([0,T];E)$ in the weak-$^*$ topology if and only if $x_\lambda (t) \wkst x(t)$ for all $t\in[0,T]$.

    For any $t\in [0,T]$ and $y\in F$ we must have
    \[
    \inprod{x_\lambda(t)}{y} = \inprod{x_\lambda}{m_{t,0}^y} \to \inprod{x}{m_{t,0}^y} = \inprod{x(t)}{y},
    \]
    and so $x_\lambda(t) \wkst x(t)$ for all $t\in[0,T]$.

    Next, assume $\{x_\lambda\}_{\lambda\in I}$ is bounded in $\alpha$-Hölder norm and $x_\lambda(t) \wkst x(t)$ for all $t\in[0,T]$. By the pointwise weak$^*$ convergence, we have for any $t\in[0,T]$ and $y\in F$
    \[
    \inprod{x_\lambda}{m_{t,0}^y}=\inprod{x_\lambda(t)}{y}   \to   \inprod{x(t)}{y}=\inprod{x}{m_{t,0}^y}.
    \]
    Any molecule $m$ is a linear combination of elementary molecules, so the above then gives us $\inprod{x_\lambda}{ m} \to\inprod{x}{ m}  $ for all $ m\in M([0,T];F)$. As molecules are dense in $\textnormal{\AE}_{\alpha}([0,T];F)$, we can apply Lemma \ref{lem: density implies generate predual}, which shows that $x_\lambda\wkst x$, which completes the proof. 
\end{proof}
\section{The \texorpdfstring{$i^*$}{weak-star} topology on  \texorpdfstring{$ \mathcal{C}^{\alpha,n} ([0,T];E )$}{the space of H\"older continuous functions}}\label{sec: weak star topology on alpha rough}

Recall from Section~\ref{sec: abstract UAT} that we defined the $i^*$ topology on $\alphageom$ as the topology induced by the weak$^*$ topology on $C_0^{\alpha}([0,T],T^{(n)}(E))$ through the natural embedding, see in particular Proposition~\ref{prop: inclusion of alpha rough into alpha holder} and Definition~\ref{def: i* star topology}. In this section we first provide the proof of Proposition~\ref{prop: inclusion of alpha rough into alpha holder} and then establish that norm-bounded sets in $\alphageom$ are indeed $i^*$ compact. 

\begin{proof}[Proof of Proposition~\ref{prop: inclusion of alpha rough into alpha holder}]
    We define the map $i$ by 
    \[
    i(\mathbf x)(s) \coloneqq \mathbf{x}_{0,s} - \mathbf{1},
    \]
    where $\mathbf{1} = (1,0,\ldots,0)$. 
    The fact that we have to subtract $\mathbf{1}$ in this map is awkward, but necessary to ensure that $i(\mathbf x)(0)=0$. Note that the map $i$ is clearly injective, provided it is well-defined. To show that $i$ is well-defined, we fix a $\mathbf{x}\in \mathcal{C}^{\alpha,n} ([0,T];E )$, and show that $ i(x)\in C_0^{\alpha}([0,T];T^{(n)}(E))$. As $ \mathbf{x}_{0,0}$ is equal to $(1, 0, \ldots,0)$, we only have to show that $i(\mathbf x)$ is indeed $\alpha$-Hölder continuous, which we show together with the second statement of the proposition. That is, we want to show there exists some $C$ such that for all $s,t\in[0,T]$
    \[
    \| \mathbf{x}_{0,t} - \mathbf{x}_{0,s} \|_{T^{(n)}(E)} \leq C \big(\|\mathbf{x}\|_{\mathcal{C^\alpha}} + 1\big)|t-s|^{\alpha},
    \]
    where the $\|\cdot\|_{T^{(n)}(E)}$ norm is given by $\|(\mathbf{x}^{(0)}, \ldots, \mathbf{x}^{(n)}\| = \|\mathbf{x}^{(0)}\|_\R + \ldots + \|\mathbf{x}^{(n)}\|_{n}$ (although any norm that respects the direct sum representation in Remark \ref{rem: banach space structure of truncated tensor algebra} would work). This is equivalent to showing that there exists some $C$ such that for all $1 \leq i \leq n$ and $s,t \in [0,T]$ we have
    \begin{equation}\label{eq:hoelderbound}
    \| \mathbf{x}^{(i)}_{0,t} - \mathbf{x}^{(i)}_{0,s} \|_i \leq C \big(\|\mathbf{x}\|_{\mathcal{C^\alpha}} + 1\big) |t-s|^{\alpha}.
    \end{equation}
    By Chen's relation (see Equation \eqref{eq: Chen final?}) we have $\mathbf{x}^{(i)}_{0,t} = \sum_{j=0}^{i} \mathbf{x}^{(j)}_{0,s} \otimes \mathbf x ^{(i-j)}_{s,t}$ for all $1\leq i \leq n$, so proving~\eqref{eq:hoelderbound} is equivalent to showing
    \begin{equation}\label{eq: aux eq for inclusion proof}
     \| \sum_{j=0}^{i-1} \mathbf x ^{(j)}_{0,s} \otimes \mathbf x ^{(i-j)}_{s,t}\|_{i} \leq C \big(\|\mathbf{x}\|_{\mathcal{C^\alpha}} + 1\big)  |t-s|^{\alpha}.
    \end{equation}
    By assumption, we have 
    \[
    \max_{1\leq i \leq n} \sup_{0\leq s<t \leq T}\frac{\|\mathbf{x}^{(i)}_{s,t}\|_i^{\nicefrac{1}{i}}}{ (t-s)^\alpha} =\|\mathbf{x}\|_{\mathcal{C}^\alpha},
    \]
    see Definition \ref{def: homogeneous alpha holder}. In other words, we have
    \[
    \| \mathbf x ^{(i)}_{s,t}\|_i \leq  \|\mathbf{x}\|^i_{\mathcal{C}^\alpha}(t-s)^{i \alpha},
    \]
    for any $1\leq i \leq n$ and $s,t \in [0,T]$. As a consequence, we have that there exists some $C>0$ (depending only on $n$, $T$, and $\alpha$), such that for all $0 \leq j <  i \leq n$,
    \[
    \| \mathbf x^{(i-j)}_{s,t}\|_{i-j} \leq C \big(\|\mathbf{x}\|_{\mathcal{C^\alpha}} + 1\big) (t-s)^{ \alpha} \quad \text{ and } \quad  \| \mathbf x^{(i)}_{0,s}\|_i \leq C.
    \]
    If we combine these two inequalities, we obtain Equation \eqref{eq: aux eq for inclusion proof}.
\end{proof}

To avoid too much notation, we sometimes omit the $i$ in notation. For example, for $m$ in the predual of $C_0^{\alpha}([0,T];T^{(n)}(E))$ and $\mathbf{x} \in  \mathcal{C}^{\alpha,n} ([0,T];E )$, we define
\begin{equation}\label{eq: def pairing i}
\inprod{\mathbf{x}}{m} \coloneqq \inprod{i (\mathbf{x})}{m}.
\end{equation}
A small word of warning however: given an elementary molecule of the form $m_{0,t}^{y}$, where $y$ is in the predual of $T^{(n)}(E)$ and $t\in[0,T]$ (see Equation \eqref{eq: explicit molecule definition}), one may be tempted to write
\[
\inprod{\mathbf{x}}{m_{0,t}^y} = \inprod{\mathbf{x}_{0,t}}{y},
\]
but note that when $y^{(0)} \neq 0$ this notation is confusing and better avoided. 

\begin{remark}\label{rem: weak star topology on alpha rough}
It follows from Theorem \ref{thm: duality of gen holder space} and Proposition \ref{prop: inclusion of alpha rough into alpha holder} that if $\{\mathbf{x}_\lambda\}_{\lambda \in I}$ is a net in $\mathcal{C}^{\alpha,n} ([0,T];E )$ that is bounded with respect to $\|\cdot \|_{\mathcal{C}^{\alpha}}$, then $\{\mathbf{x}_\lambda\}_{\lambda \in I}$ is Cauchy in the $i^*$-topology if and only if there exists a $x\in C^{\alpha}_0([0,T];T^{(n)}(E))$ such that  $i(\mathbf{x}^{}_\lambda)(t) \wkst x(t)$ for all $t\in [0,T]$. 
\end{remark}

Recall that our motivation for endowing the space $\mathcal{C}^{\alpha,n} ([0,T];E )$ with a `weak-$^*$ like' topology is that we want norm-bounded sets to be compact. This is indeed the case:

\begin{proposition}\label{prop: compactness of rough alpha ball}
    Assume the Setting \ref{assumptions}. Let $0< \alpha \leq 1$, let $n\in \N_{>0}$ and let $R>0$. Then \newline $\BRn{n}$ is compact in the $i^*$-topology. 
\end{proposition}
\begin{proof}
For notational brevity we define $
    B_R(\mathcal{C}^{\alpha,n}(E))=\BRn{n}.$
    
    For this proof we use that closed sets of compact spaces are themselves compact. By the Banach--Alaoglu Theorem, norm closed balls of the form
    \[\left\{ \mathbf{x} \in C_0^{\alpha}([0,T];T^{(n)}(E)) \mid \|\mathbf{x}\|_{{C}^{\alpha}} \leq R'\right\}
    \]are weak-$^*$ compact. Therefore, it suffices to show that $i(B_R(\mathcal{C}^{\alpha,n}(E)))$ is a weak-$^*$ closed subset of such a norm closed $C^\alpha$ ball. The second part of Proposition \ref{prop: inclusion of alpha rough into alpha holder} tells us that indeed there exists $R'$ such that
    \[
    i(B_R(\mathcal{C}^{\alpha,n}(E))) \subseteq \left\{ \mathbf{x} \in C_0^{\alpha}([0,T];T^{(n)}(E)) \mid \|\mathbf{x}\|_{{C}^{\alpha}} \leq R'\right\},
    \]
    so only remains to show that $i(B_R(\mathcal{C}^{\alpha,n}(E))) $ is weak-$^*$ closed. In order to show this, we use that a set is closed if and only if the limit of any converging net in this set remains in this set. Assume that we have a net $\{\mathbf{x}_\lambda\}_{\lambda\in I}$ in $B_R(\mathcal{C}^{\alpha,n}(E))$ and that $i(\mathbf{x}_\lambda) \wkst x\in C_0^{\alpha}([0,T];T^{(n)}(E))$. In order to show that there exists an $\mathbf{x}\in \mathcal{C}^{\alpha,n}([0,T];E)$ such that $x=i(\mathbf{x})$, we define $\mathbf{x}_{s,t}$, for $s,t\in[0,T]$, to be
    \begin{equation}\label{eq: aux weak-$^*$ closedness of unit ball}
    \mathbf{x}_{s,t} \coloneqq (x(s) + \mathbf{1})^{-1} \hat{\otimes} (x(t)+\mathbf{1}).
    \end{equation}
    As this definition automatically guarantees Chen's relation, we only have to show that $\|\mathbf{x}\|_{\mathcal{C}^{\alpha,n}} \leq R$. By Definition \ref{def: homogeneous alpha holder}, this comes down to showing for any $s,t\in [0,T]$ and $1\leq i \leq n$ that\[
    \|\mathbf{x}_{s,t}^{(i)}\| \leq R^i (s-t)^{i\alpha} .
    \] 
    Fix $1<i\leq n$ and $s,t\in [0,T]$. Equation \eqref{eq: aux weak-$^*$ closedness of unit ball} (combined with Equation \eqref{eq: explicit inverse tensor algebra element}) implies that we can write $\mathbf{x}^{(i)}_{s,t}$ as a linear combination
    \[
    \begin{split}
     \mathbf{x}^{(i)}_{s,t} &= \sum_{i_1+ \ldots + i_{k} = i}c_{i_1,\ldots,i_{k}}\mathbf{x}^{(i_1)}_{0,s} \otimes \ldots \otimes \mathbf {x}^{(i_{k-1})}_{0,s} \otimes \mathbf{x}^{(i_{k})}_{0,t}\\&= \sum_{i_1+ \ldots + i_{k} = i}c_{i_1,\ldots,i_{k}}\mathbf{x}^{(i_1)}(s) \otimes \ldots \otimes \mathbf {x}^{(i_{k-1})}(s) \otimes \mathbf{x}^{(i_{k})}(t),
     \end{split}
    \]
    for suitably chosen constants $c_{i_1,\ldots,i_{k}}$ that are independent of $\mathbf{x}$. For $(\mathbf{x}_\lambda)_{s,t}^{(i)}$ we have an analogous expression. We recall (see Remark \ref{rem: weak star topology on alpha rough}) that weak-$^*$ convergence in $C_0^\alpha([0,T];T^{(n)}(E))$ means that $(\mathbf{x}^{(j)}_\lambda)(u) \wkst \mathbf{x}^{(j)}(u)$ for any $u\in [0,T]$ and $0\leq j \leq n$. This then implies that by repeated use of Lemma \ref{lem: weak star convergence tensor product} (where by Proposition \ref{prop: inclusion of alpha rough into alpha holder} the net is norm bounded)
    \[
    i(\mathbf{x}_\lambda)^{(i_1)}(s) \otimes \ldots \otimes i(\mathbf {x}_\lambda)^{(i_{k-1})}(s) \otimes i(\mathbf{x}_\lambda)^{(i_{k})}(t) \wkst   x^{(i_1)}(s) \otimes \ldots \otimes x^{(i_{k-1})}(s) \otimes x^{(i_{k})}(t),
    \]
    and so
    \begin{equation}\label{eq: aux alpha rough compactness}
    (\mathbf{x}_\lambda)_{s,t}^{(i)} \wkst \mathbf{x}_{s,t}^{(i)}.
    \end{equation}
    As the norm is lower semicontinuous under the weak-$^*$ convergence (see Lemma \ref{lem: lower semicontinouity of norm}) and as $\|\mathbf{x}_\lambda\|_{\mathcal{C}^{\alpha,n}} \leq R$, we find that
    \[
    \|\mathbf{x}^{(i)}_{s,t}\|  \leq \liminf_{\lambda \in I} \|(\mathbf{x}_\lambda)^{(i)}_{s,t}\| \leq R^i(t-s)^{i\alpha},   
    \]
    which shows that $\mathbf{x} \in B_R(\mathcal{C}^{\alpha,n}(E))$, concluding the proof.
    
\end{proof}

Proposition \ref{prop: compactness of rough alpha ball} states that the set $\BRn{n}$ is compact. We now prove that this also holds for the weakly geometric version  $\BRgn{n}$ (see Definition~\ref{def: weakly geometric rough paths}). 
\begin{proposition}\label{prop: compactness of rough alpha ball weak}
Assume Setting \ref{assumptions}, let $0< \alpha \leq 1$, $n\in \N_{>0}$, $R\in (0,\infty)$. Then $\BRgn{n}$ is compact in the $i^*$ topology. 
\end{proposition}
\begin{proof} 
For notational brevity we define $B_R(\mathcal{C}^{\alpha,n}_g(E)) := \BRgn{n}$; the set $B_R(\mathcal{C}^{\alpha,n}(E))$ is defined analogously. As 
     $B_R(\mathcal{C}^{\alpha,n}_g(E)) \subseteq B_R(\mathcal{C}^{\alpha,n}(E)),
     $ we can apply Proposition \ref{prop: compactness of rough alpha ball} so that we only have to show that the set $B_R(\mathcal{C}^{\alpha,n}_g(E)) $ is $i^*$-closed. For this, it suffices to show that whenever we have a net $\{\mathbf{x}_\lambda\}_{\lambda \in I}$ in this space that converges to $\mathbf x$, this $\mathbf x$ must also be in $B_R(\mathcal{C}^{\alpha,n}_g(E)) $.

    Take such a net converging to $\mathbf x$. Proposition \ref{prop: compactness of rough alpha ball} implies that $\mathbf x$ is in $B_R(\mathcal{C}^{\alpha,n}_g(E)) $ and so we only have to show that $\mathbf{x}$ takes values in $G^{(n)}(E)$ (see Definition~\ref{def: group like elements strong}).

    For this, we note that Lemma \ref{cor: point separating small} is also satisfied in Setting \ref{assumptions}, as a predual separates points on the dual, and as the algebraic tensor space forms a dense subset of the predual. This implies that the equality of Proposition \ref{prop: equality of group like definitions} is still valid when the uniform crossnorm and approximation property assumptions are replaced by the assumptions of Setting \ref{assumptions}. As a consequence,  $\mathbf{x}$ takes values in $G^{(n)}(E)$ when $\mathbf{x}_{s,t}$ satisfies the shuffle product for any linear functionals $y,y' \in T_a^{(n)}(F)$ and $s,t \in [0,T]$. In other words, we need
    \[
    \inprod{\mathbf{x}_{s,t}}{ y \shuffle y'} = \inprod{\mathbf{x}_{s,t}}{ y}\inprod{\mathbf{x}_{s,t}}{ \shuffle y'} .
    \]
    However, by the same arguments as in the proof of Proposition~\ref{prop: compactness of rough alpha ball}, we obtain that Equation \eqref{eq: aux alpha rough compactness} holds, and so
    \[
    \inprod{\mathbf{x}_{s,t}}{ y \shuffle y'} = \lim_{\lambda \in I}\inprod{(\mathbf{x}_\lambda)_{s,t}}{ y \shuffle y'} = \lim_{\lambda\in I}\inprod{(\mathbf{x}_\lambda)_{s,t}}{ y}\inprod{(\mathbf{x}_\lambda)_{s,t}}{  y'} =\inprod{\mathbf{x}_{s,t}}{ y}\inprod{\mathbf{x}_{s,t}}{  y'},
    \]
    where for the second equality we used that $\mathbf{x}_\lambda$ is weakly geometric. This shows that $\mathbf{x}_{s,t}$ is weakly geometric as well, and so the proof is finished. 
\end{proof}

\section{\texorpdfstring{$i^*$}{Weak-star} continuity of Lyons lift}\label{sec: i star continuity lyons lift}
A key ingredient of the proof of Theorem \ref{thm: weak star version uat} is that for all $l\in T_a(F)$, the mapping
\[
\mathbf{x} \mapsto \inprod{S(\mathbf{x})_{0,T}}{l},
\]
is $i^*$ continuous when restricted to sets that are bounded with respect to the ${\mathcal{C}^{\alpha}}$ norm. Here $i^*$ is the topology from Definition \ref{def: i* star topology}; note also that we continue to assume Setting~\ref{assumptions} throughout this section. While it is well-known that the Lyons lift is continuous with respect to the $\varrho^{\text{hom}}_{\alpha}$ metric (see \cite[Theorem 3.10]{lyons_differential_2007} or Section~\ref{sec: Lyons lift}), continuity with respect to the $i^*$ topology is not immediately clear and the goal of this section is to establish this. More specifically, we will prove the slightly stronger statement that that the Lyons lift $S^n$, $n\in \N_{>0}$, is $i^*$ to $i^*$ continuous on sets that are bounded with respect to the ${\mathcal{C}^{\alpha}}$ norm. The reason we prove a stronger statement is that the Lyons lift is defined via an induction step that involves the whole path, and we need to show that continuity is preserved in the induction step. 

\begin{proposition}[$i^*$ to $i^*$ continuity of Lyons lift]\label{prop: continuity of Lyons lift} Assume Setting~\ref{assumptions} and let $\alpha > 0$, $n\in \N_{>0}$ with $n\geq \lfloor\nicefrac{1}{\alpha}\rfloor$, $R>0$, and let $S^n\colon \alpharough  \to \mathcal C^{\alpha,n} ([0,T];E )$ be the Lyons lift (see Theorem \ref{thm: Lyons lift}). Then, the map $S^n$ restricted to $\BR$ is $i^*$ to $i^*$ continuous. In particular, maps of the form $\mathbf{x} \mapsto \inprod{\mathbf{x}_{0,T}}{y}\coloneqq \inprod{S^n(\mathbf{x})_{0,T}}{y}$, where $y\in T^{(n)}(F)$, are $i^*$ continuous when restricted to $\BR$. 
\end{proposition}

A proof of Proposition~\ref{prop: continuity of Lyons lift} for the case that $E=\R^d$ ($d\in \N_{>0}$) can be found in~\cite{cuchiero_global_2023}, where they combine \cite[Theorem A.4]{cuchiero_global_2023} with \cite[Theorem 9.10]{friz_multidimensional_2010}. As it is not clear how to adapt that proof to the setting of $E$-valued paths, we instead demonstrate that it is possible to reduce from $E$-valued paths to the $\R^d$-valued setting. Key ingredients are Lemmas~\ref{lem:reduction of i* continuity} and~\ref{lem:reduction to finite dimensions} below. Moreover, for the readers' convenience, we also provide a streamlined proof of Proposition~\ref{prop: continuity of Lyons lift} for the case $E=\R^d$. I.e., we provide an adaptation of the proof in~\cite{cuchiero_global_2023} that is modified for our purposes (see p.\pageref{proof: fin dim continuity}). The proof of Proposition~\ref{prop: continuity of Lyons lift} can be found on page \pageref{proof: continuity of the Lyons lift}.

First of all, for the sake of notational brevity, we define the set
\begin{equation}
B_R(\mathcal{C}^{\alpha}(E))
:= \left\{ \mathbf{x} \in \mathcal{C}^{\alpha} ([0,T];E ) \mid \|\mathbf{x}\|_{\mathcal{C}^{\alpha}} \leq R\right\},
\end{equation}
where $R>0$. 
The sets $B_R(\mathcal{C}^{\alpha,n}(E))$, $B_R(\mathcal{C}^{\alpha}(\R^{d}))$, and $B_R(\mathcal{C}^{\alpha,n}(\R^d))$ , with $R>0$, $n\in \N \cap [\lfloor \nicefrac{1}{\alpha}\rfloor,\infty)$, and $d\in \N_{>0}$, are defined analogously.

\begin{lemma}\label{lem:reduction of i* continuity} Assume Setting~\ref{assumptions}, let $\alpha>0$ and let $n\in \N \cap [\lfloor \nicefrac{1}{\alpha}\rfloor, \infty)$. Suppose $S^{(n)}|_{B_{R}(\mathcal{C}^{\alpha}(E))}$ is $i^*$ to $i^*$ continuous for all $R>0$, and suppose
\begin{equation}
            B_R(\mathcal{C}^{\alpha,n}(E)) \ni \mathbf{x} \mapsto 
            \langle S^{n+1} (\mathbf{x}) , m_{0,s}^{y_1\otimes \ldots \otimes y_{n+1}} \rangle
\end{equation}
is $i^*$ continuous for all $y_1,\ldots,y_{n+1}\in F$, $s\in [0,T]$ and $R>0$, $m_{0,s}^{y_1\otimes \ldots \otimes y_{n+1}}$ is a molecule of the form~\eqref{eq: explicit molecule definition}, and the pairing $\langle \cdot , \cdot \rangle$ is given by~\eqref{eq: def pairing i}.
Then $S^{(n+1)}|_{B_R(\mathcal{C}^{\alpha}(E))} $ is $i^*$ to $i^*$ continuous for all $R>0$.
    \end{lemma}

The proof of this lemma is postponed to the end of this section.

\begin{lemma}\label{lem:reduction to finite dimensions} Assume Setting~\ref{assumptions},
let $\alpha>0$, $R>0$, $n\in \N \cap [\lfloor \nicefrac{1}{\alpha}\rfloor, \infty)$, let $y_1,\ldots, y_{n+1}\in F$ and set $\mathbf{y} = y_1 \otimes \ldots\otimes y_{n+1}$. Then there exists an $e\in \R^{n+1}$ and an $i^*$ to $i^*$ continuous mapping
    \[
    \pi^{\mathbf{y}}\colon \mathcal{C}^{\alpha,n}([0,T];E) \to \mathcal{C}^{\alpha,n}([0,T];\R^{n+1})
    \]
    such that 
    \begin{equation}\label{eq: projection aux in prove i* continuity  of lyons lfit}
     \inprod{
     S^{n+1}(\mathbf{x}) 
     }{
        m_{0,t}^{\mathbf{y}}
    } =\inprod{S^{n+1}(\pi^{\mathbf{y}}(\mathbf{x})) }{m_{0,t}^e}
    \end{equation}
for all $t\in [0,T]$ and all $\mathbf{x}\in \mathcal{C}^{\alpha,n}([0,T];E)$ (note that the second instance of $S^{n+1}$ above denotes the Lyons lift of $\R^{n+1}$-valued paths).
Moreover, $\pi^{\mathbf{y}}$ maps norm-bounded sets in $\mathcal{C}^{\alpha,n}([0,T];E)$ to norm-bounded sets in $\mathcal{C}^{\alpha,n}([0,T];\R^{n+1})$.
\end{lemma}

The proof of this lemma is also postponed to the end of this section.

\begin{proof}[Proof of Proposition~\ref{prop: continuity of Lyons lift}]\label{proof: continuity of the Lyons lift}
The proof of Proposition~\ref{prop: continuity of Lyons lift} follows immediately from Lemmas~\ref{lem:reduction of i* continuity} and~\ref{lem:reduction to finite dimensions}, the fact that the proposition is true for $E=\R^d$ ($d\in \N$) (see below), and an induction argument.  
\end{proof}

\begin{proof}[Proof of Proposition~\ref{prop: continuity of Lyons lift} for the case $E=\R^d$]\label{proof: fin dim continuity}
 For $\alpha'< \alpha$, we have for $1\leq i \leq n$, $s,t\in [0,T]$ and $\mathbf{x} \in \mathcal{C}^{\alpha,n} ([0,T];\R^{d} )$
\[
\frac{\|\mathbf{x}^{(i)}_{s,t}\|_{i}}{(s-t)^{i \alpha'}} = \frac{(s-t)^{i \alpha}}{(s-t)^{i \alpha'}}\frac{\|\mathbf{x}^{(i)}_{s,t}\|_{i}}{(s-t)^{i \alpha}} \leq T^{i(\alpha-\alpha')} \frac{\|\mathbf{x}^{(i)}_{s,t}\|_{i}}{(s-t)^{i \alpha}}.
\]
Therefore for every $\alpha'< \alpha$, we have an inclusion $ \mathcal{C}^{\alpha,n} ([0,T];\R^{d} )\subseteq  \mathcal{C}^{\alpha',n} ([0,T];\R^{d} )$, and thus we can equip the space $ \mathcal{C}^{\alpha,n} ([0,T];\R^{d}) $ with the subspace topology generated by the $\varrho^{\text{hom}}_{\alpha'}$ metric (see Definition~\ref{def: homogeneous alpha holder}). Recall from \cite[Theorem 3.10]{lyons_differential_2007} (see also Section~\ref{sec: Lyons lift}) that the Lyons lift 
\[
S^{n} \colon  \mathcal{C}^{\alpha'} ([0,T];\R^{d} ) \to \mathcal{C}^{\alpha',n} ([0,T];\R^{d}) 
\]
is $\varrho^{\text{hom}}_{\alpha'}$ continuous for all $\alpha'>0$, $d\in \N_{>0}$, and all $n\in \N\cap  [\lfloor \nicefrac{1}{\alpha'} \rfloor, \infty)$. Therefore, it suffices to show that for any $n\in \N_{>0}\cap [ \lfloor \nicefrac{1}{\alpha'} \rfloor,\infty)$, $R\in\R$ and $\alpha'<\alpha$, the $i^*$ topology restricted to $B_R(\mathcal{C}^{\alpha,n}(\R^{d}))$ coincides with the $\varrho^{\text{hom}}_{\alpha'}$ metric topology restricted this set (note this is not the case when $\alpha'$ is replaced by $\alpha$). This is equivalent to proving that the identity map
\[
\begin{split}
\id \colon  \left( B_R(\mathcal{C}^{\alpha,n}(\R^{d}) )
, \varrho^{\text{hom}}_{\alpha'}
\right) 
\to 
\left( B_R(\mathcal{C}^{\alpha,n}(\R^{d})) 
, 
i^*\right),
\end{split}
\]
is a homeomorphism. Note that this map is not an homeomorphism when $\R^{d}$ is replaced by a infinite dimensional Banach space (the weak-$^*$ topology does not coincide with the norm topology, even when restricted to norm-bounded sets). 

We start with showing that the identity map is continuous. As we have a metric topology on the domain of the identity map, it is sufficient to show that whenever a sequence $\seq{\mathbf{x}}$ converges to $\mathbf{x}$ with respect to the $\varrho^{\text{hom}}_{\alpha'}$ metric, it also converges in the $i^*$ topology. An immediate consequence of convergence with respect to the $\varrho^{\text{hom}}_{\alpha}$ metric, is that $\{(\mathbf{x}_m)_{s,t}\}_{m\in \N}$ converges to $\mathbf{x}_{s,t}$ for all $s,t \in [0,T]^2$. By Theorem~\ref{thm: duality of gen holder space} and Remark \ref{rem: weak star topology on alpha rough} it thus follows that $\seq{\mathbf{x}}$ converges to $\mathbf{x}$ in the $i^*$ topology.

To show that the inverse of the identity map is also continuous, we apply the same strategy. As the $i^*$ topology restricted to $B_R(\mathcal{C}^{\alpha,n}(\R^d))$ is metrizable (see Remark \ref{rem: nets instead of sequences}), it suffices to show that whenever a sequence $\seq{\mathbf{x}}$ converges to $\mathbf{x}$ in the $i^*$ topology, it also converges with respect to the $\varrho^{\text{hom}}_{\alpha}$ metric. The same argument we used to prove Equation \eqref{eq: aux alpha rough compactness} in the proof of Proposition \ref{prop: compactness of rough alpha ball}, holds here, so that again
\begin{equation}\label{eq: aux proof lyons lift weak star convergence equation fin dim}
(\mathbf{x}_m)_{s,t}^{(i)} \to (\mathbf{x})_{s,t}^{(i)}
\end{equation}
for all $i \leq n$ and $s,t \in [0,T]^2$. We claim that with an Arzelà--Ascoli type argument we obtain
\begin{equation}\label{eq : aux arzela-ascoli type}
\sup_{s,t \in [0,T]} |(\mathbf{x}_m)_{s,t}^{(i)} - (\mathbf{x})_{s,t}^{(i)}| \to 0;
\end{equation}
we postpone proving this claim to the end of this proof. Assuming this claim to hold, we have, using that $\|\mathbf{x}_m\|_{\mathcal{C}^{\alpha,n}}, \|\mathbf{x}_m\|_{\mathcal{C}^{\alpha,n}} \leq R$ for all $m\in \N$
\[
\begin{split}
\sup _{s,t\in[0,T]}\left|  \frac{(\mathbf{x}_m)_{s,t}^{(i)} - (\mathbf{x})_{s,t}^{(i)}}{(t-s)^{i\alpha'}}\right| \leq&\sup _{s,t\in[0,T]}  \left|\frac{(\mathbf{x}_m)_{s,t}^{(i)} - (\mathbf{x})_{s,t}^{(i)}}{(t-s)^{i\alpha}} \right|^{\ \frac{\alpha'}{\alpha}}  \left| (\mathbf{x}_m)_{s,t}^{(i)} - (\mathbf{x})_{s,t}^{(i)}\right|^{1-\frac{\alpha'}{\alpha}} \\\leq &(2R)^\frac{\alpha'}{\alpha} \left(\sup_{s,t \in [0,T]} |(\mathbf{x}_m)_{s,t}^{(i)} - (\mathbf{x})_{s,t}^{(i)}|\right)^{1-\frac{\alpha'}{\alpha}},
\end{split}
\]
which tends to zero as $m\to \infty$. This shows the inverse of the identity mapping is continuous. 

To show Equation \eqref{eq : aux arzela-ascoli type}, we start with the observation that for any $s,s',t,t'$ and \newline$\mathbf{z} \in B_R(\mathcal{C}^{\alpha,n}(E))$, using Chen's relation (see Equation \eqref{eq: Chen final?}) we have
\[
\mathbf{z}_{s',t'} = \mathbf{z}_{s',s} \hat\otimes \mathbf{z}_{s,t}\hat \otimes \mathbf{z}_{t,t'},
\]
and so, 
\[
\mathbf{z}_{s',t'} ^{(i)}= \sum_{j=0}^i \sum_{k=0}^{i-j}\mathbf{z}_{s',s}^{(j)} \otimes \mathbf{z}^{(i-j-k)}_{s,t} \otimes \mathbf{z}_{t,t'}^{(k)}.
\]
The difference $\mathbf{z}_{s',t'} ^{(i)} - \mathbf{z}_{s,t} ^{(i)}$ then can be written as a sum whose terms have either a $\mathbf{z}_{s',s}$ or a $\mathbf{z}_{t,t'}$ component of order at least 1. We also note that we can bound
\[
\|\mathbf{z}^{(j)}_{s',s} \|_j \leq R (s-s')^{j \alpha}
\]
for any $1\leq j \leq i$, and similarly for $\mathbf{z}^{(i)}_{t,t'}$. Therefore, for all $\epsilon>0$ there exists a $\delta>0$, depending only on $R$, $n$ and $\epsilon$, so that whenever $|s-s'|, |t-t'| < \delta$, we have
\begin{equation}\label{eq: aux equicontuinity of rough paths}
\|\mathbf{z}_{s',t'} ^{(i)}-\mathbf{z}_{s,t} ^{(i)}\|_i < \tfrac{\epsilon}{3},
\end{equation}
From Equation \eqref{eq: aux proof lyons lift weak star convergence equation fin dim}, we get that for any finite partition $D$ and $\epsilon>0$, we can find $N$ such that for all $m \geq N$
\begin{equation}\label{eq: aux proof lyons in partition form}
\sup_{{s,t} \in D}\|(\mathbf{x}_m)_{s,t}^{(i)} - (\mathbf{x})_{s,t}^{(i)}\|_i < \tfrac{\epsilon}{3}.
\end{equation}
In order to prove the claim of Equation \eqref{eq : aux arzela-ascoli type}, we can use a three $\epsilon$ argument with Equations \eqref{eq: aux equicontuinity of rough paths} and \eqref{eq: aux proof lyons in partition form} using $ \epsilon' = \frac{\epsilon}{3}$, and a partition $D$ with mesh size smaller then $\delta$. Specifically, for any $s',t'\in [0,T]$, we can find $s,t \in D$ with $|s-s'|,|t-t'| < \delta$ and therefore, 
\[
\|(\mathbf{x}_m)_{s',t'}^{(i)} - \mathbf{x}_{s',t'}^{(i)}\|_i \leq \|(\mathbf{x}_m)_{s',t'}^{(i)} - (\mathbf{x}_m)_{s,t}^{(i)}\|_i  + \|(\mathbf{x}_m)_{s,t}^{(i)} - \mathbf{x}_{s,t}^{(i)}\|_i  + \|\mathbf{x}_{s,t}^{(i)} - \mathbf{x}_{s',t'}^{(i)}\|_i < \tfrac{\epsilon}{3}  + \tfrac{\epsilon}{3}  + \tfrac{\epsilon}{3}  = \epsilon.
\]
This proves Equation \eqref{eq : aux arzela-ascoli type}, and so the proof is finished.

\end{proof}

    \begin{proof}[Proof of Lemma~\ref{lem:reduction of i* continuity}]Note that for every $R>0$ there exists an $R'\geq R$ such that 
    \[S^{n}(B_R(\mathcal{C}^{\alpha}(E))) \subseteq B_{R'}(\mathcal{C}^{\alpha,n}(E)) \]
    (see Theorem \ref{thm: Lyons lift}, and Remark \ref{rem: lyons lift same notation}).
    In view of the fact that $S^{n}|_{B_R(\mathcal{C}^{\alpha}(E))}$ is $i^*$ to $i^*$ continuous for all $R>0$ by assumption, it thus suffices to prove that $S^{n+1}|_{B_R(\mathcal{C}^{\alpha,n}(E))}$ is $i^*$ to $i^*$ continuous for all $R>0$.
    
    Let $(X,\tau)$ be a topological space, and let $B$ be a Banach space.  Recall that a mapping $f \colon X \rightarrow B^*$ is $\tau$ to weak-$^*$ continuous if and only if $b \circ f$ is $\tau$ continuous for all $b\in B$, see for example \cite[Theorem 1.2]{voigt_course_2020}. In view of Theorem~\ref{thm: duality of gen holder space} and Definition~\ref{def: i* star topology} it thus suffices to show that the mapping $$B_R(\mathcal{C}^{\alpha,n}(E))\ni\mathbf{x} \mapsto \inprod{S^{n+1} (\mathbf{x})}{m}$$ is $i^*$ continuous
    for all $m\in\textnormal{\AE}_\alpha ([0,T];T^{(n+1)}(F))$ and all $R>0$.

    Recall from Definitions~\ref{def: molecules} and~\ref{def: Arens--Eells} that the set of molecules is (norm) dense in $\textnormal{\AE}_\alpha ([0,T];T^{(n+1)}(F))$. Thus by Lemma \ref{lem: density implies generate predual} it suffices to show that the mapping $$B_R(\mathcal{C}^{\alpha,n}(E))\ni\mathbf{x} \mapsto \inprod{S^{n+1}  (\mathbf{x})}{m}$$ is $i^*$-continuous for every molecule $m$. As molecules are linear combinations of elementary molecules, and as linear combinations of continuous functions are continuous, it in fact suffices to show that the mapping $$B_R(\mathcal{C}^{\alpha,n}(E))\ni\mathbf{x} \mapsto \inprod{S^{n+1} (\mathbf{x})}{m_{0,s}^y}$$ is $i^*$-continuous for every elementary molecule $m^y_{0,s}$ and all $R>0$, where $y\in T^{(n+1)}(F)$ and $s\in [0,T]$.
    
    Finally, given $y\in T^{(n+1)}(F)$ we write $y = y' +\tilde{y}$, with $y'\in F^{\otimes(n+1)}$ and $\tilde y\in T^{(n)}(F)$. This means that $m^y_{0,s} = m^{y'}_{0,s}+m^{\tilde y}_{0,s}$. We can interpret $m^{\tilde y}_{0,s}$ as an element of $\textnormal{\AE}_\alpha ([0,T];T^{(n)}(F))$ and thus we have that 
    \[ \inprod{S^{n+1} (\mathbf{x})}{m^{\tilde y}_{0,s}} = \inprod{\mathbf{x}}{m^{\tilde y}_{0,s}},\] 
    which is continuous on $B_R(\mathcal{C}^{\alpha,n}(E))$ by definition of the $i^*$ topology. Thus, it suffices to prove that the mapping
    \[
     B_R(\mathcal{C}^{\alpha,n}(E))\ni\mathbf{x} \mapsto\inprod{S^{n+1} (\mathbf{x})}{m^{ y'}_{0,s}},
    \]
    is $i^*$ continuous for all $y'\in F^{\otimes n+1}$, all $s\in [0,T]$,  and all $R>0$. 
    
    Finally, note that 
    \[\|m^{y_1}_{0,s} - m^{y_2}_{0,s}\|_{\textnormal{\AE},\alpha} = \|m^{y_1-y_2}_{0,s} \|_{\textnormal{\AE},\alpha} \leq \|y_1-y_2\|_{n+1}^* s^\alpha, 
    \]
    for any $y_1,y_2 \in T^{(n+1)}(F)$. Additionally, we know that the linear span of elementary tensors is dense in $F^{\otimes(n+1)}$. Therefore, we can conclude that it suffices to prove that the mapping
    \[
     B_R(\mathcal{C}^{\alpha,n}(E))\ni\mathbf{x} \mapsto\inprod{S^{n+1} (\mathbf{x})}{m^{y_1\otimes \ldots \otimes y_{n+1}}_{0,s}},
    \]
    is $i^*$ continuous for all $y_1,\ldots,y_n\in F$, all $s\in [0,T]$, and all $R>0$. 
    \end{proof}

\begin{proof}[Proof of Lemma~\ref{lem:reduction to finite dimensions}]
Let $\{e_1,\ldots,e_{n+1}\}$ be the standard basis for $\R^{n+1}$, where $\R^{n+1}$ and its tensor powers are equipped with some admissible family of tensor norms, also denoted by $\{\| \cdot \|_n\}_{n\in \N}$ for simplicity (note that the choice of norms is irrelevant as all finite dimensional norms are equivalent).
For $\mathbf{x}\in T^{(n)}(E)$, define $\pi^{\mathbf{y}}(\mathbf{x})$ by setting $(\pi^{\mathbf{y}}(\mathbf{x}))^{(0)}=\mathbf{x}^{(0)}$ and
    \[
    \inprod{\pi^\mathbf{y}(\mathbf{x})}{e_{i_1}\otimes\ldots \otimes e_{i_j}} \coloneqq \inprod{\mathbf{x}}{y_{i_1}\otimes\ldots \otimes y_{i_j}} 
    \]
for all $1\leq j \leq n$ and all $i_1,\ldots, i_j\in \{1,\ldots,n+1\}$ (recall that $\mathbf{y} = y_1 \otimes \ldots \otimes y_{n+1}$, for $y_i \in F$ for all $1\leq i \leq n+1$). We claim that by pointwise application, $\pi^{\mathbf{y}}$ extends to a mapping from $\mathcal{C}^{\alpha,n}([0,T];E)$ to $\mathcal{C}^{\alpha,n}([0,T];\R^{n+1})$. Indeed, let $\mathbf{x}\in \mathcal{C}^{\alpha,n}([0,T];E)$. We first show that Chen's relation is satisfied for $\pi(\mathbf{x})$: 
    \begin{equation}\label{eq: rough property of projection}
    \begin{split}
        \inprod{\pi^\mathbf{y}(\mathbf{x})_{s,t}}{e_{i_1}\otimes\ldots \otimes e_{i_j}} & =\inprod{\mathbf{x}_{s,t}}{y_{i_1}\otimes\ldots \otimes y_{i_j}} = \inprod{\mathbf{x}_{s,u} \hat\otimes\mathbf{x}_{u,t}}{y_{i_1}\otimes\ldots \otimes y_{i_j}} \\&= \sum_{k=0}^j  \inprod{\mathbf{x}_{s,u}}{y_{i_1}\otimes\ldots \otimes y_{i_k}}\inprod{\mathbf{x}_{u,t}}{y_{i_{k+1}}\otimes\ldots \otimes y_{i_j}}\\& = \sum_{k=0}^j  \inprod{\pi^\mathbf{y}(\mathbf{x})_{s,u}}{e_{i_1}\otimes\ldots \otimes e_{i_k}}\inprod{\pi^\mathbf{y}(\mathbf{x})_{u,t}}{e_{i_{k+1}}\otimes\ldots \otimes e_{i_j}}\\& = \inprod{\pi^\mathbf{y}(\mathbf{x})_{s,u}\hat\otimes\pi^\mathbf{y}(\mathbf{x})_{u,t}}{e_{i_1}\otimes\ldots \otimes e_{i_j}}, 
    \end{split}
    \end{equation}
    for any combination of basis elements $e_{i_1},\ldots,e_{i_j}$ with $1\leq j \leq n$ and all $s,u,t \in [0,T]^2$. Furthermore, as $T^{(n)}(\R^{n+1})$ is a finite dimensional space, we have for some $C>0$, depending on $n$ and $y_1,\ldots,y_n$
    \[
    \|\pi^\mathbf{y}(\mathbf{x})_{s,t}^{(j)}\|_j \leq  C \max_{i_1,\ldots,i_j} |\inprod{\pi^\mathbf{y}(\mathbf{x})_{s,t}}{e_{i_1}\otimes\ldots \otimes e_{i_j}}| \leq C' \|x_{s,t}^{(j)}\|_j,
    \]
    so that $\pi^\mathbf{y}(\mathbf{x})$ is indeed a $\alpha$-Hölder continuous rough path and we have $\|\pi^\mathbf{y}(\mathbf{x})\|_{\mathcal{C}^\alpha} \leq C'\|\mathbf{x}\|_{\mathcal{C}^\alpha}$. In particular, $\pi^{\mathbf{y}}$ maps norm-bounded sets in $\mathcal{C}^{\alpha,n}([0,T];E)$ to norm-bounded sets in $\mathcal{C}^{\alpha,n}([0,T];\R^{n+1})$.
    It remains to show that $\pi^\mathbf{y}$ is $i^*$ to $i^*$ continuous and that Equation \eqref{eq: projection aux in prove i* continuity  of lyons lfit} holds; we first prove the continuity statement. As explained in the proof of Lemma~\ref{lem:reduction of i* continuity}, it suffices to show that the maps 
    \[
    \mathbf{x} \mapsto  \inprod{\pi^\mathbf{y}(\mathbf{x})_{s,t}}{e_{i_1}\otimes\ldots \otimes e_{i_j}}
    \]
    are $i^*$ continuous for any combination of basis elements $e_{i_1},\ldots,e_{i_j}$, where $1 \leq j \leq n$. This follows immediately from the definition of $\pi^{\mathbf{y}}(\mathbf{x})$.
    
    The last thing to show is that indeed
    \[
    \inprod{S^{n+1}(\mathbf{x}) }{m_{0,t}^\mathbf{y}} =\inprod{S^{n+1}(\pi^{\mathbf{y}}(\mathbf{x})) }{m_{0,t}^\mathbf{e}},
    \]
    with $\mathbf{e} = e_1 \otimes \ldots \otimes e_{n+1}$. As mentioned in Section \ref{sec: Lyons lift} the Lyons lift $S^{n+1}$ is constructed as follows: one first considers the object $\hat {\mathbf{x}}$ defined by, for $s,t\in [0,T]$,
    \begin{equation}\label{def: xhat}
    \hat{\mathbf{x}}_{s,t} \coloneqq (1, x^1_{s,t}, x^2_{s,t},\ldots, x^n_{s,t},0).
    \end{equation}
    Note that in \cite{lyons_differential_2007}, such objects are called almost rough paths. The Lyons lift is then given by the unique limit in the $T^{n+1}(E)$ norm topology
    \[
    S^{n+1}({\mathbf{x}})_{s,t} = \lim_{|D| \to 0}  \hat{\mathbf{x}}_{s,t_1} \hat\otimes \hat{\mathbf{x}}_{t_1,t_2} \hat\otimes \ldots \hat\otimes \hat{\mathbf{x}}_{t_{k},t},
    \]
    where the limit is over all partitions $D=(s,t_1,\ldots,t_{k},t)$ with mesh size going to zero. A calculation similar to the one in Equation \eqref{eq: rough property of projection} shows that
    \begin{equation}\label{eq: rough projection works nice}
   \inprod{ \hat{\mathbf{x}}_{s,t_1} \hat\otimes \ldots \hat\otimes \hat{\mathbf{x}}_{t_{k},t}}{y_1 \otimes \ldots \otimes y_{n+1}   }= \inprod{    \pi^\mathbf{y}(\hat{\mathbf{x}}_{s,t_1}) \hat\otimes \ldots \hat\otimes \pi^\mathbf{y}(\hat{\mathbf{x}}_{t_{k},t})}{e_1 \otimes \ldots \otimes e_{n+1}}
    \end{equation}
    for any partition $D=(s,t_1,\ldots,t_{k},t)$.
    Furthermore, the maps $\mathbf{x} \mapsto \inprod{\mathbf{x}}{m_{s,t}^{\mathbf{y}}}$ and $\mathbf{x} \mapsto \inprod{\mathbf{x}}{m^\mathbf{e}_{s,t}}$ are continuous with respect to the norm topology and so these maps can be interchanged with the limit. This observation together with Equation \eqref{eq: rough projection works nice}, leads to 
    \[
    \begin{split}
         \inprod{S^{n+1}(\mathbf{x}) }{m_{0,t}^{\mathbf y} }&=  \inprod{S^{n+1}(\mathbf{x})_{0,t} }{\mathbf y }
         \\&
         = 
         \inprod{\lim_{|D|\to 0} \hat{\mathbf{x}}_{0,t_1} \hat\otimes \ldots \hat\otimes \hat{\mathbf{x}}_{t_{k},t}}{y_1 \otimes \ldots \otimes y_{n+1}   } \\&
         = \lim_{|D|\to 0}   \inprod{ \hat{\mathbf{x}}_{0,t_1} \hat\otimes \ldots \hat\otimes \hat{\mathbf{x}}_{t_{k},t}}{y_1 \otimes \ldots \otimes y_{n+1}   } \\ & 
         = \lim_{|D|\to 0}  \inprod{     \widehat{\pi^\mathbf{y}(\mathbf{x}_{0,t_1}}) \hat\otimes \ldots \hat\otimes \widehat {\pi^\mathbf{y}({\mathbf{x}}_{t_{k},t})}}{e_1 \otimes \ldots \otimes e_{n+1}} \\& 
         =  \inprod{\lim_{|D|\to 0}    \widehat{\pi^\mathbf{y}(\mathbf{x}_{0,t_1}}) \hat\otimes \ldots \hat\otimes \widehat{\pi^\mathbf{y}({\mathbf{x}}_{t_{k},t})}}{e_1 \otimes \ldots \otimes e_{n+1}}
         \\& 
         = \inprod{S^{n+1}(\pi^{\mathbf{y}}(\mathbf{x}))_{0,t} }{\mathbf{e} }  \\&
         =\inprod{S^{n+1}(\pi^{\mathbf{y}}(\mathbf{x})) }{m_{0,t}^\mathbf{e}}.
    \end{split}
    \]
For the forth equality above it helps to bear in mind the definition of $\hat{\mathbf{x}}$, see~\eqref{def: xhat} (and note that $\widehat{\pi^{\mathbf{y}}(\mathbf{x}_{s,t})}$ is defined analogously), and the fact that for $\mathbf{x}_1,\ldots,\mathbf{x}_k\in T^{(n+1)}(E)$ one has
\begin{equation}
    (\mathbf{x}_1\hat{\otimes} \ldots \hat{\otimes} \mathbf{x}_k)^{(n+1)}
    =
    \sum_{a\in \N^{k},\, a_1+\ldots+a_k=n+1}
    \mathbf{x}_1^{(a_1)} \hat{\otimes} \ldots \hat{\otimes} \mathbf{x}_k^{(a_k)}.
\end{equation}
This concludes the proof of Lemma~\ref{lem:reduction to finite dimensions}.
\end{proof}

\newpage
\appendix
\section{Tensor spaces}\label{sec: tensor spaces and norms}

Recall from the introduction of Section~\ref{sec: holder norms and rough paths} that iterated integrals of vector-valued paths naturally take values in a tensor space (and thus tensor spaces are the building blocks for rough paths and signatures). The key issue in the infinite-dimensional setting is that there is no tensor norm that renders the algebraic tensor space complete, in particular, there is no canonical choice for a tensor norm: different (in themselves natural) choices lead to different spaces. \par 
In this section we summarize the theory on infinite dimensional tensor spaces needed to understand rough paths and signatures on Banach spaces, mainly based on \cite{hackbusch_tensor_2019}. In particular, in Section~\ref{ssec: algebraic tensor spaces} we introduce the algebraic tensor space. We then list some desirable properties of a tensor norm in Section~\ref{ssec: crossnorms}. In Sections~\ref{ssec: projective and injective} and~\ref{ssec: Hilbert tensors} we discuss three commonly used norms: the projective, injective and Hilbert tensor norms. In the final Section~\ref{ssec: preduals of tensor spaces} we discuss the predual of tensor spaces.

\subsection{Algebraic tensor spaces}\label{ssec: algebraic tensor spaces}
Given two real vector spaces $V$ and $W$, the algebraic tensor space $V\otimes_a W$ is essentially the vector space spanned by pairs $v \otimes w$ (with $v\in V$ and $w \in W$), with the rule that $\otimes$ is a bilinear product, i.e.,
\begin{equation}\label{eq: tensor_bilinearity_product}
    \alpha(v\otimes w) = (\alpha v) \otimes w = v \otimes (\alpha w),
\end{equation}
\begin{equation}\label{eq: tensor_bilinearity_addition}
    (v_1 + v_2) \otimes w = v_1 \otimes w + v_2 \otimes w \quad  \text{ and }\quad v \otimes (w_1 + w_2) = v \otimes w_1 + v \otimes w_2,
\end{equation}
for all $v, v_1, v_2 \in V$, $w, w_1, w_2 \in W$, and $\alpha\in \mathbb{\R}$.
More formally, the algebraic tensor space is defined to be the free vector space over $V \times W$ modulo equivalences arising from~\eqref{eq: tensor_bilinearity_addition} and~\eqref{eq: tensor_bilinearity_product}. Thus, to provide the formal definition we first recall the definition of a free vector space:

The free vector space\footnote{See section 3.1.2 in \cite{hackbusch_tensor_2019} for a more precise definition.} $\mathcal{V}_{\text{free}} (S)$ over a set $S$ consists of the formal linear combinations $\sum_{i=1}^{n} \alpha_i s_i$, with $n\in \N_{>0}$, $\alpha_1,\ldots, \alpha_n \in S$, and distinct $s_1,\ldots,s_n \in S$. In particular, elements of $\mathcal{V}_{\text{free}} (V \times W)$ are the form $\sum_{i=1}^n \alpha_i (v_i, w_i)$, with $n\in \N_{>0}$, $\alpha_1,\ldots,\alpha_n\in \R$, and distinct $(v_1,w_1),\ldots, (v_n,w_n)\in V\times W$ (note however that if $(v_1,w_1),\ldots,(v_n,w_n)$ are \emph{not} distinct, then $\sum_{i=1}^n \alpha_i (v_i, w_i)$ can be identified with an element of $\mathcal{V}_{\text{free}}$ in a canonical manner). 

The subset of $\mathcal{V}_{\text{free}} (V \times W)$ that represents the equivalence class of $0$ with respect to equations~\ref{eq: tensor_bilinearity_addition} and~\eqref{eq: tensor_bilinearity_product} is given by
\begin{equation}\label{eq: definition of zero set in free vector space}
\begin{split}
N \coloneqq \text{span} \bigg\{& \sum_{i=1}^m \sum_{j=1}^n \alpha_i \beta_j (v_i,w_j) - \bigg( \sum_{i=1}^m \alpha_i v_i ,  \sum_{j=1}^n  \beta_j w_j \bigg) \\&\text{ for } m,n\in \N_{>0}, \alpha_i,\beta_j \in \mathbb{R}, v_i \in V \text{ and } w_j\in W \bigg\}.
\end{split}
\end{equation}
This gives rise to the definition of the algebraic tensor space:
\begin{definition}[Algebraic tensor space, see Equation (3.9) in \cite{hackbusch_tensor_2019}] \label{def: algebraic tensor space}
    Let $V,W$ be real vector spaces. The \emph{algebraic tensor space} $V\otimes_a W$ is defined as the quotient vector space
    \[
    V\otimes_a W \coloneqq \mathcal{V}_{\text{free}} (V \times W) / N,
    \]
    where $N$ is defined as in Equation \eqref{eq: definition of zero set in free vector space}. The tensor product of $v\in V$ and $w\in W$, denoted by $v \otimes w$, is defined as the equivalence class of $(v,w)$. Elements of this form are called \emph{elementary tensors}.
\end{definition}

It is immediate from the definition that elementary tensors span $V\otimes_a W$. Having defined the algebraic tensor space $V\otimes_a W$, we can define $\prescript{n}{i=1}{\bigotimes}_{a} V_i$ by induction:
    \[
    \prescript{n}{i=1}{\bigotimes}_{a} V_i \coloneqq \prescript{n-1}{i=1}{\bigotimes}_{a} V_{i} \otimes_a V_n.
    \]
For brevity, we frequently use the notation $V^{\otimes_a n} := \prescript{n}{i=1}{\bigotimes}_{a} V$.
By the following proposition, the order in which we inductively define our multiple product tensor space is irrelevant.
    \begin{proposition}[Associativity of the algebraic tensor space, see Lemma 3.21 in \cite{hackbusch_tensor_2019}]
        Let $V_1,V_2,V_3$ be vector spaces. Then, we have a natural isomorphism
        \[
        V_1 \otimes_a (V_2 \otimes_a V_3) \to (V_1 \otimes_a V_2) \otimes_a V_3,
        \]
        such that
        \[
        v_1 \otimes_a (v_2 \otimes_a v_3) \mapsto (v_1 \otimes_a v_2) \otimes_a v_3.
        \]
    \end{proposition}
Equivalently, one can define $\prescript{n}{i=1}{\bigotimes}_{a} V_i$ as the free vector space over $(V_1 \times\ldots\times V_n)$ modulo equivalences given by the multilinear product.

    A key feature of the tensor space is its universality with respect to linear mappings (in fact, the algebraic tensor space is the only space, up to isomorphisms, that satisfies this universality):
\begin{proposition}[Universality of the tensor space, see Proposition 3.23 in \cite{hackbusch_tensor_2019}]\label{prop: universality of tensor space d geq 3}
   Let $\{V_i\}_{1\leq i \leq n}$ and $U$ be vector spaces. Then, for any multilinear mapping $\phi\colon(V_1 \times\ldots\times V_n) \to U$, there exists a unique linear mapping $\Phi\colon \prescript{n}{i=1}{\bigotimes}_{a} V_i\to U$ such that
   \[
   \phi ((v_1,\ldots, v_n)) = \Phi(v_1 \otimes\ldots\otimes v_n),
   \]
   for any $(v_1,\ldots, v_n)\in (V_1 \times\ldots\times V_n)$. 
\end{proposition}

\subsection{Topological tensor spaces}\label{ssec: crossnorms}

As mentioned in the introduction of this section, there is no canoncial way to equip an infinite-dimensional algebraic tensor space $\prescript{n}{j=1}{\bigotimes}_{a} V_j$ with a norm: multiple choices are possible, each giving rise to distinct \emph{topological tensor spaces}. We begin by introducing the concept of a topological tensor space and discuss some desirable properties of the associated norm. We follow the approach in \cite{hackbusch_tensor_2019}.

\begin{definition}[topological tensor space]\label{def: topological tensor space}
     Let $V_j$, $1\leq j \leq n$, be real vector spaces and let $\|\cdot\|$ be a norm on $\prescript{n}{j=1}{\bigotimes}_{a} V_j$. The \emph{topological tensor space} $\prescript{n}{j=1}{\bigotimes}_{\|\cdot \|} V_j$ is defined as the $\|\cdot\|$-closure of the algebraic tensor space:
     \[
     \prescript{n}{j=1}{\bigotimes}_{\|\cdot \|} V_j = \overline{\prescript{n}{j=1}{\bigotimes}_{a} V_j}^{\|\cdot\|}.
     \]
\end{definition}

When considering a norm on an algebraic tensor space of Banach spaces, it may be desirable for the norm on the product to somehow align with the norms on the individual spaces. The \emph{crossnorm} and \emph{reasonable crossnorm} properties are ways to describe such alignment:

\begin{definition}[Crossnorm, see Equation (4.42) in \cite{hackbusch_tensor_2019}]\label{def: crossnorm d geq 3}
    Let $(V_j,\|\cdot \|_{V_j})$, $1\leq j \leq n$, be real Banach spaces. A norm $\|\cdot \|$ on $\prescript{n}{j=1}{\bigotimes}_{a} V_j$ is called a \emph{crossnorm} if 
    \[
    \|v_1 \otimes \ldots \otimes v_n \| = \prod_{j=1}^n \|v_j\|_{V_j},\quad 
    \text{for all }v_1\in V_1,\ldots,v_n\in V_n.
    \] 
\end{definition}

Let $(V_j,\|\cdot \|_{V_j})$, $1\leq j \leq n$, be real Banach spaces and let $\phi_j^*\in V_j^*$, $1\leq j \leq n$. By Proposition~\ref{prop: universality of tensor space d geq 3} we can identify $\phi_1^*\otimes \ldots \otimes \phi_n^*$ with a linear functional on $\prescript{n}{i=1}{\bigotimes}_{a} V_i$ satisfying 
\begin{equation} 
(\phi_1^*\otimes \ldots \otimes \phi_n^*)(v_1 \otimes \ldots \otimes v_n) = \prod_{j=1}^{n} \phi_j^*(v_j),\quad v_1\in V_1,\ldots,v_n\in V_n.
\end{equation} 
Given a norm $\|\cdot \|$ on $\prescript{n}{j=1}{\bigotimes}_{a} V_j$, we denote the (topological) dual of $(\prescript{n}{j=1}{\bigotimes}_{\|\cdot \|} V_j , \| \cdot \|)$ by $((\prescript{n}{j=1}{\bigotimes}_{\|\cdot \|} V_j)^* , \| \cdot \|^*)$. We can now provide the definition of a reasonable crossnorm:

\begin{definition}[Reasonable crossnorm, see Equation (4.44) in \cite{hackbusch_tensor_2019}]\label{def: reasonable crossnorm d geq 3}
     Let $(V_j,\|\cdot \|_{V_j})$, $1\leq j \leq n$, be real Banach spaces. We call crossnorm $\|\cdot \|$ on $\prescript{n}{j=1}{\bigotimes}_{a} V_j$ \emph{reasonable} if it satisfies 
    \begin{equation}\label{eq: reasonable crossnorm def}
    \|\phi_1^*\otimes \ldots \otimes\phi_n^* \|^* = \prod_{j=1}^n \|\phi_j^*\|_{V_j^*}, \quad \text{for all } \phi_1^*\in V_1^*, \ldots,\phi_n^*\in V_n^*.
    \end{equation}
\end{definition}

In the context of signatures, the crossnorm property does not suffice to capture all desired behaviour of the tensor norm. Indeed, one aspect of tensor algebras (see Definition~\ref{def: tensor algebra}) is that it must be possible to tensor an $x_1 \in \prescript{m}{j=1}{\bigotimes}_{\| \cdot \|_m} E $ with an $x_2 \in \prescript{n}{j=1}{\bigotimes}_{\| \cdot \|_n} E$ ($n,m\in \N_{>0}$) to obtain an $x_1 \otimes x_2 \in \prescript{m+n}{j=1}{\bigotimes}_{\| \cdot \|_{m+n}}  E $ (i.e., there exists an $x\in \prescript{m+n}{j=1}{\bigotimes}_{\| \cdot \|_{m+n}}  E$ such that for any sequence $\{ x_{1,i}\}_{i\in \N}$ in $\prescript{m}{j=1}{\bigotimes}_{a} E$ that approximates $x_1$ in $\| \cdot \|_{m}$ and any sequence $\{ x_{2,i}\}_{i\in \N}$ in $\prescript{n}{j=1}{\bigotimes}_{a} E$ that approximates $x_2$ in $\| \cdot \|_n$ one has that $\{ x_{1,n} \otimes x_{2,n} \}_{n\in \N}$ approximates $x$ in $\| \cdot \|_{n+m}$, and we set $x_1 \otimes x_2=x$). This is achieved if we assume the following generalization of the crossnorm property:
\begin{equation}\label{eq: strong crossnorm property}
\| x_1 \otimes x_2 \|_{m+n} = \| x_1 \|_{m}\, \| x_2\|_{n}.
\end{equation}

These desired properties are captured by the following definition:

\begin{definition}[Strong crossnorm family, see e.g.\ Equation (4.48c) in \cite{hackbusch_tensor_2019}]\label{def: strong crossnorm}
    Let $E$ be a real Banach spaces and for $n\in \N_{>0}$ let $\|\cdot \|_n$ be a norm on $ \prescript{n}{j=1}{\bigotimes}_{a}^{} E$. Set $E^{\otimes n} = \prescript{n}{j=1}{\bigotimes}_{\| \cdot \|_n}^{} E$. We call $\{\| \cdot \|_n\}_{n\in \N_{>0}}$ a \emph{strong crossnorm family (on $\{E^{\otimes n}\}_{n\in \N_{>0}}$)} if for all $n,m\in \N_{>0}$, $x_1\in E^{\otimes m}$, $x_2\in E^{\otimes n}$ we have $x_1 \otimes x_2\in E^{\otimes (n+m)} $
    and
    \begin{equation}\label{eq:strong crossnorm property}
    \|x_1 \otimes x_2\|_{m+n} = \| x_1 \|_n \| x_2\|_{m}.
    \end{equation}
\end{definition}

\begin{remark}\label{rem: strong crossnorm family}
The definition in~\cite{hackbusch_tensor_2019} does not concern an infinite family of norms. When one is only interested in norms on $ \prescript{n}{j=1}{\bigotimes}_{a}^{} E$ for $n\in \{1,\ldots,N\}$, then the norm on $ \prescript{N}{j=1}{\bigotimes}_{a}^{} E$ naturally induces norms on the tensor spaces of lower order, see \cite[Theorem 4.107]{hackbusch_tensor_2019}.
\end{remark}

\begin{remark}
Let $E$ be a real Banach space and let $\{\| \cdot \|_n\}_{n\in \N_{>0}}$ be a strong crossnorm family on $\{E^{\otimes n}\}_{n\in \N_{>0}}$.  For all $n,m\in \N_{>0}$ the norm $\| \cdot \|_{n+m}$ induces a natural norm on $E^{\otimes m} \otimes_a E^{\otimes n}$ that we again denote by $\| \cdot \|_{n+m}$. The strong crossnorm property implies that
\begin{equation}\label{eq: strong crossnorm property for the spaces}
E^{\otimes m} \otimes_{\|\cdot\|_{n+m}} E^{\otimes n} = E^{\otimes (n+m)}.
\end{equation}
\end{remark}
\begin{remark}\label{rem: note on lyons definition of admissible}
    In the definition of an admissible family of tensor norms (see \ref{def: admissible crossnorms}), we require that the tensor norms are strong crossnorms. Instead of requiring an equality in \eqref{eq: strong crossnorm property}, one can demand that the left-hand side is only \emph{less or equal} to the right-hand side, as is done in e.g. \cite[Definition 1.25]{lyons_differential_2007}. However, we have chosen to align the definition with the literature on topological tensor spaces. We believe this choice to be harmless. Indeed, the weaker definition of admissible tensor norms of \cite[Definition 1.25]{lyons_differential_2007}, together with e.g. the assumptions of Proposition \ref{prop: when there is a time component then unique signature} or Setting \ref{assumptions} ensure that the strong crossnorm property is satisfied and thus that both definitions of admissible tensor norms align. Similarly, if one defines strongly uniform crossnorms using the weaker assumptions of \cite[Definition 1.25]{lyons_differential_2007}, see Definition \ref{def:strong-unif-crossnorm}, one ends up only with scalar multiples of strong crossnorms. Nevertheless, for the results that \emph{only} require admissible tensor norms, the weaker definition of \cite[Definition 1.25]{lyons_differential_2007} can be used.    
\end{remark}

A useful property of strong crossnorms is that the algebraic tensor product of dense subspaces is dense in the tensor space:
\begin{proposition}\label{prop: density in tensor space higher orders}
    Let $E$ be a real Banach space and let 
    $\{\| \cdot \|_n\}_{n\in \N_{>0}}$ be a strong crossnorm family on 
    $\{E^{\otimes n}\}_{n\in \N_{>0}}$. 
    Let $E'\subseteq E$ be a dense subspace of $E$. Then 
    $\prescript{n}{j=1}{\bigotimes}_{a} E'$ 
    is dense in $E^{\otimes n} $.
\end{proposition}
\begin{proof}
    We prove the statement by induction. The base case $n=1$ is true by definition. Assume it is true for some $n$. By Equation \eqref{eq: strong crossnorm property for the spaces}, we can write
    \[
    E^{\otimes (n+1)} = \left(\prescript{n}{j=1}{\bigotimes}_{\|\cdot\|_{n}} E \right) 
    \otimes_{\|\cdot\|_{n+1}} E.
    \]
    We can now apply~\cite[Lemma 4.40]{hackbusch_tensor_2019} to conclude that $\prescript{n+1}{j=1}{\bigotimes}_{a} E'$ is dense in $E^{\otimes n}$.
\end{proof}

Another natural assumption in the context of signatures is that of symmetry. To introduce this concept, we first define the permutation operators on $\prescript{n}{j=1}{\bigotimes}_{a} V$ (with $V$ a vector space): given a permutation $\sigma$ on $\{1,\ldots,n\}$, we can define a multilinear map $P_{\sigma} \colon V \times \ldots \times V \rightarrow \prescript{n}{j=1}{\bigotimes}_{a} V$, $P_{\sigma}(v_1,\ldots,v_n)=v_{\sigma{(1)}}\otimes \ldots \otimes v_{\sigma(n)}$. By the universality of the tensor space (Proposition \ref{prop: universality of tensor space d geq 3}), we can extend $P_{\sigma}$ to a map from $\prescript{n}{j=1}{\bigotimes}_{a} V$ to itself, which we denote by $P_\sigma$:

\begin{definition}\label{def: permutations on alg tensor space}
    Let $V$ be a real Banach  space, let $n\in\N_{>0}$ and let $\sigma $ be a permutation of $n$. We define the \emph{permutation map} $P_\sigma\colon \prescript{n}{j=1}{\bigotimes}_{a} V \rightarrow \prescript{n}{j=1}{\bigotimes}_{a} V$ to be the unique linear map such that 
    \[
    P_\sigma (v_1 \otimes \ldots\otimes  v_n) = v_{\sigma(1)} \otimes \ldots \otimes v_{\sigma(n)}, \quad v_1,\ldots,v_n\in V.
    \]
\end{definition}

\begin{definition}\label{def: symmetric norms}
    Let $V$ be a real Banach  space. A tensor norm $\| \cdot \|$ on $\prescript{n}{j=1}{\bigotimes}_{a} V$ is said to be \emph{symmetric} if 
    $
    \|P_\sigma (v) \| = \|v\|
    $
    for all $v\in \prescript{n}{j=1}{\bigotimes}_{a} V$ and every permutation $\sigma$ of $\{1, \ldots, n\}$, where $P_{\sigma}$ is the permutation map.
\end{definition}

In the context of signatures, it generally suffices to assume that one has a strong crossnorm family of reasonable and symmetric crossnorms. However, to obtain~\ref{prop: the dual tensor algebra can actually separate points}, we in fact need \emph{strongly uniform crossnorms}:
\begin{definition}[Strongly uniform crossnorm, see e.g.\ Definition 4.110 in \cite{hackbusch_tensor_2019}]\label{def:strong-unif-crossnorm}
    Let $E$ be a real Banach space and let 
    $\{\| \cdot \|_n\}_{n\in \N_{>0}}$ be a strong crossnorm family on 
    $\{E^{\otimes n}\}_{n\in \N_{>0}}$.  We call  $\{\| \cdot \|_n\}_{n\in \N_{>0}}$ a \emph{strongly uniform crossnorm family (on $\{E^{\otimes n}\}_{n\in \N_{>0}}$)} if for all $n,m \in \N_{>0}$ and all $A_1\in \mathcal{L}(E^{\otimes m})$, $A_2 \in\mathcal{L}(E^{\otimes n}) $ we have 
    \[
    \| A_1 \otimes A_2 \|_{\mathcal{L}(E^{\otimes (m+n)})} = \|A_1 \|_{\mathcal{L}(E^{\otimes m})} \, \|A_2\|_{\mathcal{L}(E^{\otimes n})}.
    \]
\end{definition}

\begin{remark}\label{remark: strongly uniform means reasonable}
    Let $E$ be a real Banach space and let 
    $\{\| \cdot \|_n\}_{n\in \N_{>0}}$ be a strongly uniform crossnorm family on $\{E^{\otimes n}\}_{n\in \N_{>0}}$. Then $\|\cdot \|_{n}$ is a reasonable crossnorm for all $n\in \N_{>0}$. See also \cite[Lemma 4.95a]{hackbusch_tensor_2019}.
\end{remark}

\subsection{Projective and injective tensor norms}\label{ssec: projective and injective}
The crossnorm and reasonable crossnorm properties say something about elementary tensors, but as a consequence they also enforce a bound on general elements of the tensor space.

To show this, take a crossnorm (in the sense of Definition \ref{def: crossnorm d geq 3}) $\|\cdot\|$ on $\prescript{j=1}{n}{\bigotimes}_{\|\cdot\|} V_j$, and fix an element $\mathbf{v}$ in the algebraic tensor space. Whenever $\mathbf{v}$ can be written as a sum of elementary tensors, as in 
\[
\mathbf{v} = \sum^k_{i=1} v_{i_1} \otimes \ldots \otimes v_{i_n},
\]
we can combine the crossnorm property together with the triangle inequality to find that
\[
\| \mathbf{v}\| \leq \sum^k_{i=1} \|v_{i_1}\| \ldots \|v_{i_n}\|. 
\]
In other words, we have for any $v\in \prescript{n}{j=1}{\bigotimes}_{a} V_j$ that 
\[
    \| \mathbf{v}\| \leq \inf \left\{  \sum^k_{i=1} \|v_{i_1}\|  \ldots  \|v_{i_n}\| \quad \Big | \quad \mathbf{v} = \sum^k_{i=1} v_{i_1} \otimes \ldots \otimes v_{i_n}, \,k\in \N_{> 0}\right \} .
\]
When $\|\cdot\|$ is in addition a reasonable crossnorm, we can make another observation: whenever $\phi_j\in V_j^*$ with $\|\phi_j\| =1$, we have 
\[
\|(\phi_1 \otimes \ldots \otimes \phi_n) \mathbf{v} \| \leq \|\mathbf{v}\|.
\]
In other words, we have
\begin{equation}\label{eq: injective tensor norm predef}
\|\mathbf{v}\| \geq \sup \left \{ \|(\phi_1 \otimes \ldots \otimes \phi_n) \mathbf{v} \| \quad \big | \quad \phi_j \in V_j^* \text{ for all } 1\leq j \leq n \text{ and } \|\phi_j\| =1 \right\}.
\end{equation}
These two observations provide the intuition behind the projective and injective tensor norms, which are respectively the `largest' and the `smallest' reasonable crossnorms.
\begin{definition}[Projective tensor norm, see Remark 4.96 in \cite{hackbusch_tensor_2019}]\label{def: projective tensor norm}
   Let $\{V_j\}_{1\leq j \leq n}$ be real Banach spaces. The \emph{projective tensor norm} $\|\cdot\|_\pi$ on $\prescript{n}{j=1}{\bigotimes}_{a} V_j$ (also denoted by $\|\cdot\|_\wedge$ in the literature) is defined by 
   \[
    \| \mathbf{v}\|_\pi \coloneqq\inf \left\{  \sum^k_{i=1} \|v_{i_1}\|  \ldots  \|v_{i_n}\| \quad \Big | \quad \mathbf{v} = \sum^k_{i=1} v_{i_1} \otimes \ldots \otimes  v_{i_n}, \,k\in \N_{>0}\right \}.
   \]
   The topological tensor space $\prescript{n}{j=1}{\bigotimes}_{\|\cdot\|_{\pi}} V_j$ is called the projective tensor space and is also denoted by $\prescript{n}{j=1}{\bigotimes}_{\pi} V_j$ .
\end{definition}
\begin{definition}[Injective tensor norm, see Remark 4.98 in \cite{hackbusch_tensor_2019}]\label{def: injective tensor norm}
    Let $\{V_j\}_{1\leq j \leq n}$ be Banach spaces. The \emph{injective tensor norm} $\|\cdot\|_\epsilon$  on $\prescript{n}{j=1}{\bigotimes}_{a} V_j$ (also denoted by $\|\cdot\|_\vee$ in the literature) is defined by 
    \[
        \|\mathbf{v}\|_{\epsilon} \coloneqq \sup \left \{ \|(\phi_1 \otimes \ldots \otimes \phi_n) \mathbf{v} \| \quad \big | \quad \phi_j \in V_j^* \text{ for all } 1\leq j \leq n \text{ and } \|\phi_j\| =1 \right\}.
    \]
    The topological tensor space $\prescript{n}{j=1}{\bigotimes}_{\|\cdot\|_{\epsilon}} V_j$ is called the injective tensor space and is also denoted by $\prescript{n}{j=1}{\bigotimes}_{\epsilon} V_j$.  
\end{definition}
\begin{proposition}\label{prop: properties of projective and injective}
    Let $E$ be a real Banach space and let $\| \cdot \|_n$ be the projective tensor norm on $\prescript{n}{j=1}{\bigotimes}_{a}E$, $n\in \N_{>0}$. Then $\{ \| \cdot  \|_{n} \}_{n\in \N_{>0}}$ is a symmetric strongly uniform crossnorm family. The analogous statement holds for the injective tensor norm.
\end{proposition}
\begin{proof}
    Apart from the symmetry property, this is \cite[Theorem 4.111]{hackbusch_tensor_2019}, with the exception that they consider the norms on the spaces $\prescript{n}{j=1}{\bigotimes}_{a}E$ to be the induced norms, rather a norm that was set a priori (see also Remark~\ref{rem: strong crossnorm family}). However, in~\cite{hackbusch_tensor_2019} it is shown that all induced norms on $\prescript{n}{j=1}{\bigotimes}_{a}E$, $n\in \N_{>0}$, are projective [respectively, injective] tensor norms, so that results still hold with the adapted definition. The symmetry property follows by the fact the definition of the projective [respectively, injective] tensor norm is symmetric.
\end{proof}
\subsection{Hilbert tensor spaces}\label{ssec: Hilbert tensors}
When we are working with tensor spaces over Hilbert spaces, we can define the so-called Hilbert tensor norm. This norm carries the inner product of the Hilbert spaces over to the tensor space, so that it too can be made into a Hilbert space. 

Let $H_j$ be real Hilbert spaces for $1 \leq j \leq n$. We define a product $\inprod{.}{.} $ on the elementary tensors of $\prescript{n}{j=1}{\bigotimes}_{a} H_j$ by setting \begin{equation}\label{eq: def hilb prod}
\inprod{v_1\otimes \ldots \otimes v_n}{w_1 \otimes \ldots \otimes w_n} = \prod_{j=1}^n \inprod{v_j}{w_j},
\end{equation}
with $v_j,w_j \in H_j$, for all $1 \leq j \leq n$. This product is a sesquilinear form, and so we can extend it to the entirety of $\prescript{n}{j=1}{\bigotimes}_{a} H_j$. One can also show that this product is an inner product:
\begin{proposition}[See Lemma 4.147 if \cite{hackbusch_tensor_2019}]\label{prop: Hilb is inner prod}
    Let $\{H_j\}_{1\leq i \leq n}$ be real Hilbert spaces. The unique sesquilinear map $\prescript{n}{j=1}{\bigotimes}_{a} H_j \times \prescript{n}{j=1}{\bigotimes}_{a} H_j \to \mathbb{K}$ coming from the extension of Equation \eqref{eq: def hilb prod} is an inner product.
\end{proposition}
\begin{definition}[Hilbert tensor space]
    Let $\{H_j\}_{1\leq i \leq n}$ be real Hilbert spaces. The \emph{Hilbert tensor space} $\prescript{n}{j=1}{\bigotimes}_{2} H_j$ is the topological vector space under the norm generated by the inner product of Proposition \ref{prop: Hilb is inner prod}. In particular, the Hilbert tensor space is itself another real Hilbert space.
\end{definition}
There is a natural construction for an orthonormal basis on a Hilbert tensor space given orthonormal bases on the individual Hilbert spaces:
\begin{proposition}[See Remark 4.148 in \cite{hackbusch_tensor_2019}]\label{prop: hilbert tensor onb}
    Let $\{H_i\}_{1\leq i \leq n}$ be real Hilbert spaces, and let $\{e_{i,j}\}_{j\in I_i}$ be an orthonormal basis for $H_i$, $1\leq i \leq n$. Then, 
    \[
    \{ e_{1,j_1} \otimes \ldots \otimes e_{n,j_n} \mid j_i\in I_i \text{ for } 1\leq i \leq n \} 
    \]
    is an orthonormal basis for the Hilbert tensor space $\prescript{n}{i=1}{\bigotimes}_{2} H_i$.
\end{proposition}
The Hilbert tensor space has all the properties we like to see from our topological tensor spaces:
\begin{proposition}\label{prop: properties hilbert norm}
    The Hilbert tensor norm is strongly uniform crossnorm and is symmetric.
\end{proposition}
\begin{proof}
    See \cite[Proposition 4.150]{hackbusch_tensor_2019}, in combination with the observation made in step (v) of the proof. The symmetry of the tensor norm follows by definition.
\end{proof}

\begin{remark}
Note that one can also define the injective and projective tensor norms on Hilbert tensor products, neither of which correspond with the Hilbert tensor norm.
\end{remark}
\subsection{Predual of tensor spaces}\label{ssec: preduals of tensor spaces}

Let $E$ be a real Banach space with predual $F$, and let $\|\cdot\|$ be a reasonable crossnorm on $\prescript{n}{j=1}{\bigotimes}_{a} E$. We have a natural embedding $F\subset F^{**} = E^*$, and so whenever $y_i \in F$ for all $1\leq i \leq n$, we can interpret $y_1 \otimes \ldots \otimes y_n$ as an element in $\left(\prescript{n}{j=1}{\bigotimes}_{\|\cdot\|} E\right)^*$.

To prove our universal approximation theorem, we have to endow the topological tensor space $E^{\otimes n}$ with another topology than the norm topology such that norm bounded subsets become compact, but that linear evaluations of the form $y_1 \otimes \ldots \otimes y_n$ with $y_i \in F$ are still continuous. We do this by taking a weak-$^*$ topology. This requires the tensor spaces we are working with to have a predual. 

In the Hilbert tensor space setting the existence of a predual is straightforward by the Riesz representation theorem, as Hilbert tensor spaces are themselves Hilbert spaces. For the projective tensor norm, it is a bit more complicated -- in fact, Proposition~\ref{prop: no good topology} shows that what we seek is not entirely trivial. However, if a Banach space  $E$ has the approximation property, the Radon Nikodým property, and a predual $F$ (see Definitions~\ref{def: approximation property} and~\ref{def: Radon-nikodym} below), then $(F \otimes_\epsilon F)^* = E\otimes_\pi E$, see Proposition~\ref{prop: Radon--Nikodym} below.
\begin{definition}[Approximation property, see Proposition 4.1 in \cite{ryan_introduction_2002}]\label{def: approximation property}
    Let $E$ be a real Banach space. Then, $E$ has\emph{ the approximation property} if for every compact $K\subset X$ and every $\epsilon > 0$, there exists a finite rank operator $S\colon X \to X$ such that for every $x\in K$
    \[
    \|x - Sx\| < \epsilon.
    \]
\end{definition}

\begin{definition}[Radon--Nikodým property, see Section 5.1 of \cite{ryan_introduction_2002}]\label{def: Radon-nikodym}
    Let $E$ be a real Banach space. Then, $E$ has the \emph{Radon--Nikodým property} if for every finite measure $\mu$ and every bounded linear operator $T: L_1(\Omega,\mu)\to E$ there exists a bounded $\mu$-measurable function $g\colon \Omega \to E$ such that
    \[
    Tf = \int fg d\mu,\quad f\in L_1(\Omega,\mu).
    \]
\end{definition}
Most of the commonly used Banach spaces satisfy the approximation property; the first counterexample was found only in 1972 by Per Enflo. The Radon--Nikodÿm property is less common, as for example $L^1$ spaces do not have this property. However, there is still a relatively large class of spaces that do meet this requirement, as is illustrated by the following proposition.    
\begin{proposition}\label{prop: when Radon--Niko}[See Corollary 5.42 in \cite{ryan_introduction_2002}]
    Every separable space with a predual has the Radon--Nikodým property.
\end{proposition}

\begin{proposition}\label{prop: Radon--Nikodym}
    Let $E$ be a real Banach space with the Radon--Nikodým property and the approximation property, assume $E$ has a predual $F$, and let $n\in \N_{>0}$. Then, 
    \[
    \prescript{n}{j=1}{\bigotimes}_{\pi} E= \left(\prescript{n}{j=1}{\bigotimes}_{\epsilon} F\right)^{*}.
    \]
    In particular, linear functionals of the form $y_1 \otimes \ldots \otimes y_n$, with $y_i\in F $ for all $1\leq i \leq n$, are continuous with respect to the weak$^*$ topology of $\prescript{n}{j=1}{\bigotimes}_{\pi} E$.
\end{proposition}
\begin{proof}
    This follows by inductively applying \cite[Theorem 5.33]{ryan_introduction_2002}, which states that $(X\otimes_\epsilon Y)^* = X^* \otimes_\pi Y^*$ if $X^*$ has the Radon--Nikodým property and either $X^*$ or $Y^*$ has the approximation property. For the base step we can directly apply \cite[Theorem 5.33]{ryan_introduction_2002} to get
    \[
    E \otimes_\pi E = \left( F \otimes_\epsilon F\right)^*.
    \]
    Then, assuming for $i\in \N_{>0}$ we have $ \prescript{i}{j=1}{\bigotimes}_{\pi} E= \left(\prescript{n}{j=1}{\bigotimes}_{\epsilon} F\right)^{*}$, we get
    \[
     \prescript{i+1}{j=1}{\bigotimes}_{\pi} E=  E \otimes_\pi \prescript{i}{j=1}{\bigotimes}_{\pi} E  =\left(F \otimes_\epsilon \prescript{i}{j=1}{\bigotimes}_{\epsilon} F \right)^* =\left(\prescript{i+1}{j=1}{\bigotimes}_{\epsilon} F\right)^{*},
    \]
    where in the middle equality, we could apply this Theorem, because $F^*=E$ has both the Radon--Nikodým property and the approximation property.
\end{proof}
As announced above, we now show that what we seek cannot always be obtained:
\begin{proposition}\label{prop: no good topology}
    Let $H$ be a real Hilbert space, let $\{e_i\}_{i\in\N}$ be an orthonormal basis for $H$. Then, there is no topology on $H\otimes_\epsilon H$ for which the norm unit ball is relatively compact and such that the linear operators $e_i\otimes e_j$ are continuous for all $i,j\in \N$.
\end{proposition}
\begin{proof}
    Consider the sequence $\{x_n\}_{n\in\N}$ in $H\otimes_a H$ given by
    \[
    x_n = \sum_{i=1}^n e_i \otimes e_i .
    \]
    By H\"older's inequality, we have for $h_1,h_2\in H$ with $\|h_1\|=\|h_2\|=1$ that
    \[
   | (h_1 \otimes h_2)\, (x_n)| = |\sum_{i=1}^n \inprod{e_i}{h_1} \inprod{e_i}{h_2}| \leq  \left(\sum_{i=1}^n |\inprod{e_i}{h_1}|^2\right) ^{\frac{1}{2}} \left(\sum_{i=1}^n |\inprod{e_i}{h_2}|^2\right) ^{\frac{1}{2}} \leq \| h_1 \| \, \|h_2 \| = 1.
    \]
    Therefore, by definition of the injective tensor norm and the fact that $(e_1\otimes e_1)x_n=1$, we have $\|x_n\|_\epsilon =1$.
    
    Assume now, that there exists a topology, such that $\{ x_n \}_{n\in \N}$ has a subsequence converging to some $x\in H\otimes_\epsilon H$, and that the linear operators $e_i\otimes e_j$ are continuous. Because the $e_i\otimes e_j$ are continuous, we obtain that $e_i \otimes e_j (x) = \delta_{ij}$ for any $i,j\in \N$.
    
    By \cite[Lemma 4.40]{weaver_lipschitz_2018}, the span of $\{e_i\otimes e_j\}_{i,j\in\N}$ is dense in $H\otimes_\epsilon H$. However, for any element $y$ in this span, there exists $i\in \N$ such that $(e_i \otimes e_i) \, (y) = 0$, and so \[
    (e_i \otimes e_i) (x-y) = 1.
    \]
    This implies that $\|x-y\|_\varepsilon \geq 1$ for any $y$ in the span of $\{e_i\otimes e_j\}_{i,j\in\N}$, so that $x$ cannot be in $H\otimes_\epsilon H$. This is a contradiction, and so the proposition is proven.
\end{proof}

\section{The predual}
\begin{lemma}\label{lem: lower semicontinouity of norm}
    Let $E$ be a real Banach space with predual $F$. Let $\{x_\lambda\}_{\lambda \in I}$ be a net in $E$ that converges to $x\in E$ in the weak$^*$ topology. Then $\|x\| \leq \liminf_\lambda \|x_\lambda\|$.
\end{lemma}
\begin{proof}

As $F$ is norming for $E$ and due to the weak-$^*$ convergence, we have
    \[
    \|x\| =  \sup_{y\in F, ||y||^*=1} \inprod{x}{y} = \sup_{y\in , ||y||^*=1F} \lim_{\lambda \in I}\inprod{x_\lambda}{y} \leq \liminf_{\lambda \in I} \sup_{y\in F, ||y||^*=1} \inprod{x_\lambda}{y} = \liminf_{\lambda \in I} \|x_\lambda\|.
    \]
\end{proof}
\begin{lemma}\label{lem: density implies generate predual}
    Let $E$ be a real Banach space with predual $F$. Let $\{x_\lambda\}_{\lambda \in I}$ be a norm bounded net in $E$, and let $x\in X$. Additionally, assume there exists a (norm) dense set $F'\subset F$, such that for all $y \in F'$
    \[
    \inprod{x_\lambda}{y} \to \inprod{x}{y}.
    \]
    Then $x_\lambda \wkst x$. 
\end{lemma}
\begin{proof}
    Take an arbitrary element $y\in F$. We want to show
    \[
    \lim_{\lambda \in I} \inprod{x_\lambda}{y} = \inprod{x}{y}.
    \]
    Let $\epsilon>0$. By assumption, there exists $y' \in F'$ such that $\|y- y'\| < \frac{\epsilon}{3M}$, with $M\coloneqq \sup_\lambda \|x_\lambda\|$. Take $\lambda'$ such that for all $\lambda \geq  \lambda'$ we have $|\inprod{x_\lambda}{y'} - \inprod{x}{y'}| < \frac{\epsilon}{3}$. Then, for any $\lambda \geq \lambda '$, we have
    \[
    \begin{split}
        |\inprod{x_\lambda}{y} - \inprod{x}{y}| \leq &|\inprod{x_\lambda}{y}-\inprod{x_\lambda}{y'}| + |\inprod{x_\lambda}{y'} - \inprod{x_}{y'}| + |\inprod{x}{y'} - \inprod{x_\lambda}{y}| \\ \leq & \|y-y'\|^*\, \|x_\lambda\| + \frac{\epsilon}{3} + \|y-y'\|^*\, \|x\| \leq \epsilon,
    \end{split}
    \]
    where in the last step we used Lemma \ref{lem: lower semicontinouity of norm}. As $\epsilon$ was arbitrary, this shows, that
    \[
    \lim_{\lambda \in I} \inprod{x_\lambda}{y} = \inprod{x}{y},
    \]
    and so we are done.
\end{proof}
\begin{lemma}\label{lem: weak star convergence tensor product}
    Assume Setting \ref{assumptions}. Let $\{x_\lambda\}_{\lambda \in I}$ be a norm bounded net in $E^{\otimes i}$ converging weak-$^*$ to $x$, and $\{x'_\lambda\}_{\lambda \in I}$ be a norm bounded net in $E^{\otimes j}$ converging to $x'$ in the weak-$^*$ topology, for some $i,j\in \N_{>0}$. Then $x_\lambda \otimes x'_\lambda \wkst x \otimes x_\lambda$. 
\end{lemma}
\begin{proof}
    By the assumptions of the setting, the predual of $E^{\otimes (i+j)}$ is $F^{\otimes (i+j)}$, of which $F^{\otimes i} \otimes F^{\otimes j}$ is a dense subspace. Therefore, by Lemma \ref{lem: density implies generate predual} it is sufficient to show that 
    \[
    \inprod{x_\lambda \otimes x'_\lambda}{y\otimes y'} \to \inprod{x \otimes x'}{y\otimes y'},
    \]
    for any $y\in F^{\otimes i}$ and $y' \in F^{\otimes j}$. However, this is essentially by assumption as
    \[
     \inprod{x_\lambda \otimes x'_\lambda}{y\otimes y'} = \inprod{x_\lambda}{y} \inprod{x'_\lambda}{y'} \to  \inprod{x}{y} \inprod{x'}{y'} =  \inprod{x \otimes x'}{y\otimes y'}.
    \]
\end{proof}

\section{Point separation by the tensor dual}
The Stone--Weierstrass theorem is the foundation for the universal approximation theorems proven in Section~\ref{sec: UATs}. A crucial assumption in the Stone--Weierstrass theorem is that the set of approximating functions separates points, which, in the context of our universal approximation theorems, results in the requirement that a certain set of linear functionals separates points of $T((E))$. In this section we collect some relevant results on point separation in this context.

\begin{lemma}\label{lem : point separating dense is enough}
    Let $E$ be a real Banach space, let $F \subset E^*$ separate points in $E$, and let $G$ be dense in $F$. Then $G$ separates points in $E$. 
\end{lemma}

\begin{proof}
    Take any two points $x,x' \in E$. By assumption there exists $y \in F$ such that $y(x) \neq y(x')$, so that in particular $|y(x) - y(x')| = | y (x - x') | \geq C$ for some $C >0$. As $G$ is dense in $F$, we can find $y' \in G$ so that $\| y -y'\|^* \leq \frac{C}{2 \|x-x'\|}$. Then $|y'(x-x')|  \geq | y (x-x')| - |(y-y') (x- x')| \geq \frac{1}{2}C $, so that $y'$ separates $x$ and $x'$, which shows $G$ is point separating.
\end{proof}

The following lemma not only ensures that $(E^*)^{\otimes_a n}$ separates points in the injective tensor space $E^{\otimes_{\epsilon} n}$ (see Definition~\ref{def: injective tensor norm}), but also provides the corner stone for a more general statement -- see Corollary~\ref{cor: point separating small} below.
    
\begin{lemma}\label{lem: point separating of injective tensor space}
    Let $E$ be a real Banach space and let $n \in \N_{>0}$. Then for all $u\in E^{\otimes_{\epsilon} n}\setminus \{ 0\}$ there exist $\phi_i \in E^*$ with $\|\phi\|^*=1$, for $1\leq i \leq n$, such that
    $
    | (\phi_1\otimes\ldots \otimes\phi_n) (u) | > 0.
    $
\end{lemma}
\begin{proof}
Let $u\in E^{\otimes_{\epsilon} n}\setminus \{ 0\}$ and let $u'\in E^{\otimes_a n}$ be such that $\|u - u'\|_\epsilon < \frac{1}{4}\|u\|_\epsilon$. By the triangle inequality, we have $\|u'\|_\epsilon > \frac{3}{4} \|u\|_\epsilon$, and so by definition of the injective tensor norm there exist  $\phi_i \in E^*$, $1\leq i \leq n$, such that $\|\phi_i\|^* = 1$ for all $1\leq i\leq n$ and that
    \[
    | (\phi_1\otimes\ldots \otimes\phi_n) (u')| > \tfrac{1}{2}\|u\|_\epsilon.
    \]
    As $\|\phi_i \|^* = 1$ for all $i$, we have that $\|\phi_1\otimes\ldots \otimes\phi_n \|_\epsilon^* = 1$ and so $|(\phi_1\otimes\ldots \otimes\phi_n) (u - u')| \leq \|u-u'\|_\epsilon$, so that
    \[
    | (\phi_1\otimes\ldots \otimes\phi_n)  (u) | > \tfrac{1}{4} \|u\|_\epsilon > 0.
    \]
This completes the proof.
\end{proof}

The injective tensor norm is special, as it is the smallest possible reasonable crossnorm. Therefore, whenever some sequence is Cauchy in some other reasonable crossnorm, it is also Cauchy in the injective tensor norm. It is therefore tempting to conclude that the injective tensor space is a larger space than any other tensor space of a reasonable crossnorm. If that were the case, one could essentially directly generalize the previous lemma. However, it is not directly guaranteed that two Cauchy sequences that converge to different objects with respect to a reasonable crossnorm, also converge to different objects in the injective tensor norm. However, we can show that this is guaranteed if the underlying Banach space has the approximation property (see Definition~\ref{def: approximation property}) and the crossnorm is strongly uniform (see Definition~\ref{def:strong-unif-crossnorm}). 

In the proposition below, $E^{\otimes_{\epsilon} n}$ denotes the topological tensor space $\prescript{n}{j=1}{\bigotimes_{\epsilon}}E$ obtained by taking the injective tensor norm.
\begin{proposition}\label{prop: natural embedding injective tensor space is injective}
Let $E$ be a real Banach space with the approximation property, let \mbox{$\{\|\cdot \|_n\}_{n\in \N_{>0}}$} be a strongly uniform crossnorm family on $\{E^{\otimes n}\}_{n\in \N_{>0}}$, and let $J_n\colon E^{\otimes_{a} n} \rightarrow E^{\otimes_\epsilon n}$, $n\in \N_{>0}$, be the canonical embedding. Then $J_n$ extends to a bounded injective map $\tilde{J}_n\colon  E^{\otimes n} \rightarrow E^{\otimes_\epsilon n}$ for all $n\in \N_{>0}$.
\end{proposition}

\begin{proof}
As the injective tensor norm is the smallest reasonable crossnorm (see Equation \eqref{eq: injective tensor norm predef}, Definition~\ref{def: injective tensor norm}, and Proposition~\ref{prop: properties of projective and injective} combined with the fact that a uniform crossnorm is a reasonable crossnorm), the maps $J_n$, $n\in \N_{>0}$, are contractions and can therefore be extended continuously to maps $\tilde J_n \colon E^{\otimes n} \to E^{\otimes_\epsilon n}$, $n\in \N_{>0}$.  
It remains to prove that the maps $\tilde{J}_n$, $n\in \N_{>0}$, are injective. For $J_2$ this follows e.g.\ from \cite[Proposition 1 in Section 21.7]{defant_tensor_1993}. With some additional effort this proposition also provides injectivity for $J_n$, $n\in \N_{>0}$.

We prove the injectivity of $\tilde J_n$ by induction, recall that we have already established that $\tilde J_2$ is injective. Assume $\tilde J_{n-1}$ is injective for some $n\geq 3$. Define the map $J'_n \colon E\otimes_a E^{\otimes n-1 } \to E\otimes_{\epsilon} E^{\otimes n-1 } $, and $\tilde J'_n$ as its continuous extension $\tilde J_n' \colon E\otimes_{\|\cdot \|_{n}} E^{\otimes n-1} \to E\otimes_\epsilon E^{\otimes n-1 }$. As $E$ has the approximation property, by \cite[Proposition 1 in Section 21.7]{defant_tensor_1993}, the extension $\tilde J_n'$ is injective (note that the definition of tensor norms in~\cite{defant_tensor_1993} corresponds to our definition of uniform crossnorms). We can decompose $\tilde J_n$ as follows (by the fact that the left-hand side below coincides with the right-hand side on  $E^{\otimes_a n}$)
    \[
     \tilde J_n = (Id \otimes \tilde J_{n-1}) \circ\tilde J_n'.
    \]
By \cite[Proposition 2 in Section 4.3]{defant_tensor_1993} the map $(Id \otimes \tilde J_{n-1})  \colon E \otimes_\epsilon E^{\otimes n -1} \to E \otimes_\epsilon E^{\otimes_\epsilon n -1}$ is injective. Thus, $\tilde J_n$ is the composition of two injective maps and therefore injective itself. Finally, we recall that the strong crossnorm property implies that $E^{\otimes n} = E\otimes_{\| \cdot \|_{n}} E^{\otimes n-1 }$, see~\eqref{eq: strong crossnorm property for the spaces}. As the injective tensor norm is also a strong crossnorm (see Proposition~\ref{prop: properties of projective and injective}), this completes the proof.
\end{proof}

\begin{corollary}\label{cor: point separating small}
Let $E$ be a real Banach space with the approximation property, let $\{\|\cdot \|_n\}_{n\in \N_{>0}}$ be a strongly uniform crossnorm family on $\{E^{\otimes n}\}_{n\in \N_{>0}}$. Then for $n\in \N_{>0}$ and $u\in E^{\otimes n}\setminus \{0\}$, there there exists $\phi_i \in E^*$ with $\|\phi\|^*=1$, for $1\leq i \leq n$, such that
    \[
    | (\phi_1\otimes\ldots \otimes\phi_n)(u) | > 0.
    \]    
\end{corollary}
\begin{proof}
Let $n\in \N_{>0}$, $u \in E^{\otimes n} \setminus \{ 0 \}$, and let $u'\in E^{\otimes_a n}$ be such that $\|u-u'\|_n < \frac{1}{4}\|u\|_n$. As $\tilde J_n$ is injective by Proposition~\ref{prop: natural embedding injective tensor space is injective}, we can apply Lemma \ref{lem: point separating of injective tensor space} to find $\phi_i \in E^*$, $1\leq i \leq n$, such that
    \[
    | (\phi_1\otimes\ldots \otimes\phi_n )(\tilde J_n u) | > C,
    \]
    for some $C>0$. As $\tilde J_n$ is a contraction, we have that $\|\tilde J_n u - \tilde J_n u' \|_n < \frac{1}{4} $ and so
    \[
     | (\phi_1\otimes\ldots \otimes\phi_n) ( u') | =| \phi_1\otimes\ldots \otimes\phi_n (\tilde J_n u') | >\tfrac{3}{4} C.
    \]
    As $\|\cdot\|_n$ is a reasonable crossnorm, we have that $\|\phi_1\otimes\ldots \otimes\phi_n \|_n^*  = 1$ and so by the above and by the fact that $\|u-u'\|_n < \frac{1}{4}\|u\|_n$, we have
    \[
     | (\phi_1\otimes\ldots \otimes\phi_n) ( u) | >\tfrac{1}{2} C.
    \]
\end{proof}
Corollary~\ref{cor: point separating small} implies that the algebraic tensor algebra $T_a(E^*)$ separates points in the tensor algebra $T((E))$ (see Definitions~\ref{def: tensor algebra} and~\ref{def: algebraic tensor algebra}): 
\begin{proposition}\label{prop: the dual tensor algebra can actually separate points}
Let $E$ be a real Banach space with the approximation property, let \mbox{$\{\|\cdot \|_n\}_{n\in \N_{>0}}$} be a strongly uniform crossnorm family on $\{E^{\otimes n}\}_{n\in \N_{>0}}$, and let $T((E))$ be the associated tensor algebra. Then, $T_a(E^*)$ separates points in $T((E))$.
\end{proposition}
\begin{proof}
    To show that $T_a(E^*)$ separates points $T((E))$ it is sufficient to show that $(E^*)^{\otimes_a i}$ separates points in $E^{\otimes i}$ for any $i\in \N_{\geq 2}$ (the cases $i=0$ and $i=1$ are apparent). This however follows directly from Corollary \ref{cor: point separating small}
\end{proof}

\section{Proof of Proposition~\ref{prop: equality of group like definitions}}\label{sec: group-like elements}

The goal of this section is to provide a proof of Proposition~\ref{prop: equality of group like definitions}, which states that  the two notions of group-like multiplicative functionals provided in Section~\ref{sec: weakly geometric} coincide in Banach spaces that have the approximation property and whose tensor norms satisfy the strongly uniform crossnorm property. 

Our proof is based on the following result, which is a consequence of~\cite[Theorem 3.5]{geng_introduction_nodate} (recall that $\text{Sh}(m,l)$ denotes the set of all $(m,l)$-shuffles, see Definition~\ref{def: shuffle}):

\begin{theorem}\label{thm: group like iff xi geng group like}
    Let $T((E))$ be the tensor algebra for some Banach space $E$. Then, a multiplicative functional $\mathbf{x} \in T((E))$ is group-like if and only if for all $m,l\in \N$
    \begin{equation}\label{eq: group-like property of xi geng}
        \mathbf{x}^{(m)} \otimes \mathbf{x}^{(l)} = \sum_{\sigma \in \text{Sh}(m,l)} P_\sigma\left( \mathbf{x}^{(m+l)}\right).
    \end{equation}
\end{theorem}
\begin{proof}
Theorem 3.5 in \cite{geng_introduction_nodate} states that $ \log \mathbf{x}$ is a Lie series, if and only if $\mathbf{x}$ satisfies Equation \eqref{eq: group-like property of xi geng}. Note that in~\cite{geng_introduction_nodate}, group-like elements are \emph{defined} to be elements satisfying Equation \eqref{eq: group-like property of xi geng}. Also note that~\cite[Theorem 3.5]{geng_introduction_nodate} contains implicit assumptions on the crossnorms involved, which can be found on~\cite[p.\ 10]{geng_introduction_nodate}; in particular, it is assumed (and crucial to the proof of~\cite[Theorem 3.5]{geng_introduction_nodate}) that the crossnorms are symmetric in the sense of Definition \ref{def: symmetric norms}. 
\end{proof}
With Theorem \ref{thm: group like iff xi geng group like}, we can show the equivalence of the two definitions of group-like elements.
\begin{proof}[Proof of Proposition \ref{prop: equality of group like definitions}]
    To start, take any $\mathbf{x}\in T((E))$ and fix $m,l\in \N$. For any $\mathbf{y} \in (E^*)^{\otimes_a m}, \mathbf{y}' \in (E^*)^{\otimes_a l}$, we note that we have
    \begin{equation}\label{eq: aux algebraic grouplike equivalence 1}
    \begin{split}
        \inprod{\mathbf{x}}{\mathbf{y} \shuffle \mathbf{y}'}  = \inprod{\mathbf{x}^{(m+l)}}{\sum_{\sigma\in \text{Sh}(m,l)} P_{\sigma^{-1}} \left( \mathbf{y} \otimes \mathbf{y}'\right)} \\
        = \inprod{\sum_{\sigma\in \text{Sh}(m,l)} P_{\sigma} \left( \mathbf{x}^{(m+l)}\right)}{\mathbf{y} \otimes \mathbf{y}'},
    \end{split}
    \end{equation}
    and that
    \begin{equation}\label{eq: aux algebraic grouplike equivalence 2}
     \inprod{\mathbf{x}}{\mathbf{y}}\inprod{\mathbf{x}}{\mathbf{y}'}= \inprod{\mathbf{x}^{(m)} \otimes \mathbf{x}^{(l)}}{\mathbf{y}\otimes \mathbf{y}'} .
    \end{equation}
    If we assume that $\mathbf{x}$ is group-like, we can thus using Theorem \ref{thm: group like iff xi geng group like}, conclude that
    \[
     \inprod{\mathbf{x}}{\mathbf{y}} \inprod{\mathbf{x}}{\mathbf{y}'} = \inprod{\mathbf{x}}{\mathbf{y}\shuffle \mathbf{y}'}
    \]
    holds for all $\mathbf{y}\in (E^*)^{\otimes_a m}$, $\mathbf{y}'\in (E^*)^{\otimes_a l}$, and therefore that $\mathbf{x}$ is weakly group-like.

    If on the other hand we assume that $\mathbf{x}$ is weakly group-like we have using Equations \eqref{eq: aux algebraic grouplike equivalence 1} and \eqref{eq: aux algebraic grouplike equivalence 2} that for all $\mathbf{y}\in (E^*)^{\otimes_a m}$ and $\mathbf{y}'\in (E^*)^{\otimes_a l}$ we have
    \[
    \inprod{\sum_{\sigma\in \text{Sh}(m,l)} P_{\sigma} \left( \mathbf{x}^{(m+l)}\right)}{\mathbf{y} \otimes \mathbf{y}'} = \inprod{\mathbf{x}^{(m)} \otimes \mathbf{x}^{(l)}}{\mathbf{y}\otimes \mathbf{y}'},
    \]
    or equivalently
    \begin{equation}\label{eq: almost grouplike equivalence formulation}
    \inprod{\mathbf{x}^{(m)} \otimes \mathbf{x}^{(l)} -\sum_{\sigma\in \text{Sh}(m,l)} P_{\sigma} \left( \mathbf{x}^{(m+l)}\right)}{\mathbf{y} \otimes \mathbf{y}'} = 0.
    \end{equation}
    Therefore, if we assume that $E$ satisfies the approximation property and that we have a strongly uniform crossnorm family, we can apply Corollary \ref{cor: point separating small} to conclude that
    \[
    \mathbf{x}^{(m)} \otimes \mathbf{x}^{(l)} =\sum_{\sigma\in \text{Sh}(m,l)} P_{\sigma} \left( \mathbf{x}^{(m+l)}\right)
    \]
    This implies that Equation \eqref{eq: group-like property of xi geng} holds, and so that $\mathbf{x}$ is group-like by Theorem \ref{thm: group like iff xi geng group like}.

    To show the analogous statements involving $G^{(n)}(E)$ and $ G_w^{(n)}(E)$, note that if we define $\pi_n \colon T((E)) \to T^{(n)}(E)$ to be the canonical projection map, we have
    \[
    \pi_n (\text{Lie}(E)) = \text{Lie}_n(E), 
    \]
    and therefore, by the fact for any $\mathbf{x} \in T((E))$, $(\exp{\mathbf{x}})^{(i)}$ only depends on the terms of $\mathbf{x}$ of order less or equal to $i$, we also have 
    \[
    \pi_n (G(E)) = G^{(n)}(E).
    \]
    Additionally, it also holds that
    \[
    \pi_n (G_w(E)) = G^{(n)}_w(E).
    \]
    As a consequence, the statements regarding $G^{(n)}(E)$ and $G_w^{(n)}(E)$ follow from what we have proven for $G(E)$ and $G_w(E)$.
\end{proof}

\section{Tensor norms of the time extended Banach space}\label{app: tensor spaces}
\subsection{Uniformity of crossnorms is not automatically preserved}
\label{sec: time extension1 of Banach spaces}
To define the time extended rough paths of Section \ref{ssec: time-extended rough paths}, we assume (see Propositions~\ref{prop: time extension} and~\ref{prop: when there is a time component then unique signature}) that there exist admissible families of tensor norms on $E$ and $\R \oplus E$ that are compatible in the sense of Assumption~\ref{ass: time-extended}. One may wonder whether it is possible to start with a family of admissible tensor norms on $E$, and construct a compatible family of admissible on $\R \oplus E$ from there. However, there does not seem to be a canonical way to do so while preserving the strong uniform crossnorm property (which is not needed for admissibility, but \emph{is} needed in e.g.\ in Theorem~\ref{thm: norm-compact uat}). The goal of this section is to illustrate how problems may arise. 

Let $(e_i)_{i=1}^{n}$ denote the canonical orthonormal basis of $\R^n$, $n\in \N$; for $1\leq i,j\leq n$ we define $e_{ij}=e_i \otimes e_j \in \R^n \otimes \R^n$. We identify $\R^n\otimes \R^n$ with $\R^{n\times n}$ by the mapping 
\begin{equation*}
    \R^n \otimes \R^n \ni A \rightarrow (\langle A, e_{ij} \rangle)_{1\leq i,j \leq n} =\colon (A_{ij})_{1\leq i,j\leq n}.
\end{equation*}
We take $E = \R^2$ endowed with the norm $\| (x,y) \|=\max(|x|,|y|)$, $(x,y)\in E$. In addition, we endow $E\otimes E\simeq \R^{2\times 2}$ with the norm 
\begin{equation*}
\| A \|_{2} = \max_{1\leq i,j\leq 2}(|A_{ij}|).
\end{equation*}
One readily verifies that $\| \cdot \|_2$ is the injective tensor norm on $E\otimes E$, in particular, $(\| \cdot \|, \| \cdot \|_2 )$ are strongly uniform symmetric crossnorms by Proposition~\ref{prop: properties of projective and injective} (thus in particular admissible).\par 
Next, we endow $\R\oplus E\simeq \R^{3}$ with the norm
\begin{equation*}
\varnorm{ (t,x,y) }= |t| + \max(|x|,|y|),\quad (t,x,y)\in \R\oplus E,
\end{equation*}
and we endow $(\R\oplus E)\otimes (\R\oplus E)\simeq \R^{3\times 3}$ with the norm 
\begin{equation*}
\varnorm{ A }_2 = |A_{11}| + \max(|A_{12}|,|A_{13}|) + \max(|A_{21}|,|A_{31}|) + \max(|A_{22}|,|A_{23}|,|A_{32}|,|A_{33}|),
\end{equation*}
$A \in (\R\oplus E)\otimes (\R\oplus E)$.
Note that the canonical embedding 
$J \colon (\R \oplus E)\otimes (\R \oplus E) \rightarrow \R \oplus E \oplus E \oplus (E\otimes E)$ is a isomorphism, i.e., Assumption~\ref{ass: time-extended} is satisfied and thus Proposition~\ref{prop: time extension} is applicable. However, $(\varnorm{\,\cdot\,}, \varnorm{\,\cdot\,}_2) $ is \emph{not} strongly uniform, as we will now demonstrate.

Indeed, let $\phi \in \mathcal{L}(\R \oplus E)$ be given by  $\phi ( t,x,y) = (x+y , x -y , x- y)$, $(t,x,y)\in \R\oplus E$. Note that 
\[
\varnorm{\phi(t,x,y)} = |x+y| + \max(|x-y|,|x-y|) \leq 2 \max(|x|,|y|),    \quad (t,x,y)\in \R\oplus E.
\]
As $\phi((0,1,1))=2$ we conclude that $\|\phi\|_{\mathcal{L}(\R\oplus E)} = 2$. 

However, we claim that $\|\phi \otimes \phi \|_{\mathcal{L}((\R\oplus E)\otimes (\R\oplus E))} \geq 8 > \|\phi \|_{\mathcal{L}(\R\oplus E)}^2$, implying that $(\varnorm{\,\cdot\,},\varnorm{\,\cdot\,}_2)$ are not strongly uniform crossnorms. To see that $\|\phi \otimes \phi \|_{\mathcal{L}((\R\oplus E)\otimes (\R\oplus E))} \geq 8$, 
observe that \[
A \coloneqq e_{22}+ e_{23} + e_{32} - e_{33} \in (\R \oplus E)\otimes (\R \oplus E) \] satisfies $\varnorm{A}_2 = 1$. On the other hand, we have that 
\begin{align*}
(\phi\otimes \phi) e_{22} &= \phi(e_2)\otimes \phi(e_2) = (1,1,1)^T (1,1,1) \in \R^{3\times 3},
\\
(\phi\otimes \phi) e_{23} &= \phi(e_2)\otimes \phi(e_3) = (1,1,1)^T (1,-1,-1) \in \R^{3\times 3},
\\
(\phi\otimes \phi) e_{32} &= \phi(e_3)\otimes \phi(e_2) = (1,-1,-1)^T (1,1,1) \in \R^{3\times 3},
\\
(\phi\otimes \phi) e_{33} &= \phi(e_3)\otimes \phi(e_3) = (1,-1,-1)^T (1,-1,-1) \in \R^{3\times 3} .
\end{align*}
Thus, 
\[
(\phi\otimes \phi) (A) = \begin{pmatrix}
2 & 2 & 2\\
2 & -2 & -2 \\
2 & -2 & -2
\end{pmatrix} 
\]
so that $\varnorm{(\phi \otimes \phi) (A)}_2 = 8$ and thus $\|\phi \otimes \phi \|_{\mathcal{L}((\R\oplus E)\otimes (\R\oplus E))} \geq 8$ (with some additional effort one can show that $\|\phi\otimes \phi\|_{\mathcal{L}((\R\oplus E)\otimes (\R\oplus E))} =8$).

\subsection{Proof of Lemma \ref{lem : standard tensor norms give compatible norms}} \label{sec: proof of standard tensor norms give compatible norms}
We start with showing it holds for the Hilbert tensor norm. If $\{e_1, e_2,\ldots\}$ is an ONB for $E$, we can define an ONB $\{e_0, e_1, \ldots\}$ for $\R \oplus E$ by setting $e_0 \coloneqq (t,0)$. It thus follows that, under the identification $j\colon (\R \oplus E) \otimes_a (\R \oplus E) \rightarrow \R \oplus E \oplus E \oplus (E \otimes_a E)$, the Hilbert tensor norm is given by (for example by using Proposition \ref{prop: hilbert tensor onb})
\[
\varnorm{(t, x,y ,\mathbf{z})}_2 = \sqrt{|t|^2 + \| x \|^2 + \| y \|^2 + \| \mathbf{z} \|_2^2},
\]
which in particular shows that the extension of $j$ to a map $j\colon (\R \oplus E) \otimes (\R \oplus E) \rightarrow \R \oplus E \oplus E \oplus (E \otimes E)$ exists and is a homeomorphism.

For the projective tensor norm, we first note that the norm defined in Equation \eqref{eq: R plus E norm} is equivalent to the norm
\begin{equation}\label{eq: R plus E norm proj}
\varnorm{(t,x)} \coloneqq |t| + \| x\|,
\end{equation}
and as a consequence so are the projective tensor norms with respect to these two norms. Therefore we will prove the statement for the norm taken as in Equation \eqref{eq: R plus E norm proj} as in this case the projective tensor norm has a convenient explicit form: under the identification of $j\colon (\R \oplus E) \otimes_a (\R \oplus E) \rightarrow \R \oplus E \oplus E \oplus (E \otimes_a E)$ the projective norm on $(\R \oplus E)\otimes_a (\R \oplus E)$ is given by
\begin{equation}\label{eq: proj tensor norm identification}
\varnorm{(t, x,y ,\mathbf{z})}_\pi = |t| + \| x \| + \| y \| + \| \mathbf{z} \|_\pi,
\end{equation}
from which we can conclude the statement. To show that this is indeed true, we can directly adapt the proof of \cite[Example 4.53]{hackbusch_tensor_2019}, which relies on the fact that the projective crossnorm is the largest crossnorm (see e.g. \cite[Proposition 4.52]{hackbusch_tensor_2019}. Indeed, one can verify that Equation \eqref{eq: proj tensor norm identification} defines a crossnorm. Furthermore, for any other crossnorm $\varnorm{.}_1$ one finds that
\begin{equation}\label{eq: component wise inequality proj}
\varnorm{(0, 0,0 ,\mathbf{x} \otimes\mathbf{y})}_1  = \varnorm{ (0,\mathbf{x}) } \varnorm{ (0,\mathbf{y})}  = \| x\| \| y\| .
\end{equation}
In particular, one can use this norm to induce a crossnorm on $E \otimes_a E$ by
\[
\| \mathbf{z} \|_1 \coloneqq \varnorm{(0, 0,0 ,\mathbf{z})}_1.
\]
However, the projective crossnorm property is the largest crossnorm, and thus
\[
 \varnorm{(0, 0,0 ,\mathbf{z})}_1 = \| \mathbf{z} \|_1 \leq \| \mathbf{z} \|_\pi.
\]
Additionally, one gets similar equalities as in Equations \eqref{eq: component wise inequality proj}, for the other components, so that by the triangle inequality
\[
\varnorm{(t, x,y ,\mathbf{z})}_1 \leq |t| + \| x \| + \| y \| + \| \mathbf{z} \|_\pi.
\]
As a consequence, the crossnorm defined in Equation \eqref{eq: proj tensor norm identification} is the largest crossnorm, and as such must be equal to the projective tensor norm.  

For the injective tensor norm, the idea is the same as for the projective tensor norm, except that the injective tensor norm is the \emph{smallest} reasonable crossnorm (see \cite[Proposition 4.74]{hackbusch_tensor_2019}). Take as the equivalent norm on $(\R \oplus E)$ 
\[
\varnorm{(t,x)} \coloneqq \max (|t|,\| x\|).
\]
The injective tensor norm is then given by
\begin{equation}\label{eq: inj tensor norm identification}
\varnorm{(t, x,y ,\mathbf{z})}_\epsilon = \max(|t|, \| x \| , \| y \| , \| \mathbf{z} \|_\epsilon).
\end{equation}
To see that the norm given on the right-hand side is indeed a reasonable crossnorm, note that elements of the dual of $\R \oplus E$ can be written as $(\phi_1, \phi_2)$, with $\phi_1 \in \R^*$, and $\phi_2 \in E^*$ with $\varnorm {\phi_1,\phi_2}^* = |\phi_1| + \|\phi_2\|^*$. One can verify, analogously to the case of the projective tensor norm,  that for any other crossnorm $\varnorm{.}_1$ on $(\R \oplus E) \otimes_a (\R \oplus E)$, one has
\[
\varnorm{(t, x,y ,\mathbf{z})}_\epsilon = \max(|t|, \| x \| , \| y \| , \| \mathbf{z} \|_\epsilon)\leq \varnorm{(t, x,y ,\mathbf{z})}_1,
\]
which concludes the proof.

\section*{Acknowledgments}
The authors thank Marten Wortel for help on tensor spaces, in particular, for pointing out Proposition~\ref{prop: natural embedding injective tensor space is injective}. 

\bibliographystyle{alphaurl}
\bibliography{eerstebibtex}
\end{document}